\documentclass{amsart}

\usepackage{amssymb}
\usepackage{amsmath}
\usepackage{graphicx}
\usepackage{amsthm}
\usepackage{enumerate}

\newtheorem{theorem}{Theorem}[section]
\newtheorem{lemma}[theorem]{Lemma}
\newtheorem{proposition}[theorem]{Proposition}
\newtheorem{corollary}[theorem]{Corollary}

\newtheorem{_definition}[theorem]{Definition}
\newenvironment{definition}{\begin{_definition}\rm}{\end{_definition}}

\newtheorem{_remark}[theorem]{\it Remark}
\newenvironment{remark}{\begin{_remark}\rm}{\end{_remark}}

\newtheorem{_example}[theorem]{Example}
\newenvironment{example}{\begin{_example}\rm}{\end{_example}}

\newtheorem{_assumption}[theorem]{Assumption}
\newenvironment{assumption}{\begin{_assumption}\rm}{\end{_assumption}}

\numberwithin{equation}{section}
\numberwithin{table}{section}
\numberwithin{figure}{section}

\newcommand{\A}{\mathord{\mathbb A}}
\newcommand{\C}{\mathord{\mathbb C}}
\newcommand{\F}{\mathord{\mathbb F}}
\renewcommand{\P}{\mathord{\mathbb  P}}
\newcommand{\Q}{\mathord{\mathbb  Q}}
\newcommand{\R}{\mathord{\mathbb R}}
\newcommand{\Z}{\mathord{\mathbb Z}}

\newcommand{\CCC}{\mathord{\mathcal C}}
\newcommand{\EEE}{\mathord{\mathcal E}}
\newcommand{\FFF}{\mathord{\mathcal F}}
\newcommand{\GGG}{\mathord{\mathcal G}}
\newcommand{\MMM}{\mathord{\mathcal M}}
\newcommand{\QQQ}{\mathord{\mathcal Q}}
\newcommand{\TTT}{\mathord{\mathcal T}}
\newcommand{\XXX}{\mathord{\mathcal X}}

\newcommand{\SSSS}{\mathord{\mathfrak S}}
\newcommand{\CCCC}{\mathord{\mathfrak C}}

\newcommand{\maprightsb}[1]{\; \smash{\mathop{\; \longrightarrow \; }\limits\sb{#1}}\; }
\newcommand{\maprightspsb}[2]{\; \smash{\mathop{\; \longrightarrow \; }\limits\sp{#1}\limits\sb{#2}}\; }

\newcommand{\mapleftsb}[1]{\; \smash{\mathop{\; \longleftarrow \; }\limits\sb{#1}}\; }

\newcommand{\mapdown}{\phantom{\Big\downarrow}\hskip -8pt \downarrow}
\newcommand{\mapdownright}[1]{\mapdown\rlap{$\vcenter{\hbox{$\scriptstyle#1$}}$}}
\newcommand{\mapdownleft}[1]{\llap{$\vcenter{\hbox{$\scriptstyle#1$\;}}$}\mapdown}

\newcommand{\inj}{\hookrightarrow}
\newcommand{\surj}{\mathbin{\to \hskip -7pt \to}}
\newcommand{\isom}{\xrightarrow{\sim}}

\newcommand{\set}[2]{\{\; {#1} \; \mid \; {#2} \;  \}}
\newcommand{\shortset}[2]{\{ {#1} \,|\, {#2}   \}}
\newcommand{\sethd}[3]{\left\{\;\;  {#1}\;\; \left|\;\;  \vcenter{\hbox{\parbox{#2}{#3}}}\;\;  \right. \right\}}

\newcommand{\gen}[1]{\langle {#1}  \rangle}

\newcommand{\tensor}{\otimes}

\newcommand{\sprime}{\sp\prime}
\newcommand{\spar}[1]{\sp{(#1)}}
\newcommand{\spprime}{\sp{\prime\prime}}
\newcommand{\sptimes}{\sp{\times}}
\newcommand{\sperp}{\sp{\perp}}

\newcommand{\dual}{\sp{\vee}}

\newcommand{\semidirectproduct}{\rtimes}

\newcommand{\inv}{\sp{-1}}

\newcommand{\Hom}{\mathord{\mathrm{Hom}}}

\newcommand{\OG}{\mathord{\mathrm{O}}}

\newcommand{\id}{\mathord{\mathrm{id}}}

\newcommand{\Ker}{\operatorname{\mathrm{Ker}}\nolimits}
\newcommand{\Coker}{\operatorname{\mathrm{Coker}}\nolimits}

\newcommand{\Aut}{\operatorname{\mathrm{Aut}}\nolimits}
\newcommand{\Gal}{\operatorname{\mathrm{Gal}}\nolimits}

\newcommand{\pr}{\mathord{\mathrm{pr}}}
\newcommand{\rank}{\operatorname{\mathrm{rank}}\nolimits}
\newcommand{\disc}{\operatorname{\mathrm{disc}}\nolimits}

\newcommand{\rmdisc}{\mathord{\mathrm{disc}}}

\newcommand{\mystruth}[1]{\phantom{\hbox{\vrule height #1}}}
\newcommand{\mystrutd}[1]{\phantom{\hbox{\vrule depth #1}}}
\newcommand{\mystruthd}[2]{\phantom{\hbox{\vrule  height #1 depth #2}}}

\newcommand{\pione}{\pi_1}

\newcommand{\spin}{\mathop{\mathrm {spin}}}
\newcommand{\sign}{\mathord{\mathrm {sign}}}
\newcommand{\Det}{\mathord{\mathrm {Det}}}
\newcommand{\Sign}{\mathord{\mathrm {Sign}}}
\newcommand{\Zp}{\Z_{p}}
\newcommand{\Qp}{\Q_{p}}
\newcommand{\Zpp}{\Z_{(p)}}
\newcommand{\intf}[1]{\langle #1 \rangle}
\newcommand{\leng}{\mathord{\mathrm {leng}}}

\newcommand{\AG}{\bar{G}}
\newcommand{\agg}{g}
\newcommand{\simAG}{\sim_{\AG}}
\newcommand{\funcq}[1]{q_{#1}}
\newcommand{\OGA}{\OG_{\A}}
\newcommand{\OGAz}{\OG_{\A, 0}}
\newcommand{\OGsharp}{\OG^{\sharp}}
\newcommand{\SigmaSharp}{\Sigma^{\sharp}}
\newcommand{\SigmaSharpA}{\SigmaSharp_{\A}}

\newcommand{\ord}{\mathord{\mathrm {ord}}}
\newcommand{\ordp}{\ord_p}

\newcommand{\detspin}{(\mathord{\det}, \mathord{\spin})}

\newcommand{\adele}[1]{\mathord{\boldsymbol #1}}
\newcommand{\Dq}[1]{(D_{#1}, q_{#1})}
\newcommand{\OGDq}[1]{\OG(D_{#1}, q_{#1})}

\newcommand{\adelesigma}{{\adele{\sigma}}}
\newcommand{\adeletau}{{\adele{\tau}}}
\newcommand{\adelerho}{{\adele{\rho}}}
\newcommand{\adelepsi}{{\adele{\psi}}}

\newcommand{\PD}{P(d)}
\newcommand{\GammaQtilde}{ \Gamma_{\Q}^{\mathord{\sim}}}
\newcommand{\GammaA}{\Gamma_{\A}}
\newcommand{\GammaAz}{\Gamma_{\A, 0}}
\newcommand{\GammaAD}{\Gamma_{\A,d}}
\newcommand{\GammaD}{\Gamma_{d}}
\newcommand{\GammaQ}{\Gamma_{\Q}}

\newcommand{\Br}{\mathord{\rm Br}}
\newcommand{\Orb}{\mathord{\rm Orb}}
\newcommand{\Stab}{\mathord{\rm Stab}}
\newcommand{\Roots}{\mathord{\rm Roots}}

\newcommand{\minordp}{\mathord{\rm minord}_p}

\newcommand{\Dp}{D}

\newcommand{\bp}{b}
\newcommand{\Fq}{F_{q}}

\newcommand{\Matl}{{\rm M}_{\ell}}
\newcommand{\ESMat}{\Delta}
\newcommand{\TMTT}[2]{#1 \cdot #2\cdot {}^{t} #1}
\newcommand{\KerD}{N_{\Dp}}
\newcommand{\basisDp}{\varepsilon}
\newcommand{\basisLam}{e}
\newcommand{\basisLamdual}{e\dual}

\newcommand{\Lam}{\Lambda}
\newcommand{\MLam}{M}
\newcommand{\MLamdual}{M\inv}
\newcommand{\tilg}{\tilde{g}}
\newcommand{\tilT}{\tilde{T}}
\newcommand{\Matlp}[1]{{\rm M}_{\ell, #1}}

\newcommand{\MV}{M_{V}}
\newcommand{\sqnorm}[1]{\intf{#1, #1}}

\newcommand{\transpose}{\hskip 1pt {}^{t} \hskip -.2pt}
\newcommand{\apprx}[1]{\hskip 1pt {}^{\mathord{\rm a}} \hskip -1pt #1}
\newcommand{\aT}{{}^{\mathord{\rm a}} \hskip .5pt T}
\newcommand{\ar}{\raise 2pt \hbox{\scriptsize a}  \hskip .5pt r}
\newcommand{\ab}{\raise 4pt \hbox{\scriptsize a}  \hskip .5pt b}
\newcommand{\ag}{ \raise 2pt \hbox{\scriptsize a}  \hskip .5pt g}

\newcommand{\Lgen}[1]{L(#1)}
\newcommand{\barLgen}[1]{M(#1)}

\newcommand{\nnu}{\nu_{k-1}}
\newcommand{\pnnu}{p^{\nnu}}
\newcommand{\fk}{f_{k}}
\newcommand{\gk}{g_{k}}
\newcommand{\agk}{\ag_{k}}

\newcommand{\Real}{\mathop{\rm Re}}
\newcommand{\Imag}{\mathop{\rm Im}}

\newcommand{\mathordsim}{\mathord{\sim}}
\newcommand{\TTTGGG}{\TTT_{\GGG}}

\newcommand{\vfib}{v_{\mathord{\rm fib}}}
\newcommand{\vzero}{v_{\mathord{\rm zero}}}
\newcommand{\vf}{v_f}
\newcommand{\compconj}{c}

\newcommand{\spari}{^{(i)}}
\newcommand{\lift}{{}^{\mathord{\sim}}}

\newcommand{\NS}{\mathop{\rm NS}\nolimits}

\begin{document}

\title[Moduli of elliptic $K3$ surfaces]
{Connected components of the moduli of elliptic $K3$ surfaces}

\author{Ichiro Shimada}
\email{ichiro-shimada@hiroshima-u.ac.jp}
\address{Department of Mathematics, 
Graduate School of Science, 
Hiroshima University,
1-3-1 Kagamiyama, 
Higashi-Hiroshima, 
739-8526 JAPAN
}
\thanks{This work was supported by JSPS KAKENHI Grant Number 16H03926, 16K13749}

\begin{abstract}
 The combinatorial type of an elliptic $K3$ surface with a zero section 
 is the pair of the $ADE$-type of the singular fibers
 and the torsion part of the Mordell-Weil group.
 We determine the set of connected components of 
 the moduli of elliptic $K3$ surfaces with fixed combinatorial type.
 Our method relies 
 on the theory of Miranda and Morrison
 on the structure of a genus of even indefinite lattices, 
 and computer-aided calculations of $p$-adic quadratic forms.
\end{abstract}



\maketitle


\section{Introduction}\label{sec:Intro}
Elliptic $K3$ surfaces have been intensively studied 
by many authors
from various points of view, not only in 
algebraic and arithmetic geometry,
but also in theoretical physics of string theory.
In this paper,
we investigate certain moduli of complex elliptic $K3$ surfaces,
and determine the connected components of the moduli.
\par
An \emph{elliptic $K3$ surface} is a triple $(X, f, s)$,
where $X$ is a complex $K3$ surface,
$f\colon X\to \P^1$ is a fibration whose general fiber is a curve of genus $1$,
and $s\colon\P^1\to X$
is a section of $f$.
An elliptic $K3$ surface $(X, f, s)$ is sometimes denoted simply by $f$
with $X$ and $s$ being understood.
\par
Let $(X, f, s)$ be an elliptic $K3$ surface.
Then the set of sections of $f$ has a natural structure of abelian group
with zero element $s$.
This group is called the \emph{Mordell-Weil group}.
We denote by $A_f$ the torsion part of the Mordell-Weil group of $(X, f, s)$.
If an irreducible curve $C$ on $X$ is contained in a fiber of $f$
and is disjoint from the zero section $s$,
then $C$ is a smooth rational curve.
The set $\Phi_f$ of the classes of these smooth rational curves
form an $ADE$-configuration of vectors of square-norm $-2$ in $H^2(X, \Z)$.
(See Section~\ref{subsec:roots} for the definition of an $ADE$-configuration.)
The \emph{combinatorial type} of an elliptic $K3$ surface $(X, f, s)$
is the pair $(\Phi_f, A_f)$.
Let $\Phi$ be an $ADE$-configuration, and $A$ a finite abelian group.
We say that an elliptic $K3$ surface $(X, f, s)$ is \emph{of type $(\Phi, A)$}
if $\Phi\cong \Phi_f$ and $A\cong A_f$.
\par
In our previous papers~\cite{MR1820211},~\cite{MR1813537},
we made the complete list of $(\Phi, A)$ that can be realized as combinatorial type of elliptic $K3$ surfaces.
The cardinality of this list is $3693$.
In this paper, we refine this result to the following:
\begin{theorem}\label{thm:main}
The moduli of elliptic $K3$ surfaces of type $(\Phi, A)$ has more than one connected component
if and only if $(\Phi, A)$ appears in Tables~{\rm I}~and~{\rm II} in Section~\ref{sec:tables}.
\end{theorem}
See Section~\ref{subsec:defconn}
for the precise definition of the connected components of the moduli of elliptic $K3$ surfaces 
of fixed type $(\Phi, A)$.
In Tables~I~and~II, 
the $ADE$-configuration $\Phi$  is presented by the $ADE$-type of the configuration. 
The finite abelian group $\Z/a\Z$ is denoted by $[a]$, and $\Z/a\Z\times \Z/b\Z$
is denoted by $[a,b]$.
\par
Tables~{\rm I}~and~{\rm II}  are obtained by machine-computation.
The purpose of this paper is to explain the algorithm to calculate the set of connected components 
of the moduli.
%
\par
\medskip
The non-connectedness of the moduli is caused by two totally different reasons;
one is algebraic and the other is transcendental.
\par
For a $K3$ surface $X$,
we denote by $\NS(X):=H^{1,1}(X)\cap H^2(X, \Z)$ the \emph{N\'eron-Severi lattice of $X$},
and by $T(X)$ the \emph{transcendental lattice of $X$};
that is, $T(X)$ is the orthogonal complement 
of $\NS(X)$ in $H^2(X, \Z)$.
\par
Let $(X, f, s)$ be an elliptic $K3$ surface.
We denote by $U_f\subset H^2(X, \Z)$
the sublattice generated by the class
of a fiber of $f$ and the class of the section $s$, 
by $\Lgen{\Phi_f}$ the sublattice of $H^2(X, \Z)$
generated by $\Phi_f\subset H^2(X, \Z)$,
and by $\barLgen{\Phi_f}$ the primitive closure of $\Lgen{\Phi_f}$ in $H^2(X, \Z)$.
It is well-known that
$A_f$ is isomorphic to  $\barLgen{\Phi_f}/\Lgen{\Phi_f}$.
We then denote by $T_f$ the orthogonal complement of $U_f\oplus \barLgen{\Phi_f}$
in $H^2(X, \Z)$.
We obviously have $\NS(X)\supset U_f\oplus \barLgen{\Phi_f}$
and $T(X)\subset T_f$.
\begin{definition}
Let $\CCC$ be a connected component of the moduli of elliptic $K3$ surfaces of type $(\Phi, A)$.
Suppose that an elliptic $K3$ surface $(X, f, s)$
corresponds to a  point of  $\CCC$.
The \emph{N\'eron-Severi lattice of  $\CCC$} is defined to be 
the isomorphism class of the  lattice $U_f\oplus \barLgen{\Phi_f}$,
and  the \emph{transcendental lattice of $\CCC$}
is defined to be  the isomorphism class of the  lattice  $T_f$.
\end{definition}
It is obvious that the N\'eron-Severi lattice 
and the transcendental lattice of a connected component $\CCC$ do not depend on the choice of
the member $(X, f, s)$ of $\CCC$.
It will be seen that,
if $(X, f, s)$ is chosen generally in $\CCC$,
then the N\'eron-Severi lattice of $\CCC$ is isomorphic to $\NS(X)$,
and the transcendental lattice of $\CCC$ is isomorphic to $T(X)$.
(See the proof of Theorem~\ref{thm:barzeta}.)
\begin{definition}\label{def:algequivX}
We say that two elliptic $K3$ surfaces $(X, f, s)$ and $(X\sprime, f\sprime, s\sprime)$
of the same type $(\Phi, A)$
are \emph{algebraically equivalent}
if there exists an isomorphism $\Phi_f\cong\Phi_{f\sprime}$ of $ADE$-configurations
such that 
the induced isometry $\Lgen{\Phi_f}\cong\Lgen{\Phi_{f\sprime}}$ maps the even overlattice $\barLgen{\Phi_f}$ of $\Lgen{\Phi_f}$
to the even overlattice $\barLgen{\Phi_{f\sprime}}$ of $\Lgen{\Phi_{f\sprime}}$.
If there exist no such isomorphisms $\Phi_f\cong\Phi_{f\sprime}$, 
we say that $(X, f, s)$ and $(X\sprime, f\sprime, s\sprime)$
are \emph{algebraically distinguished}.
\end{definition}
If $(X, f, s)$ and $(X\sprime, f\sprime, s\sprime)$ are algebraically distinguished,
then the intersection patterns of the torsion sections and the irreducible components
of the reducible fibers for $(X, f, s)$ and for $(X\sprime, f\sprime, s\sprime)$
are different,
and hence they cannot be in the same connected component of the moduli.
\begin{definition}\label{def:algequivC}
We say that two connected components $\CCC_1$ and $\CCC_2$ are \emph{algebraically distinguished}
if an elliptic $K3$ surface belonging to $\CCC_1$ and an elliptic $K3$ surface belonging to $\CCC_2$
are algebraically distinguished.
Otherwise, we say that $\CCC_1$ and $\CCC_2$ are \emph{algebraically equivalent}.
\end{definition}
By definition, if $\CCC_1$ and $\CCC_2$ are algebraically equivalent,
then their N\'eron-Severi lattices are isomorphic,
but their transcendental lattices may be non-isomorphic.
\par
An elliptic $K3$ surface $(X, f, s)$ is called \emph{extremal}
if the rank of $\Lgen{\Phi_f}$ attains the possible maximum $18$.
Suppose that $(X, f, s)$ is extremal.
Then the transcendental lattice $T(X)$ of $X$  is an even positive definite lattice of rank $2$,
and $T(X)$ is equal to the transcendental lattice of  the connected component containing $(X, f, s)$.
\par
\medskip
{\bf Explanation of the entries of Tables~I~and~II.}
Table~I is the list of non-connected moduli of extremal elliptic $K3$ surfaces.
The horizontal line in the 4th-5th columns separates 
the connected components that are algebraically distinguished.
(This separating line appears only in nos.\;27~and\;64. 
See Section~\ref{subsec:algdist} for the detail of example no.\,64.) 
The 4th column shows the list of components $[a, b, c]$ of the transcendental lattice
$$
T=\left[\begin{array}{cc} a & b \\ b & c \end{array}\right]
$$
written in the reduced form in the sense of Gauss~(see~\cite{MR1820211}).
The 5th column displays $[r, c]$,
where $r$ (resp.~$c$) is the number of connected components 
that are (resp.~are \emph{not}) invariant under complex conjugation.
In particular, the number $c$ is always even.
\par
Table~II is the list of non-connected moduli of non-extremal elliptic $K3$ surfaces.
The 2nd column shows the rank of $\Lgen{\Phi_f}$.
The list $[c_1, \dots, c_k]$ in the 5th column indicates that 
there exist exactly $k$ algebraic equivalence classes of connected components,
and that each algebraic equivalence class has 
exactly $c_i$ connected components.
Examining Table~II and 
investigating the set of connected components
further (see Remark~\ref{rem:complexconj}), we obtain the following: 
\begin{corollary}\label{cor:twelve}
The moduli of non-extremal elliptic $K3$ surfaces of type $(\Phi, A)$
has more than one connected component
that can not be algebraically distinguished
if and only if $A$ is trivial and $\Phi$ is one of the following:
\begin{eqnarray*}
&& 
E_{7}+2A_{5},\;\;
E_{6}+A_{11},\;\;
E_{6}+A_{6}+A_{5}, \;\;
E_{6}+2A_{5}+A_{1}, \;\;
\\
&&
D_{5}+2A_{6}, \;\;
D_{4}+2A_{6}+A_{1}, \;\;
A_{11}+A_{5}+A_{1}, \;\;
A_{7}+2A_{5}, \;\;
\\
&&
2A_{6}+A_{3}+2A_{1}, \;\;
A_{6}+2A_{5}+A_{1}, \;\;
E_{6}+2A_{5}, \;\;
3A_{5}+A_{1}.
\end{eqnarray*}
For each of these types $(\Phi, A)$,
the moduli has exactly two connected components,
and they are complex conjugate to each other.
\end{corollary}
If $(X, f, s)$ is extremal, then the $K3$ surface $X$ is \emph{singular} in the sense of~\cite{MR0441982}.
It is known that 
a pair of singular $K3$ surfaces with 
isomorphic N\'eron-Severi lattices and 
non-isomorphic transcendental lattices 
has some  interesting properties.
See~\cite{MR2346573},~\cite{MR2452829} for arithmetic properties,
and~\cite{MR2541164},~\cite{MR2914799}, and~\cite{MR2604087} for  topological properties.
On the other hand, 
for non-extremal elliptic $K3$ surfaces, 
Corollary~\ref{cor:twelve} implies the following:
\begin{corollary}\label{cor:trans}
The transcendental lattice of a connected component of the moduli of 
non-extremal elliptic $K3$ surfaces of fixed type
is determined by the algebraic equivalence class of the connected component.
\end{corollary}
In fact,
we present an algorithm
to calculate the set $\CCCC(\Phi, A, G)$
of $G$-connected components of the moduli 
of \emph{marked} elliptic $K3$ surfaces 
of type $(\Phi, A)$,
where 
$G$ is a subgroup of the automorphism group $\Aut(\Phi)$
of the $ADE$-configuration $\Phi$.
(See Section~\ref{subsec:defconn}
for the definition of the set $\CCCC(\Phi, A, G)$.)
Theorem~\ref{thm:main} and Corollaries~\ref{cor:twelve},~\ref{cor:trans}
are the statements for the case where $G$ is 
the full automorphism group $\Aut(\Phi)$, 
which means that elliptic $K3$ surfaces are not marked.
See Section~\ref{subsec:monodromy}.
%
\par
\medskip
Torelli theorem for the period map of complex $K3$ surfaces~\cite{MR0284440}
enables us to study moduli of 
$K3$ surfaces by lattice-theoretic tools.
In order to investigate moduli of
lattice polarized $K3$ surfaces,
we have to determine the set of primitive embeddings 
of the polarizing lattice into the $K3$ lattice.
This task is easy when 
the $K3$ surfaces are singular,
because the transcendental lattices are positive definite of rank $2$ in this case.
When the transcendental lattices are indefinite of rank $\ge 3$,
we use Miranda-Morrison theory~\cite{MR834537, MR0839800, MMB}.
Let $L$ be an even indefinite lattice of rank $\ge 3$,
let $\GGG$ be the genus containing $L$,
and let $\OG(L)\to \OGDq{L}$ be the natural homomorphism 
from the orthogonal group of $L$ to
the automorphism group of the discriminant form $\Dq{L}$ of $L$.
Miranda and Morrison defined a certain 
finite abelian group $\MMM$ that fits in an exact sequence
$$
0\;\longrightarrow\; \Coker(\OG(L)\to \OGDq{L})\;\longrightarrow\; \MMM \;\longrightarrow\; \GGG \;\longrightarrow\; 0.
$$
Then they gave a method to calculate this exact sequence 
in terms of the spinor norms of certain isometries of the $p$-adic lattices $L\tensor \Zp$.
When we apply this theory to the study of moduli of elliptic $K3$ surfaces,
the genus $\GGG$ is the genus containing 
the transcendental lattices of
algebraic equivalence classes of connected components of the moduli.
We have to incorporate
the positive sign structures of $L$ in the theory, 
and to calculate 
the action on $\MMM$ of a subgroup of  $\Aut(\Phi)$
 explicitly.
 The flipping of positive sign structures corresponds to the complex conjugation,
 and the action of a subgroup of  $\Aut(\Phi)$
 corresponds to changing the marking.
\par
Miranda-Morrison theory was first applied to the study of moduli of $K3$ surfaces
by Akyol and Degtyarev~\cite{MR3447795}
in their study of equisingular family of irreducible plane sextics.
Recently,
G{\"u}ne{\c s} Akta{\c s}~\cite{2015arXiv150805251G} used it to the study of certain classes of quartic surfaces.
In these works,
the calculation of isometries of $p$-adic lattices
and their spinor norms was not fully-automated, and
a case-by-case method was employed at several points.
The complete list of
$ADE$-types of singularities of these normal $K3$ surfaces had been obtained by
Yang~(\cite{MR1387816}, \cite{MR1468283}).
\par
A  new ingredient of this paper is a refinement of the Miranda-Morrison group $\MMM$,
which enables us to treat the positive sign structures in a simplified way.
Another new ingredient  is an algorithm
to lift a given automorphism of 
the discriminant form $\Dq{L\tensor\Zp}$ of a $p$-adic lattice $L\tensor\Zp$
to an isometry of $L\tensor \Zp$,
and to calculate the spinor norm of this isometry.
Our method employs  approximate calculations in $p$-adic topology.
To obtain precise results, 
the estimation of approximation errors is in need.
Using  this algorithm,
we can compute the set $\CCCC(\Phi, A, G)$ 
of connected components of our moduli 
by computer.
\par
\medskip
This paper is organized as follows.
In Section~\ref{sec:lattices}, 
we collect preliminaries about lattices.
In particular, 
we recall the theory of discriminant forms due to Nikulin~\cite{MR525944}
and its application to the genus theory.
In Section~\ref{sec:conn},
we define the set $\CCCC(\Phi, A, G)$
of $G$-connected components of the moduli 
of marked elliptic $K3$ surfaces
of  fixed type $(\Phi, A)$,
where $G$ is a subgroup of $\Aut(\Phi)$.
We then introduce
a set $\QQQ(\Phi, A)/\mathordsim_G$,
which is defined in purely lattice-theoretic terms.
Using the theory of refined period map of marked $K3$ surfaces~\cite[Chapter VIII]{MR2030225},
we show that there exists a natural bijection between
$\CCCC(\Phi, A, G)$ and $\QQQ(\Phi, A)/\mathordsim_G$.
In Section~\ref{sec:MM}, 
we formulate a refinement of Miranda-Morrison theory,
and interpret the set $\QQQ(\Phi, A)/\mathordsim_G$ 
as a finite disjoint union of certain finite dimensional $\F_2$-vector spaces $\TTTGGG/\mathordsim_{\AG}$,
which are closely related to the Miranda-Morrison group $\MMM$.
Section~\ref{sec:Psip} is the technical core of our algorithm.
We present an algorithm to calculate the spinor norm
of an isometry of a $p$-adic lattice that induces a given automorphism of 
the discriminant form.
The results in Sections~\ref{sec:MM}~and~\ref{sec:Psip} 
establish an algorithm to
calculate the $\F_2$-vector spaces $\TTTGGG/\mathordsim_{\AG}$.
Using this algorithm combined with the results in Section~\ref{sec:conn},
we can compute the set $\CCCC(\Phi, A, G)$.
Applying this calculation to the case $G=\Aut(\Phi)$,
we obtain Tables~I~and~II in Section~\ref{sec:tables}.
In Section~\ref{sec:examples}, 
we investigate some examples in detail.
\par
The data computed in this paper is available from the author's web-page~\cite{compdataConnEllK3}.
For the computation, we used~{\tt GAP}~\cite{GAP}.
\par
Thanks are due to Alex Degtyarev, Klaus Hulek  and Michael L\"onne for many discussions.
The author also thanks the referee for many comments and suggestions
on the first version of this paper.
\par
\medskip
{\bf Convention.}
In this paper,
a homomorphism $f\colon M\to M\sprime$
of abelian groups is written as $v\mapsto v^f$.
In particular,
we denote 
the composite of $f\colon M\to M\sprime$ and $f\sprime\colon M\sprime\to M\spprime$ 
by $ff\sprime$ or $f\cdot f\sprime$.

\section{Lattices}\label{sec:lattices}
We fix notions and notation about lattices,
and recall some classical results.
Let $R$ be either $\Z$, $\Q$, $\Zp$, $\Qp$, or $\R$,
and let $k$ be the quotient field of $R$.
\subsection{Gram matrix}
An \emph{$R$-lattice} is a free $R$-module $L$ of finite rank
equipped with a non-degenerate symmetric bilinear form
$$
\intf{\phantom{a}, \phantom{a}}\colon L\times L \to R.
$$
Let $L$ be an $R$-lattice of rank $n$.
By choosing a basis $e_1, \dots, e_n$ of the free $R$-module $L$,
the form $\intf{\phantom{a}, \phantom{a}}$ is expressed by a symmetric 
matrix $M$ of size $n$ whose $(i, j)$-component is $\intf{e_i, e_j}$.
This matrix $M$ is called the \emph{Gram matrix} of $L$ with respect to the basis 
$e_1, \dots, e_n$.
The \emph{discriminant} $\disc(L)$ of $L$ is defined by 
$$
\disc(L):= \det(M) \bmod (R\sptimes)^2 \;\;\in \;\; (R\setminus \{0\})/(R\sptimes)^2.
$$
We denote by $\OG(L)$ the group of isometries of $L$.
By our convention,
the group $\OG(L)$ acts on $L$ from the \emph{right}.
The determinant of matrices representing isometries of $L$
gives rise to a homomorphism
$$
\det \colon \OG(L)\to \Det:=\{\pm 1\}.
$$
Note that $L\tensor k$ has a natural structure of $k$-lattice,
and $\OG(L)$ is naturally embedded in $\OG(L\tensor k)$.
\subsection{Positive sign structure}\label{subsec:positivesignstructure}
Let $L$ be an $\R$-lattice.
It is well-known that $L$ has a diagonal Gram matrix $M$ 
whose diagonal components are $\pm 1$,
and that the number $s_+$ of $+1$ (resp. $s_-$ of $-1$)  on the diagonal 
is independent of the choice of $M$.
The \emph{signature} $\sign(L)$ of $L$ is $(s_+, s_-)$.
We say that $L$ is \emph{indefinite} if $s_+>0$ and $s_->0$,
whereas $L$ is \emph{positive {\rm or} negative definite} if $s_-=0$ or $s_+=0$, respectively.
We say that $L$ is \emph{hyperbolic} if $s_+=1$.
\par 
According to~\cite{MR842425},
we define 
a \emph{positive sign structure} of $L$ to be a choice of one of the connected components
of the manifold parametrizing \emph{oriented}
$s_+$-dimensional subspaces $\Pi$ of $L$ such that
the restriction $\intf{\phantom{a}, \phantom{a}}|_{\Pi}$ of $\intf{\phantom{a}, \phantom{a}}$ to $\Pi$ is positive-definite.
Unless $L$ is negative-definite,
 $L$ has exactly two positive sign structures.
 \par
The signature and the positive sign structures of a $\Z$- or a $\Q$-lattice $L$ are defined to be those of $L\tensor\R$.
The orthogonal group $\OG(L)$ acts on the set of positive sign structures of $L$ in a natural way.
\subsection{Discriminant form}
The theory of discriminant forms was developed by Nikulin in~\cite{MR525944}.
A \emph{finite quadratic form} is a quadratic form
$$
q\colon D \to \Q/2\Z,
$$
where $D$ is a finite abelian group.
The \emph{length} $\leng(D)$ of $D$ is 
the minimal number of generators of $D$.
Let $\Dq{}$ be a  finite quadratic form.
We say that $\Dq{}$ is \emph{non-degenerate}
if the associated symmetric bilinear form 
$b\colon D\times D \to \Q/\Z$ is non-degenerate, where
$$
b(x, y):=\frac{1}{2}(q( x+ y)-q( x)-q( y)).
$$
We denote by $\OGDq{}$
the automorphism group of $\Dq{}$.
Note again that $\OGDq{}$ acts on $\Dq{}$ from the \emph{right}.
\par
Suppose that $R$ is either $\Z$ or $\Zp$,
and let $L$ be an $R$-lattice of rank $n$.
The \emph{dual lattice} $L\dual$ of $L$ is defined by 
\begin{equation}\label{eq:Ldual}
L\dual:=\set{x\in L\tensor k}{\intf{x, v}\in R\;\; \textrm{for all}\;\; v\in L},
\end{equation}
which is a free $R$-module of rank $n$
containing $L$ as a submodule of finite index.
The dual lattice $L\dual$ has
a natural $k$-valued non-degenerate symmetric bilinear form
that extends the $R$-valued form $\intf{\phantom{a}, \phantom{a}}$ on $L$.
We put
$$
D_L:=L\dual/L,
$$
and call it the \emph{discriminant group} of $L$.
We say that $L$ is \emph{unimodular} if $D_L$ is trivial.
When $R$ is $\Z$, 
the order of $D_L$ is equal to $|\rmdisc(L)|$.
\par
We say that $L$ is \emph{even} if $\intf{x,x}\in 2R$ holds for all $x\in L$.
(When $R=\Zp$ with $p$ odd,
every $R$-lattice is even.
A $\Z$-lattice $L$ is even if and only if the $\Z_2$-lattice $L\tensor\Z_{2}$ is even.)
Note that
we have a natural isomorphism
 $$
 \Q/2\Z\;\;\cong\;\; \bigoplus_{p} \Qp/2\Zp.
 $$
Hence, when $R=\Z_p$, we can regard $k/2R$ as a submodule of  $\Q/2\Z$.
Suppose that $L$ is even.
Then the \emph{discriminant form} 
$$
q_L\colon D_L\to \Q/2\Z
$$
of $L$ is a finite quadratic form defined 
by $q_L(\bar x):=\intf{x, x} \bmod 2R$,
where $\bar x\in D_L$ denotes $x \bmod L$ for $x\in L\dual$.
Since $\intf{\phantom{a}, \phantom{a}}$ is non-degenerate, the finite quadratic form $\Dq{L}$ is non-degenerate.
If $\varphi\colon L\isom L\sprime$ is an isometry of even $R$-lattices,
then $\varphi$ induces 
an isomorphism $L\dual \isom L^{\prime\vee}$
and hence an isomorphism 
$$
\funcq{\varphi} \colon \Dq{L}\isom \Dq{L\sprime}
$$
of their discriminant forms.
In particular, we have a natural homomorphism
$$
\OG(L) \to \OGDq{L}.
$$
\begin{remark}
If we adapt $L\dual :=\Hom(L, R)$ as the definition 
of the dual lattice,
it is natural to say that an isometry
$\varphi\colon L\isom L\sprime$
induces contravariantly an isomorphism $\Dq{L\sprime}\isom \Dq{L}$.
Under the present definition~\eqref{eq:Ldual}, however, 
the functor $\varphi\mapsto \funcq{\varphi}$ is covariant.
\end{remark}
\subsection{Roots}\label{subsec:roots}
Let $L$ be an even $\Z$-lattice.
A vector $r\in L$ is said to be a \emph{root} of $L$ if
$\intf{r, r}=-2$.
We put
$$
\Roots(L):=\set{r\in L}{\intf{r, r}=-2}.
$$
Let $\Phi=\{r_1, \dots, r_m\}$ be a set of roots of $L$.
Suppose that $\intf{r_i, r_j}\in \{0, 1\}$ holds for any $i\ne j$.
The \emph{dual graph} of $\Phi$ is the graph whose set of vertices
is $\Phi$ and whose set of edges is the set of pairs $\{r_i, r_j\}$
such that $\intf{r_i, r_j}=1$.
We say that $\Phi$ is an \emph{$ADE$-configuration}
if each connected component of the dual graph of $\Phi$
is a Dynkin diagram of type 
$A_{l}$ ($l\ge 1$),
$D_{m}$ ($m\ge 4$), or 
$E_{n}$ ($n=6, 7, 8$).
(See~\cite[Figure~1.7]{MR2977354} for the definition of these Dynkin diagrams).
Let $\Phi$ be an $ADE$-configuration.
The formal sum of the types $A_l, D_m, E_n$ of the connected components 
of the dual graph is called the \emph{$ADE$-type of $\Phi$}.
An isomorphism of $ADE$-configurations $\Phi$ and $\Phi\sprime$
is a bijection $\gamma\colon \Phi\isom \Phi\sprime$
such that $\intf{r^\gamma, r^{\prime\gamma}}=\intf{r, r\sprime}$
holds for all $r, r\sprime\in \Phi$.
An isomorphism class of $ADE$-configurations is determined by the $ADE$-type.
The automorphism group $\Aut(\Phi)$ of an $ADE$-configuration $\Phi$
is just the automorphism group of the dual graph of $\Phi$.
\par
A negative definite even $\Z$-lattice $L$ is said to be a \emph{root lattice}
if $L$ is generated by $\Roots(L)$.
We have the following classical result. See~\cite[Theorem~1.2]{MR2977354}, for example.
\begin{proposition}
Let $L$ be a root lattice.
Then there exists an $ADE$-configuration $\Phi\subset \Roots(L)$
that forms a basis of $L$.
\end{proposition}
The $ADE$-configuration $\Phi$ in this proposition is called a \emph{fundamental root system} of the root lattice $L$.
When an $ADE$-configuration $\Phi$ is given,
we denote by $\Lgen{\Phi}$ the root lattice with a fundamental root system $\Phi$.
%
\subsection{Even $\Zp$-lattices}\label{subsec:evenZplattices}
\begin{table}
\begin{tabular}{cll}
 \mystruthd{0pt}{6pt}
When $p$ is odd:\\
\hline 
\mystruthd{11pt}{5pt}
Name&$\Dq{}$&Brown invariant \\
\hline
$w_{p, \nu}^1$ & $\left(\Z/p^\nu\Z, \left[\dfrac{2}{p^\nu}\right]\right)$ &
$\begin{cases}
0 & \textrm{if $\nu$ is even}, \\
1-(-1)^{(p-1)/2} & \textrm{if $\nu$ is odd}, \\
\end{cases}$ \mystrutd{16pt}\mystruth{20pt}\\
\hline
$w_{p, \nu}^{-1}$ & $\left(\Z/p^\nu\Z, \left[\dfrac{2n_p}{p^\nu}\right]\right)$ &
$\begin{cases}
0 & \textrm{if $\nu$ is even}, \\
-3-(-1)^{(p-1)/2} & \textrm{if $\nu$ is odd}. \\
\end{cases}$ \mystrutd{16pt}\mystruth{20pt}\\
\hline 
\end{tabular}
\par
\medskip
\begin{tabular}{cll}
 \mystruthd{0pt}{6pt}
When $p=2$:\\
\hline 
\mystruthd{11pt}{5pt}
Name&$\Dq{}$&Brown invariant \\
\hline
\mystruthd{17pt}{12pt}
$w_{2, \nu}^{\varepsilon}$ & $\left(\Z/2^\nu\Z, \;\;\left[\dfrac{\varepsilon}{2^\nu}\right]\right)$ & $\varepsilon+\nu(\varepsilon^2-1)/2$
\\
\hline
\mystruthd{17pt}{12pt}
$u_{\nu}$ & $\left(\Z/2^\nu\Z\times \Z/2^\nu\Z, \;\; \dfrac{1}{2^\nu} \left[\begin{array}{cc} 0&1\\1&0\end{array}\right]\right)$ & 0
\\
\hline
\mystruthd{17pt}{12pt}
$v_{\nu}$ & $\left(\Z/2^\nu\Z\times \Z/2^\nu\Z, \;\; \dfrac{1}{2^\nu} \left[\begin{array}{cc} 2&1\\1&2\end{array}\right]\right)$ & $4\nu$
\\
\hline 
\end{tabular}%
\par
\bigskip
\parbox{10cm}{In this table, $\nu$ runs through $\Z_{>0}$.
When $p$ is odd, 
$n_p\in \Z$ represents a non-square residue in $\F_p\sptimes$.
When $p=2$, $\varepsilon\in \{1,3,5,7\}$ if $\nu>1$, whereas $\varepsilon\in \{1,3\}$ if $\nu=1$.}
\par
\bigskip
\caption{Indecomposable $p$-adic finite quadratic forms}\label{table:indecomppfqf}
\end{table}%
Let $p$ be a prime integer.
The isomorphism classes of even $\Zp$-lattices and their discriminant forms are 
well-understood.
See~\cite{MR525944} or~\cite[Chapter IV]{MMB} for details.
\par
We say that a finite quadratic form $q\colon D\to \Q/2\Z$
is \emph{$p$-adic} if the order of $D$ is a power of $p$. 
If $\Dq{}$ is $p$-adic, then the image of $q$ is
included in the subgroup $\Qp/2\Zp\subset \Q/2\Z$.
It is obvious that the discriminant form $\Dq{L}$ 
of an even $\Zp$-lattice $L$ is $p$-adic. 
We have the \emph{normal form theorems}
for non-degenerate $p$-adic finite quadratic forms and for even $\Zp$-lattices.
\begin{proposition}\label{prop:normalformpfqf}
A non-degenerate $p$-adic finite quadratic form is isomorphic to 
an orthogonal direct-sum of indecomposable $p$-adic finite quadratic forms
listed in Table~\ref{table:indecomppfqf}.
\end{proposition}
More precisely,
we have an algorithm to decompose a given non-degenerate $p$-adic finite quadratic form
into an orthogonal direct-sum of indecomposable ones.
See~\cite[Chapter~IV]{MMB}.
\begin{proposition}\label{prop:normalformZplat}
An even $\Zp$-lattice is isomorphic to 
an orthogonal direct-sum of indecomposable even $\Zp$-lattices
whose Gram matrices are listed below.
\par
When $p$ is odd:
$$
[ 2 p^\nu]\;\;\textrm{or}\;\;\; [ 2 p^\nu n_p],
$$
where $\nu$ runs through $\Z_{\ge 0}$,
and $n_p\in \Z$ represents a non-square residue in $\F_p\sptimes$.
\par
When $p=2$:
$$
2^{\nu}[ 2\varepsilon],
\;\textrm{or}\;\; \;
2^{\nu}\left[ \begin{array}{cc} 0 & 1 \\ 1 & 0 \end{array}\right],
\;\textrm{or}\;\;\;
2^{\nu}\left[ \begin{array}{cc} 2 & 1 \\ 1 & 2 \end{array}\right], 
$$
where $\nu$ runs through $\Z_{\ge 0}$, and $\varepsilon\in \{1,3,5,7\}$.
\par
An indecomposable even $\Zp$-lattice $L$ in the list above is unimodular
if and only if $\nu=0$ and $L\not\cong [ 2\varepsilon]$.
\end{proposition}
Moreover,
we have an algorithm to determine whether 
two given orthogonal direct-sums of these indecomposable objects
are isomorphic or not.
(See~\cite[Chapter~IV]{MMB}.)
In particular,
for a given positive integer $r$,
a given element $d\in (\Zp\setminus \{0\})/(\Zp\sptimes)^2$,
and a given non-degenerate $p$-adic finite quadratic form $\Dq{}$,
we can easily determine
whether there exists an even $\Zp$-lattice $L$
such that $\rank (L)=r$, $\disc(L)=d$, and $\Dq{L}\cong\Dq{}$,
and if it exists,
we can write a Gram matrix of such an even  $\Zp$-lattice explicitly
(see~Section~\ref{subsec:step2}).
As corollaries, we obtain the following:
\begin{proposition}\label{prop:isomclassZp}
The isomorphism class of an even $\Zp$-lattice $L$ is determined by
$\rank(L)$, $\disc(L)\in (\Zp\setminus \{0\})/(\Zp\sptimes)^2$, 
and the isomorphism class of the discriminant form $\Dq{L}$.
\end{proposition}
\begin{proposition}\label{prop:etasurjZp}
If $L$ is an even $\Zp$-lattice, 
then the natural homomorphism $\OG(L) \to \OGDq{L}$
is surjective.
\end{proposition}
\subsection{Even overlattices}\label{subsec:overlattice}
Suppose that $L$ is an even $\Z$-lattice.
An \emph{even overlattice} of $L$ is a $\Z$-submodule $M$ of $L\dual$ containing $L$
such that the natural $\Q$-valued symmetric bilinear form on $L\dual$ takes values in $\Z$ on $M$,
and
that $M$ is an even $\Z$-lattice by this $\Z$-valued form.
The following theorem is due to Nikulin~\cite{MR525944}.
\begin{proposition}\label{prop:isot}
Let $L$ be an even $\Z$-lattice,
and let $\pr_L\colon L\dual \to D_L$ denote the natural projection.
Then the mapping $K\mapsto \pr_L\inv(K)$
gives rise to a bijection from 
the set of totally isotropic subgroups $K\subset D_L$ of $\Dq{L}$ 
to the set of even overlattices of $L$.
\end{proposition}
A submodule $A$ of a free $\Z$-module $M$ is said to be \emph{primitive}
if $M/A$ is torsion free.
The \emph{primitive closure} of $A$ in $M$ is the primitive submodule $(A\tensor\Q)\cap M$ of $M$.
As a corollary of Proposition~\ref{prop:isot}, we obtain the following:
\begin{corollary}\label{cor:unimodoverlattice}
Let $S$ and $T$ be even $\Z$-lattices.
Then there exists a canonical bijective correspondence between the set of even unimodular overlattices $H$ 
of the orthogonal direct sum $S\oplus T$ such that $S$ and $T$ are primitive in $H$, 
and the set of anti-isometries $(D_S, -q_S)\isom\Dq{T}$.
\end{corollary}
The correspondence is given as follows.
Let $\gamma\colon (D_S, -q_S)\isom\Dq{T}$
be an anti-isometry of the discriminant forms.
Then the pull-back of the graph of $\gamma$ in $D_S\oplus D_T$
by the natural projection $S\dual \oplus T\dual \to D_S\oplus D_T$
is an even unimodular overlattice of $S\oplus T$.
\subsection{Genus of even $\Z$-lattices}\label{subsec:genus}
Let $\Dq{}$ be a non-degenerate finite quadratic form.
For each prime divisor $p$ of $d:=|D|$,
let $D_p$ denote the $p$-part
 $$
 \set{x\in D}{p^\nu x=0\;\;\textrm{for some integer}\;\; \nu\ge 0}
 $$
of $D$,
and let $q_p$ denote the restriction of $q$ to $D_p$.
Then $\Dq{p}$ is a non-degenerate $p$-adic
finite quadratic form.
We say that $\Dq{p}$ is the \emph{$p$-part} of $\Dq{}$.
If $p\ne p\sprime$,
then $D_p$ and $D_{p\sprime}$ are orthogonal with respect to the bilinear form $b$ of $\Dq{}$.
Hence we obtain a canonical orthogonal direct-sum decomposition
\begin{equation}\label{eq:pdecomp}
\Dq{}= \bigoplus_{p | d} \,\Dq{p}.
\end{equation}
\par
Suppose that $L$ is an even $\Z$-lattice.
Then the even $\Zp$-lattice $L\tensor \Zp$
is not unimodular if and only if $p$ divides 
the order $|D_L|$ of the discriminant group,
and the $p$-part of 
 the discriminant form $\Dq{L}$ is isomorphic to
$ \Dq{L\tensor \Zp}$.
Moreover,  the discriminant $\disc(L\tensor\Zp)$
is equal to $\disc(L) \bmod (\Zp\sptimes)^2$.
If $\sign(L)=(s_+, s_-)$, we have $\disc(L)=(-1)\sp{s_-} |D_L|$.
Hence, 
by the results we have stated so far, we obtain the following:
\begin{proposition}\label{prop:genus}
Let $L$ and $L\sprime$ be even $\Z$-lattices.
Then the following conditions are equivalent:
\begin{itemize}
\item[(i)] $\sign(L)=\sign(L\sprime)$ and $\Dq{L}\cong\Dq{L\sprime}$.
\item[(ii)] $L\tensor \R\cong L\sprime\tensor \R$, and $L\tensor \Zp\cong L\sprime\tensor \Zp$
for all $p$.
\end{itemize}
\end{proposition}
\begin{definition}
We say that even $\Z$-lattices $L$ and $L\sprime$ are \emph{in the same genus}
if the two conditions in Proposition~\ref{prop:genus} are satisfied.
\end{definition}
\begin{definition}
Let $(s_{+}, s_{-})$  be a pair of non-negative integers such that $r:=s_+ + s_->0$,
and let $\Dq{}$ be a non-degenerate finite quadratic form.
The \emph{genus determined by $(s_{+}, s_{-})$ and $\Dq{}$}
is the set of isomorphism classes of even $\Z$-lattices 
$L$ of rank $r$ such that $\sign(L)=(s_{+}, s_{-})$ and $\Dq{L}\cong \Dq{}$.
\end{definition}
We have the following criterion, due to Nikulin~\cite{MR525944}, for the genus determined by $(s_{+}, s_{-})$ and $\Dq{}$
to be non-empty. (See also Theorem 5.2 in~\cite[Chapter V]{MMB}.)
The \emph{Brown invariant $\Br\Dq{}$} of a non-degenerate finite quadratic form $\Dq{}$ is 
defined to be the element of $\Z/8\Z$ that satisfies
$$
\exp\left(\frac{2\pi \sqrt{-1}}{8}\, \Br\Dq{}\right)=\frac{1}{\sqrt{|D|}}\sum_{x\in D} \exp(\sqrt{-1}\pi \,q(x)).
$$
(See~\cite[Chapter III]{MMB} for the existence of the Brown invariant.)
The Brown invariant is additive 
under the operation of orthogonal direct-sum of non-degenerate finite quadratic forms,
and the values of this invariant 
for the indecomposable non-degenerate $p$-adic  finite quadratic forms 
are given in Table~\ref{table:indecomppfqf}.
Hence, using the decomposition~\eqref{eq:pdecomp}
and Proposition~\ref{prop:normalformpfqf},
we can easily calculate $\Br\Dq{}$ for any $\Dq{}$.
\begin{theorem}\label{thm:genusexistence}
Let $s_{+}$ and $ s_{-}$  be non-negative integers such that $s_+ + s_->0$,
and let $\Dq{}$ be a non-degenerate finite quadratic form.
We put $r:=s_+ + s_- $ and $d:=(-1)^{s_-} |D|$.
Then the genus determined by $(s_{+}, s_{-})$ and $\Dq{}$ is non-empty if and only if the following hold:
\begin{enumerate}[{\rm (i)}]
\item $r\ge \leng(D)$,
\item $\Br\Dq{}\equiv s_+ -s_- \bmod 8$, and
\item for each prime divisor $p$ of $d$,
there exists an even $\Zp$-lattice of rank $r$, discriminant $d \bmod (\Zp\sptimes)^2$,
and with the discriminant form isomorphic to the $p$-part $\Dq{p}$ of $\Dq{}$.
\end{enumerate}
\end{theorem}
\begin{remark}
%
Another formulation of the criterion by means of $p$-excess is given 
by Conway and Sloan~\cite[Chapter 15]{CSB}.
In our previous papers~\cite{MR1813537, MR2369942},
we used this $p$-excess version.
\end{remark}
By the weak Hasse principle,
we obtain the following proposition~(Theorem 1.1 in~\cite[Chapter VIII]{MMB}).
\begin{proposition}\label{prop:corweakHasse}
If even $\Z$-lattices $L$ and $L\sprime$ are in the same genus,
then the $\Q$-lattices $L\tensor\Q$ and $L\sprime\tensor\Q$ are isomorphic.
\end{proposition}
\section{Connected components of moduli}\label{sec:conn}
In this section,
we fix an $ADE$-configuration $\Phi$, a finite abelian group $A$,
and a subgroup $G$ of the automorphism group $\Aut(\Phi)$ of $\Phi$.
\subsection{Definition of connected components}\label{subsec:defconn}
Let $(X, f, s)$ be an elliptic $K3$ surface.
We consider the second cohomology group $H^2(X, \Z)$ as a $\Z$-lattice by the cup-product.
It is well-known that $H^2(X, \Z)$ is a \emph{$K3$-lattice}; 
that is, $H^2(X, \Z)$ is an even unimodular $\Z$-lattice of signature $(3, 19)$,
which is unique up to isomorphism.
Recall that $\Phi_f\subset H^2(X, \Z)$ is the set of 
classes of smooth rational curves on $X$ 
that are contracted by $f$ and are disjoint from $s$,
and that $A_f$ is the torsion part of the Mordell-Weil group. 
By the classical theory of elliptic surfaces~(see~\cite{MR0184257}),
we know that $\Phi_f$ is an $ADE$-configuration.
Recall that  $\Lgen{\Phi_f}$ denotes the root sublattice of $H^2(X, \Z)$ generated by $\Phi_f$.
\par
Suppose that $(X, f, s)$ is of type $(\Phi, A)$; 
that is, $\Phi_f\cong \Phi$ and $A_f\cong A$.
A \emph{marking} of $(X, f, s)$  is an isomorphism
$\phi\colon \Phi\isom \Phi_f$ of $ADE$-configurations.
If $\phi$ is a marking of $(X, f, s)$,
we denote by $(X, f, s, \phi)$ the marked elliptic $K3$ surface.
We say that two markings $\phi$ and $\phi\sprime$ 
of $(X, f, s)$ are \emph{$G$-isomorphic}
if there exists an element $g\in G$ such that 
$\phi=g\cdot \phi\sprime$ holds.
More generally, 
we say that two marked elliptic $K3$ surfaces 
$(X, f, s, \phi)$ and $(X\sprime, f\sprime, s\sprime, \phi\sprime)$
of type $(\Phi, A)$ are \emph{$G$-isomorphic}
if there exist an isomorphism $\psi\colon X\sprime \isom X$ of $K3$ surfaces and 
an element $g$ of $G$ 
that satisfy the following:
\begin{itemize}
\item 
We have $f\circ \psi =f\sprime$ and $\psi\circ s\sprime=s$, 
so that $\psi$ induces an isomorphism of elliptic $K3$ surfaces $ (X\sprime, f\sprime, s\sprime)\isom (X, f, s)$.
Hence the pull-back by $\psi$ induces an isomorphism
$$
\Phi_{\psi}\colon \Phi_f \isom \Phi_{f\sprime}
$$
of $ADE$-configurations.
\item
The diagram
$$
\begin{array}{ccc}
\Phi& \xrightarrow{g} & \Phi \\
\mapdownleft{\llap{\scriptsize $\phi$}} && \mapdownright{\phi\sprime} \\
\Phi_f & \maprightsb{\Phi_{\psi}} & \Phi_{f \sprime}
\end{array}
$$
commutes.
\end{itemize}
We consider the moduli space that parameterizes 
the $G$-isomorphism classes of marked elliptic $K3$ surfaces of type $(\Phi, A)$,
and define the set $\CCCC(\Phi, A, G)$ of connected components of this moduli space. 
\begin{remark} 
Theorem~\ref{thm:main} and Corollaries~\ref{cor:twelve},~\ref{cor:trans}
 stated in Introduction are for the case where $G=\Aut(\Phi)$.
\end{remark}
A \emph{connected family $(\XXX, F, S)/B$ of elliptic $K3$ surfaces of type $(\Phi, A)$}
is a commutative diagram
$$
\begin{array}{ccc}
\XXX & \xrightarrow{F} & \P^1_B \\
 \rlap{{\scriptsize $\pi$} $\searrow$} \hskip 5pt & \mystruth{14pt} &\hskip 14pt\llap{$\swarrow$ \hbox {\scriptsize $ \pi_P$}} \\
 &B& 
\end{array}
$$
with a section $S\colon \P^1_B\to \XXX$ of $F$ such that the following hold:
\begin{itemize}
\item $B$ is a connected analytic variety, $\pi\colon \XXX\to B$ is a family of $K3$ surfaces, $\pi_P\colon \P^1_B\to B$ is a $\P^1$-fibration, and 
\item for any point $t\in B$, the pullback $(X_t, f_t, s_t)$ of $(\XXX, F, S)$ by $\{t\}\inj B$
is an elliptic $K3$ surface of type $(\Phi, A)$.
\end{itemize}
Let $(\XXX, F, S)/B$ be a connected family as above.
For a point $t\in B$, 
we denote by 
$\Phi_t$ the $ADE$-configuration $\Phi_{f_t}$ of the elliptic $K3$ surface $(X_t, f_t, s_t)$.
The family $\shortset{\Phi_t}{t\in B}$ defines
a locally constant system 
$$
\Phi_B\to B
$$
of $ADE$-configurations.
A \emph{marking} of a connected family $(\XXX, F, S)/B$ 
is a choice of a base point $o\in B$ and a marking
$\phi_o\colon \Phi\isom \Phi_o$ of $(X_o, f_o, s_o)$.
We say that a marked connected family $( (\XXX, F, S)/B, \phi_o)$ 
is \emph{$G$-connected} if the image of the monodromy representation
$$
m_B \colon \pione(B, o) \to \Aut(\Phi)
$$
obtained from the locally constant system $\Phi_B\to B$ and the marking $\phi_o\colon \Phi\isom \Phi_o$
is contained in $G$.
%
%
Suppose that $( (\XXX, F, S)/B, \phi_o)$ is $G$-connected,
and let $t$ be a point of $B$.
Since $B$ is connected,
there exists a path $\gamma\colon [0,1]\to B$
from the base point $o$ to $t$.
The composite of $\phi_o\colon \Phi\isom \Phi_o$ and 
the transportation $\Phi_o\isom \Phi_t$ in the locally constant system $\Phi_B\to B$ along $\gamma$
gives rise to a marking
$$
\phi_t\colon \Phi\isom \Phi_t
$$
of $(X_t, f_t, s_t)$.
This marking depends on the choice of the path $\gamma$,
but the $G$-isomorphism class of $\phi_t$ 
is independent of the choice of $\gamma$.
Therefore a $G$-connected family
parametrizes a family of $G$-isomorphism classes of marked elliptic $K3$ surfaces.
\par
We say that two marked elliptic $K3$ surfaces $(X, f, s, \phi)$ and $(X\sprime, f\sprime, s\sprime, \phi\sprime)$
of type $(\Phi, A)$
are \emph{$G$-connected}
if there exists a marked $G$-connected family $( (\XXX, F, S)/B, \phi_o)$ of elliptic $K3$ surfaces of type $(\Phi, A)$
with two fibers $G$-isomorphic to $(X, f, s, \phi)$ and $(X\sprime, f\sprime, s\sprime, \phi\sprime)$,
respectively.
This relation of $G$-connectedness is an equivalence relation.
The transitivity is proved as follows.
Suppose that $(X_1, f_1, s_1, \phi_1)$ and $(X_2, f_2, s_2, \phi_2)$
are $G$-isomorphic to the fibers of a marked $G$-connected family $( (\XXX, F, S)/B, \phi_o)$
over $t_1\in B$ and $t_2\in B$, respectively,
and that $(X_2, f_2, s_2, \phi_2)$ and $(X_3, f_3, s_3, \phi_3)$ 
are $G$-isomorphic to the fibers of $( (\XXX\sprime, F\sprime, S\sprime)/B\sprime, \phi_{o\sprime})$
over $t_2\sprime\in B\sprime$ and $t_3\sprime\in B\sprime$, respectively.
Let $B\spprime$ be the connected analytic space obtained by
gluing $B$ and $B\sprime$ at $t_2\in B$ and $t_2\sprime\in B\sprime$,
and let $(\XXX\spprime, F\spprime, S\spprime)/B\spprime$ be the family
obtained by gluing $(\XXX, F, S)/B$ and $(\XXX\sprime, F\sprime, S\sprime)/B\sprime$ along the fibers 
over $t_2\in B$ and $t_2\sprime\in B\sprime$,
both of which are isomorphic to $(X_2, f_2, s_2)$.
Then $((\XXX\spprime, F\spprime, S\spprime)/B\spprime, \phi_o)$ is 
a marked $G$-connected family, and hence $(X_1, f_1, s_1, \phi_1)$
and $(X_3, f_3, s_3, \phi_3)$ are $G$-connected.
\par
We define a \emph{$G$-connected component of the moduli of elliptic $K3$ surfaces of type $(\Phi, A)$}
to be an equivalence class of the relation of $G$-connectedness
of marked elliptic $K3$ surfaces of type $(\Phi, A)$.
We denote by $\CCCC(\Phi, A, G)$ the set of $G$-connected components of this moduli.
\subsection{Lattice invariant of connected components}\label{subsec:invconn}
In this section,
we define a set $\QQQ(\Phi, A)/\mathordsim_G$
in purely lattice-theoretic terms, and establish a bijection 
$$
\bar\zeta\colon \CCCC(\Phi, A, G)\isom \QQQ(\Phi, A)/\mathordsim_G.
$$
We denote by $\Lgen{\Phi}$ the root lattice with a fundamental root system $\Phi$,
and put
$$
r_{\Phi}:=\rank{\Lgen{\Phi}}.
$$
\par
Let $(X, f, s)$ be an elliptic $K3$ surface of type $(\Phi, A)$.
Recall that  $\barLgen{\Phi_f}$ denotes the primitive closure of $\Lgen{\Phi_f}$ in $H^2(X, \Z)$,
that   $U_f$ denotes the sublattice of $H^2(X, \Z)$ generated by the class of a fiber of $f$
and the class of the section $s$.
Then $U_f$ is an even unimodular hyperbolic $\Z$-lattice of rank $2$,
and is orthogonal to $\Lgen{\Phi_f}$ in $H^2(X, \Z)$.
Hence $U_f\oplus \barLgen{\Phi_f}$ is a primitive sublattice of $H^2(X, \Z)$.
\begin{proposition}\label{prop:nnr}
We have $\barLgen{\Phi_f}/\Lgen{\Phi_f}\cong A$
and $\Roots(\barLgen{\Phi_f})=\Roots (\Lgen{\Phi_f})$.
\end{proposition}
\begin{proof}
The isomorphism $\barLgen{\Phi_f}/\Lgen{\Phi_f}\cong A_f\cong A$ 
is classically known.
See~\cite{MR1081832}, for example.
Note that $\Roots(\barLgen{\Phi_f})\supset \Roots (\Lgen{\Phi_f})$.
Suppose that $\Roots(\barLgen{\Phi_f})$ were strictly larger than $\Roots (\Lgen{\Phi_f})$.
Then there would be a smooth rational curve on $X$ whose class 
is orthogonal to $U_f$ but does not belong to $\Phi_f$,
which is a contradiction.
\end{proof}
Since $U_f$ is unimodular,
we have a canonical isomorphism 
$$
\Dq{U_f\oplus \barLgen{\Phi_f}}\isom\Dq{ \barLgen{\Phi_f}}.
$$
Recall that $T_f$ denotes the orthogonal complement of $U_f\oplus \barLgen{\Phi_f}$ in $H^2(X, \Z)$.
Then $T_f$ is an even $\Z$-lattice of signature $(2, 18-r_{\Phi})$.
Moreover, Corollary~\ref{cor:unimodoverlattice} implies that
we have a unique anti-isomorphism
$$
\alpha_f\colon (D_{\barLgen{\Phi_f}}, -q_{\barLgen{\Phi_f}})\isom \Dq{T_f}
$$
of discriminant forms 
that gives rise to the even unimodular overlattice $H^2(X, \Z)$ of $(U_f\oplus \barLgen{\Phi_f})\oplus T_f$.
Hence $T_f$ belongs to the genus determined by the signature $(2, 18-r_{\Phi})$
and the finite quadratic form  $(D_{\barLgen{\Phi_f}}, -q_{\barLgen{\Phi_f}})$.
Let $\omega_X\in H^2(X, \C)$ 
denote the class of a nowhere vanishing holomorphic $2$-form
on $X$,
which is unique up to a non-zero multiplicative constant.
We have $\omega_X\in T_f\tensor\C$, $\intf{\omega_X, \omega_X}=0$,
and $\intf{\omega_X, \bar{\omega}_X}>0$.
Let $H^{1, 1} (X, \R)\sperp $ denote the orthogonal complement of $H^{1, 1} (X, \R)$ in $H^2(X, \R)$.
Then  $H^{1, 1} (X, \R)\sperp $ is  a positive definite $2$-dimensional $\R$-lattice.
The two real vectors $\Real \omega_X$ and $\Imag \omega_X$ in this order
form an oriented orthogonal basis 
of the real subspace $H^{1, 1} (X, \R)\sperp $ of $ T_f\tensor \R$.
Thus the Hodge structure of $H^2(X)$ canonically defines a positive sign structure
$\theta_f$ of $T_f$.
\par
These geometric objects $M(\Phi_f)$, $T_f$, $\alpha_f$, and $\theta_f$
motivate the following lattice-theoretic definitions.
\begin{definition}
For an even overlattice $M$ of $L(\Phi)$,
let $\GGG(M)$ denote the genus of even $\Z$-lattices 
determined by the signature
$(2, 18-r_{\Phi})$ and the discriminant form $(D_{M}, -q_M)$.
Let $\EEE(\Phi, A)$ denote the set of even overlattices $M$ of $\Lgen{\Phi}$
such that
\begin{itemize}
\item $M/\Lgen{\Phi}\cong A$ and $\Roots(M)= \Roots (\Lgen{\Phi})$, and
\item $\GGG(M)$ is non-empty.
\end{itemize}
We define $\QQQ(\Phi, A)$ to be the set of quartets $(M, T, \alpha, \theta)$
of the following objects;
$M$ is an element of $\EEE(\Phi, A)$,
$T$ is an even $\Z$-lattice belonging to the genus $\GGG(M)$,
$\alpha$ is an isomorphism $(D_{M}, -q_M)\isom \Dq{T}$, 
and $\theta$ is a positive sign structure of $T$.
\end{definition}
\par
We define an equivalence relation $\mathordsim_G$ on the set $\QQQ(\Phi, A)$.
Since we have a natural homomorphism $\Aut(\Phi)\to \OG(\Lgen{\Phi})$,
the subgroup $G$ of $\Aut(\Phi)$ acts on the set $\EEE(\Phi, A)$.
Note that this action is from the right.
If $g\in G$ maps $M\in \EEE(\Phi, A)$ to $M\sprime\in \EEE(\Phi, A)$,
then $g$ induces an isometry $g|{M}\colon M\isom M\sprime$,
and hence an isomorphism $\funcq{g|{M}}\colon \Dq{M}\isom \Dq{M\sprime}$.
\begin{definition}
Let $(M, T, \alpha, \theta)$ and $(M\sprime, T\sprime, \alpha\sprime, \theta\sprime)$ be elements
of $\QQQ(\Phi, A)$.
We put $(M, T, \alpha, \theta)\sim_G (M\sprime, T\sprime, \alpha\sprime, \theta\sprime)$ 
if  there exist an automorphism $g\in G$
and an isometry $\psi\colon T\isom T\sprime$
with the following properties.
\begin{itemize}
\item $g$ maps $M$ to $M\sprime$,
\item $\psi$ maps $\theta$ to $\theta\sprime$, and 
\item the following diagram is commutative:
$$
\begin{array}{cc} 
\begin{array}{ccc}
\Dq{M} & \xrightarrow{\funcq{g|{M}}} & \Dq{M\sprime} \\
\mapdownleft{\alpha} & & \mapdownright{\alpha\sprime} \\
\Dq{T} &\maprightsb{\funcq{\psi} }& \Dq{T\sprime}.
\end{array} 
\end{array}
$$
\end{itemize}
\end{definition}
\par
\medskip %
Next we define a map
$\bar\zeta$ from $ \CCCC(\Phi, A, G)$ to $\QQQ(\Phi, A)/\mathordsim_G$.
Let $(X, f, s, \phi)$ be a marked elliptic $K3$ surface of type $(\Phi, A)$.
The marking $\phi\colon \Phi\isom \Phi_f$ induces an isometry 
$\phi_L\colon\Lgen{\Phi}\isom \Lgen{\Phi_f}$.
By Proposition~\ref{prop:nnr} and the existence of $T_f$, 
there exists a unique element $M_{f, \phi}$ of $ \EEE(\Phi, A)$
such that the isometry $\phi_{L}$
induces an isometry $\phi_M\colon M_{f, \phi}\isom M(\Phi_f)$.
The composite of the isomorphism $\Dq{M_{f, \phi}}\isom \Dq{M(\Phi_f)}$ 
induced by $\phi_M$ 
and the isomorphism $\alpha_f$ yields an isomorphism
$$
\alpha_{f, \phi}\colon (D_{M_{f, \phi}}, -q_{M_{f, \phi}})\isom \Dq{T_f}.
$$
Thus we obtain a quartet 
$$
\zeta (X, f, s, \phi):=(\,M_{f, \phi}, \; T_f, \; \alpha_{f, \phi}, \; \theta_f\,) \in \QQQ(\Phi, A).
$$
Suppose that 
marked elliptic $K3$ surfaces $(X, f, s, \phi)$ and $(X\sprime, f\sprime, s\sprime, \phi\sprime)$
are $G$-isomorphic.
Then  
we obviously have $\zeta(X, f, s, \phi) \sim_G\zeta(X\sprime, f\sprime, s\sprime, \phi\sprime)$
by definitions.
Let $((\XXX, F, S)/B, \phi_o)$
be a marked $G$-connected family of elliptic $K3$ surfaces of type $(\Phi, A)$.
For a point $t\in B$, 
let $(X_t, f_t, s_t)$ denote the fiber of $(\XXX, F, S)/B$ over $t$, 
and 
let  $\Phi_t$, $U_t$,  $\Lgen{\Phi_t}$, 
$\barLgen{\Phi_t}$, $T_t$, $\alpha_t$, and $\theta_t$ be the geometric objects
associated with $(X_t, f_t, s_t)$ defined above.
Let $\phi_t\colon \Phi\isom \Phi_t$ be the marking of $(X_t, f_t, s_t)$
induced by a path $\gamma\colon [0,1]\to B$ connecting $o$ and $t$.
The transportation along $\gamma$
induces an isometry $H^2(X_o, \Z)\isom H^2(X_t, \Z)$,
and this isometry induces isometries of sublattices
 $U_o\isom U_t$, $\Lgen{\Phi_o}\isom \Lgen{\Phi_t}$, 
$\barLgen{\Phi_o}\isom \barLgen{\Phi_t}$,
and $T_o\isom T_t$.
Hence the anti-isometries $\alpha_o$ and $\alpha_t$
are compatible with the isomorphisms $\Dq{M_o}\isom \Dq{M_t}$
and $\Dq{T_o}\isom \Dq{T_t}$ obtained by
the transportation along $\gamma$.
Since $\theta_t$ is defined by the Hodge structure of $X_t$,
the analytic structure of $F\colon\XXX\to B$ implies  that 
the isometry $T_o\isom T_t$ along $\gamma$ maps $\theta_o$ to $\theta_t$.
Hence we have $\zeta(X_o, f_o, s_o, \phi_o)\sim_G \zeta(X_t, f_t, s_t, \phi_t)$.
In other words, the map $\zeta$ induces a map 
$$
\bar\zeta\colon \CCCC(\Phi, A, G)\to \QQQ(\Phi, A)/\mathordsim_G.
$$
\begin{theorem}\label{thm:barzeta}
The map $\bar\zeta$ is a bijection.
\end{theorem}
\begin{proof}
Let $(M, T, \alpha, \theta)$ be an element of $\QQQ(\Phi, A)$.
Let $U$ denote the 
even unimodular hyperbolic 
$\Z$-lattice of rank $2$ with a basis $\vfib, \vzero$ and the Gram matrix
$$
\left[
\begin{array}{cc}
 0 & 1 \\ 1 & -2
 \end{array}
 \right].
$$
We define $H$ to be the even unimodular overlattice of $(U\oplus M)\oplus T$
defined by the anti-isometry
$$
(D_{U\oplus M}, -q_{U\oplus M})=(D_{M}, -q_{M})
\maprightspsb{\sim}{\alpha}\;\Dq{T}
$$
of the discriminant forms given by $\alpha$.
Then $H$ is a $K3$-lattice.
An \emph{$H$-marking} of a $K3$ surface $X$ is an isometry $\mu\colon H\isom H^2(X, \Z)$.
\par
Let $\P_*(T\tensor \C)$ denote the projective space of
$1$-dimensional subspaces of $T\tensor\C$.
We put 
$$
\Omega_{T}:=\set{\C\omega\in \P_*(T\tensor \C)}{\intf{\omega, \omega}=0,\; \intf{\omega, \bar{\omega}}>0}.
$$
A non-zero vector  $\omega=u+\sqrt{-1} v\in T\tensor\C$ with $u, v\in V:=T\tensor\R$
satisfies $\C\omega\in \Omega_T$
if and only if $(u,v)$ belongs to
$$
Z:=\set{(x, y)\in V\times V}{\intf{x, x}=\intf{y,y}>0, \intf{x, y}=0}.
$$
The image $Z_1$ of the first projection $\pr\colon Z\to V$ is connected
and $\pi_1(Z_1)\cong \Z$.
Since $t_+=2$, the orthogonal complement of a vector $u\in Z_1$ in $V$ has signature $(1, t_-)$, and hence 
$\pr\inv (u)=\shortset{y\in V}{\intf{y, y}=\intf{u,u}, \intf{u, y}=0}$
has two connected components.
We can easily see that $\pi_1(Z_1)$ acts on the set of these  connected components trivially.
Therefore $\Omega_{T}$ has exactly two connected components,
and they are complex conjugate to each other.
For $\C \omega\in \Omega_T$,
the ordered pair of vectors $\Real\omega$ and $ \Imag\omega$ in $T\tensor\R$
defines an oriented positive definite $2$-dimensional subspace of $T\tensor\R$.
By this correspondence,
the set of connected components of $\Omega_T$ can be identified with
the set of positive sign structures of $T$.
Let $\Omega_{(T, \theta)}$ be the connected component
corresponding to $\theta$.
\par
By the theory of the refined period map of marked $K3$ surfaces,
which is stated in Barth, Hulek, Peters and Van de Ven~\cite[Chapter VIII]{MR2030225},
 we see that there exists a universal family
 $$
\FFF_{(M, T, \alpha, \theta)}\colon \XXX \to
\Omega^0_{(T, \theta)}
$$
of $H$-marked $K3$ surfaces $(X_t, \mu_t)$
with parameter space $\Omega^0_{(T, \theta)}$
being an open dense subset of $\Omega_{(T, \theta)}$
such that, for each $t\in \Omega^0_{(T, \theta)}$,
the $H$-marking $\mu_t\colon H\isom H^2(X_t, \Z)$
satisfies the following:
\begin{itemize}
\item $\mu_t\tensor\C$ maps $\C\omega_t$ 
to $H^{2,0}(X_t)$, where $\C\omega_t$ is the $1$-dimensional subspace of 
$T\tensor\C \subset H\tensor \C$
corresponding to the point $t$, 
and hence the N\'eron-Severi lattice $\NS(X_t)$ of $X_t$ 
contains the primitive sublattice $\mu_t(U\oplus M)$ of $H^2(X_t, \Z)$, and 
\item $\mu_t$ maps the vectors $\vfib\in U$, $\vzero\in U$,
and each $r\in \Phi\subset H$
to the classes of certain irreducible curves
$F$, $Z$, and $C_r$ on $X_t$, respectively.
\end{itemize}
Then the complete linear system
$|F|$ defines an elliptic fibration $f_t\colon X_t\to \P^1$
and $Z$ provides us with a section $s_t$ of $f_t$.
Moreover,
the set $\shortset{C_r}{r\in \Phi}$
is the set of smooth rational curves on $X_t$ contracted by $f_t$ and disjoint from $s_t$,
and hence $\Phi_t:=\Phi_{f_t}$ is equal to $\shortset{[C_r]}{r\in \Phi}$.
We have $M(\Phi_t)/L(\Phi_t)\cong M/L(\Phi)\cong A$.
Therefore $(X_t, f_t, s_t)$ is of type $(\Phi, A)$,
and the $H$-marking $\mu_t$ yields a marking $\phi_t\colon \Phi\isom \Phi_t$.
\par
Thus, each element $(M, T, \alpha, \theta)$ of $\QQQ(\Phi, A)$ gives 
a connected family $\FFF_{(M, T, \alpha, \theta)}$ of marked elliptic $K3$ surfaces of type $(\Phi, A)$.
By the existence of $H$-markings, 
the monodromy of the family $\shortset{\Phi_t}{t\in \Omega^0_{(T, \theta)}}$ 
of $ADE$-configurations is trivial.
Any marked elliptic $K3$ surface $(X, f,s,\phi)$ of type $(\Phi, A)$
is isomorphic to a member of the family $\FFF_{\zeta(X, f,s,\phi)}$.
Hence the surjectivity of $\bar{\zeta}$ follows.
It follows from the universality of the family $\FFF_{\zeta(X, f,s,\phi)}$
that, if $(M, T, \alpha, \theta)\sim_G (M\sprime, T\sprime, \alpha\sprime, \theta\sprime)$,
then each member of $\FFF_{\zeta(X, f,s,\phi)}$
is $G$-isomorphic to a member of $\FFF_{(M\sprime, T\sprime, \alpha\sprime, \theta\sprime)}$.
Hence $\bar{\zeta}$ is injective.
\end{proof}
\subsection{Computation of the set $\QQQ(\Phi, A)/\mathordsim_G$}
Thus our problem of computing the set $\CCCC(\Phi, A, G)$ is reduced to the calculation of 
the set $\QQQ(\Phi, A)/\mathordsim_G$.
\par
Recall that $G$ acts on the set $\EEE(\Phi, A)$ from the right.
We have a projection 
$$
\pr_1 \colon \QQQ(\Phi, A)/\mathordsim_G \to \EEE(\Phi, A)/G
$$
given by $(M, T, \alpha, \theta)\mapsto M$.
The set of even overlattices of $\Lgen{\Phi}$ and the action of $G$ on it
can be easily calculated by Proposition~\ref{prop:isot}. 
From each $G$-orbit, we choose an even overlattice $M$, 
calculate $M/\Lgen{\Phi}$ and $\Roots(M)$, 
and determine whether $\GGG(M)$ is empty or not
by the criterion of Theorem~\ref{thm:genusexistence}.
In this way, we can compute the set 
$\EEE(\Phi, A)/G$.
\begin{remark}
For the calculation of $\Roots(M)$,
the technique of the lattice reduction bases
due to Lenstra, Lenstra, and Lov{\'a}sz~\cite{MR682664}
is very useful.
See~\cite[Chapter 2]{MR1228206}.
\end{remark}
The notion of algebraic equivalence
of connected components 
defined in Definitions~\ref{def:algequivX}~and~\ref{def:algequivC} 
 is now succinctly defined as follows.
\begin{definition}
We say that two connected components $\CCC_1, \CCC_2 \in \CCCC(\Phi, A, G)$
are \emph{algebraically equivalent} if $\pr_1 (\bar{\zeta}(\CCC_1))=\pr_1 (\bar{\zeta}(\CCC_2))$ holds.
\end{definition}
For a positive sign structure $\theta$ of $T$,
let $-\theta$ denote the other positive sign structure.
The mapping $(M, T, \alpha, \theta)\mapsto (M, T, \alpha, -\theta)$
defines an involution
$$
\compconj \colon \QQQ(\Phi, A)/\mathordsim_G \;\to\; \QQQ(\Phi, A)/\mathordsim_G.
$$
It is obvious that, via the bijection $\bar{\zeta}$,
this involution $\compconj$ 
corresponds to the action of $\Gal(\C/\R)$ on
$\CCCC(\Phi, A, G)$.
\begin{definition}\label{def:compconj}
Let $\CCC_1, \CCC_2\in \CCCC(\Phi, A, G)$ be two connected components.
We say that $\CCC_1$ is  \emph{complex conjugate to $\CCC_2$} if
$\compconj(\bar{\zeta}(\CCC_1))=\bar{\zeta}(\CCC_2)$ holds.
We say that $\CCC_1$ is \emph{real} if $\compconj(\bar{\zeta}(\CCC_1))=\bar{\zeta}(\CCC_1)$ holds.
\end{definition}
\par
We fix an even overlattice $M\in \EEE(\Phi, A)$.
Our next task is to calculate 
the fiber of $\pr_1$ over the $G$-orbit $[M]$ containing $M$.
Let $\Stab(M)\subset G$ denote the stabilizer subgroup of $M$ 
for the action of $G$ on $\EEE(\Phi, A)$.
Then we have a natural homomorphism
$$
\Stab(M)\to \OG(M).
$$
To ease the notation in the next section, we put 
$$
\GGG:=\GGG(M), \quad \Dq{\GGG}:=(D_M, -q_M).
$$
We further denote by
$$
\AG\subset\OGDq{\GGG}
$$
the image of $\Stab(M)$ by the natural homomorphism 
$\OG(M)\to \OGDq{\GGG}$.
\begin{definition}\label{def:TTTGGG}
Let $\TTTGGG$ be the set of triples $(T, \alpha, \theta)$,
where $T$ is an even $\Z$-lattice belonging to $\GGG$,
$\alpha$ is an isomorphism $\Dq{\GGG}\isom \Dq{T}$ of finite quadratic forms,
and $\theta$ is a positive sign structure of $T$.
\end{definition}
We define an equivalence relation $\simAG$ on $\TTTGGG$ by the following.
\begin{definition}\label{def:simAG}
Let  $(T, \alpha, \theta)$ and $(T\sprime, \alpha\sprime, \theta\sprime)$
be triples belonging to   $\TTTGGG$.
We put 
$(T, \alpha, \theta)\simAG (T\sprime, \alpha\sprime, \theta\sprime)$ if 
there exist an element $g\in \AG$ and an isometry $\phi\colon T\isom T\sprime$ 
that satisfy the following:
\begin{itemize}
\item
The diagram 
$$
\begin{array}{ccc}
\Dq{\GGG} & \xrightarrow{g} & \Dq{\GGG} \\
\mapdownleft{\alpha} & & \mapdownright{\alpha\sprime} \\
\Dq{T} &\maprightsb{\funcq{\phi} }& \Dq{T\sprime}
\end{array}
$$
commutes.
\item The isometry  $\phi$ maps $\theta$ to $\theta\sprime$.
\end{itemize}
\end{definition}
Then it is easy to see that
the fiber of $\pr_1$ over $[M]\in \EEE(\Phi, A)/G$
is canonically identified with 
$\TTTGGG/\mathord{\simAG}$.
In the next section,
we present an algorithm to calculate the set $\TTTGGG/\mathord{\simAG}$.
\section{Miranda-Morrison theory}\label{sec:MM}
This section and the next section are devoted to 
purely lattice-theoretic investigations,
and are completely independent of the geometry of $K3$ surfaces.
\par
Let $\GGG$ be a non-empty genus of even $\Z$-lattices 
determined by a signature $(t_+, t_-)$ with $t_+=2$
and a non-degenerate finite quadratic form $\Dq{\GGG}$.
Let $\AG$ be a subgroup of $\OGDq{\GGG}$.
We give an algorithm to calculate 
the set $\TTTGGG/\mathord{\simAG}$ defined in Definitions~\ref{def:TTTGGG}~and~\ref{def:simAG}.
We put
$$
\Sign:=\{1, -1\}.
$$
For a vector $v$ of an even lattice,
we put
$$
Q(v):=\frac{\intf{v,v}}{2}.
$$
\subsection{Spinor norm}
Let $R$ be $\Z$, $\Zp$, or $\R$,
let $k$ denote the quotient field of $R$,
and let $L$ be an even $R$-lattice.
Let $v$ be a vector of $L\tensor k$ such that $Q(v)\ne 0$.
Then we have the \emph{reflection} $\tau (v)\in \OG(L\tensor k)$ defined by
$$
\tau(v)\colon x\mapsto x-\frac{\intf{x, v}}{Q(v)} v.
$$
The classical theorem of Cartan (see~\cite[Chapter 1]{MR522835}) says that 
$\OG(L\tensor k)$ is generated by reflections.
Suppose that an isometry $g\in \OG(L)$ is decomposed into a product
$\tau (v_1)\cdots \tau (v_m)$ of reflections in $\OG(L\tensor k)$.
We define the \emph{spinor norm}  $\spin (g)$ of $g$ by 
$$
\spin (g):= Q(v_1) \cdots Q(v_m) \bmod (k\sptimes )^2.
$$
It is known that $\spin (g)\in k\sptimes/ (k\sptimes )^2$
does not depend on the choice of the decomposition $g=\tau (v_1)\cdots \tau (v_m)$,
and hence the map $\spin \colon \OG(L)\to k\sptimes/ (k\sptimes )^2$ is a group homomorphism.
(See~\cite[Chapter 10]{MR522835}.)
\begin{remark}\label{rem:warningspin1}
We use
the definition of $\spin (g)$ of $g=\tau (v_1)\cdots \tau (v_m)$
given in~\cite{MMB}, 
which differs from the one given in ~\cite{MR522835} by the multiplicative factor $2^m\in k\sptimes/ (k\sptimes )^2$.
\end{remark}
The following is due to~\cite{MR842425}.
See also~\cite{MR834537}.
\begin{proposition}\label{prop:detspin}
Suppose that $L$ is an $\R$-lattice,
so that the spinor norm takes values in $\Sign$. 
The action of an isometry $g\in \OG(L)$
on the set of positive sign structures of $L$ is trivial if and only if
$\det (g)\cdot \spin (g)>0$ holds.
\end{proposition}
\subsection{The case of positive definite genus}
Suppose that $t_-=0$,
so that $\GGG$ is a genus of even positive definite $\Z$-lattices of rank $2$.
By an algorithm that goes back to Gauss~(see, for example,~\cite[Chapter 15]{CSB}),
we can make the complete set of isomorphism classes of 
even positive definite $\Z$-lattices of rank $2$
with discriminant $|D_{\GGG}|$.
From this list, we sort out those lattices 
whose discriminant forms are isomorphic to $\Dq{\GGG}$,
and calculate a complete set 
$$
\{T_1, \dots, T_k\}
$$
of representatives of the genus $\GGG$.
For each $T$ in this list,
we calculate the finite groups $\OG(T)$ and $\OGDq{T}$,
an isomorphism $\alpha_{0}\colon \Dq{\GGG}\isom \Dq{T}$,
and the natural homomorphism $\OG(T)\to \OGDq{T}$.
Then the set of isomorphisms 
from $\Dq{\GGG}$ to $\Dq{T}$
is equal to 
$$
\set{\alpha_{0} \cdot h}{h \in \OGDq{T}}=\set{h\sprime\cdot \alpha_{0} }{h\sprime \in \OGDq{\GGG}}.
$$
Let $\alpha_{0*}\colon \OGDq{\GGG}\isom\OGDq{T}$
be the isomorphism induced by $\alpha_{0}$.
Since $T$ is positive definite, 
an isometry $\tilde{h}$ of $T$ preserves the positive sign structures of $T$
if and only if $\det(\tilde{h})=1$. 
We make $\AG\subset \OGDq{\GGG}$ act on $\OGDq{T}\times \Sign$ from the left by 
$$
(g, (\gamma, \theta))\mapsto (\,\alpha_{0*}(g)\cdot \gamma, \; \theta\,),
\;\;\textrm{where $g\in \AG$ and $(\gamma, \theta) \in \OGDq{T}\times \Sign$}.
$$
We also make $\OG(T)$ act on $\OGDq{T}\times \Sign$ from the right by
$$
((\gamma, \theta), \;\tilde{h})\mapsto (\,\gamma \cdot h, \;\det(\tilde{h})\cdot\theta\,),
\;\;\textrm{where $h\in \OGDq{T}$ is induced by $\tilde{h}\in \OG(T)$}.
$$
We consider the set of orbits
$$
\Orb(T)\;:=\;\AG\,\backslash (\OGDq{T}\times \Sign)/ \,\OG(T)
$$
under these actions.
Then the set of all $(\,T, \; \alpha_{0}\cdot \gamma, \;\theta\,)\in \TTTGGG$,
where $T$ runs through the set $\{T_1, \dots, T_k\}$,
and for each $T$,
$ (\gamma, \theta)$
runs through the set of representatives of $\Orb(T)$, 
is a complete set of representatives of $\TTTGGG/\mathord{\simAG}$.
By this algorithm,
we compute Table~I.
\subsection{Miranda-Morrison theory}\label{subsec:MirandaMorison}
From now on to the end of this section,
we assume that $t_->0$.
Hence $\GGG$ is a genus of even indefinite $\Z$-lattices of rank $\ge 3$.
We formulate a refinement of Miranda-Morrison theory~\cite{MMB}
on the structure of a genus of this kind.
\par
We first review the original version of Miranda-Morrison theory,
which calculates the set $\TTTGGG\sprime/\mathordsim$ defined as follows.
Let $\TTTGGG\sprime$ be the set of pairs $(T, \alpha)$,
where $T$ is an even $\Z$-lattice belonging to $\GGG$,
and $\alpha$ is an isomorphism $\Dq{\GGG} \isom\Dq{T}$. 
For elements $(T, \alpha)$ and $(T\sprime, \alpha\sprime)$ of $\TTTGGG\sprime$,
we put 
$(T, \alpha)\sim (T\sprime, \alpha\sprime)$ if 
there exists an isometry $\phi\colon T\isom T\sprime$ such that
the diagram
$$
\begin{array}{ccc}
\Dq{\GGG} &= &\Dq{\GGG} \\
\mapdownleft{\alpha} & & \mapdownright{\alpha\sprime} \\
\Dq{T} &\maprightsb{\funcq{\phi} }& \Dq{T\sprime}
\end{array}
$$
commutes.
\par
We fix an element $(L, \lambda)$ of $\TTTGGG\sprime$, and put
\begin{eqnarray*}
\OGAz (L) &:= &\prod_{p} \OG(L\tensor\Zp), \\
\OGA(L) &:= &\set{(\sigma_p)\in \prod_{p} \OG(L\tensor\Qp)}{\sigma_p\in \OG(L\tensor\Zp)\;\;\textrm{for almost all}\;\; p}.
\end{eqnarray*}
Note that we have a natural homomorphism $\OG(L\tensor\Q)\to \OGA(L) $. 
Let $\adele{\sigma}=(\sigma_p)$ be an element of $\OGA(L)$.
Then there exists a unique $\Z$-submodule $L^\adelesigma $ of $L\tensor \Q$ such that
$L^\adelesigma \tensor\Zp =(L\tensor\Zp)^{\sigma_p}$ holds in $L\tensor\Qp$ for all $p$,
where $(L\tensor\Zp)^{\sigma_p}$ is the image of $L\tensor\Zp\subset L\tensor\Qp$
by $\sigma_p\in \OG(L\tensor\Qp)$.
(See~Theorem 4.1 in~\cite[Chapter VI]{MMB}.)
We restrict the symmetric bilinear form of $L\tensor \Q$ to $L^\adelesigma $.
Since
the $\Zp$-lattices
 $L^\adelesigma \tensor\Zp$ and $L\tensor\Zp$ are isomorphic  for all $p$,
 we see that $L^\adelesigma $ is an even $\Z$-lattice belonging to $\GGG$.
 Note that we have $L^\adelesigma =L$ if and only if $\adelesigma \in \OGAz (L)$.
 Let $\adeletau=(\tau_p)$ be an element of $\OGA(L)$.
 Then each component $\tau_p$ of $\adeletau$ induces an isometry $L^\adelesigma \tensor \Zp\isom L^{\adelesigma \adeletau}\tensor\Zp$, 
 and hence induces an isomorphism
 $$
 \funcq{\tau_p}\colon \Dq{L^\adelesigma \tensor\Zp}\isom \Dq{L^{\adelesigma \adeletau} \tensor\Zp}.
 $$
 Their product over the primes $p$ dividing $|\rmdisc(L)|=|D_{\GGG}|$ gives rise to an isomorphism 
 $$
 \funcq{\adeletau}|_{L^\adelesigma }\colon \Dq{L^\adelesigma }\isom \Dq{L^{\adelesigma \adeletau}}.
 $$
If $\adeletau\in \OGAz (L)$, then $ \funcq{\adeletau}|_{L}\in \OGDq{L}$.
As a corollary of Proposition~\ref{prop:etasurjZp}, we obtain the following:
\begin{proposition}\label{prop:adicsurjadele}
The homomorphism $\OGAz (L) \to\OGDq{L}$ given by $\adeletau\mapsto \funcq{\adeletau}|_{L}$
is surjective.
\end{proposition}
For $\adelesigma\in \OGA(L)$, 
we put 
 $$
 \lambda^\adelesigma :=\lambda\cdot \funcq{\adelesigma }|_{L}
 \colon \Dq{\GGG} \isom \Dq{L^\adelesigma }.
 $$
 Thus we obtain 
 a map $\OGA(L)\to \TTTGGG\sprime$
 given by $\adelesigma\mapsto (L^\adelesigma, \lambda^\adelesigma )$.
We show that this map is surjective.
Let $(T, \alpha)$ be an arbitrary element of $\TTTGGG\sprime$.
By Proposition~\ref{prop:corweakHasse},
we can assume that $T$ is embedded into $L\tensor\Q$ isometrically,
and hence we have $L\tensor \Q=T\tensor\Q$.
Then the equality $L\tensor \Zp=T\tensor\Zp$ holds in $L\tensor\Qp$ for almost all $p$.
For each $p$, we have an isometry $\sigma_p\colon L\tensor\Zp\isom T\tensor\Zp$,
which we regard as an element of $\OG(L\tensor \Qp)$.
We put $\adelesigma:=(\sigma_p)$.
Since $L\tensor \Zp=T\tensor\Zp$  for almost all $p$,
we see that  $\adelesigma$ belongs to $\OGA(L)$, and 
we have  $T=L^\adelesigma $.
We also obtain an isomorphism $\funcq{\adelesigma }|_{L}\colon \Dq{L}\isom \Dq{T}$.
Consider the diagram 
$$
\begin{array}{ccc}
&\Dq{\GGG}& \\
\rlap{\phantom{a}$\raise 5pt \hbox{$\scriptstyle \alpha$}\swarrow$ }& & 
\llap{\hbox{$\searrow \raise 4pt \hbox{$\scriptstyle \lambda$}$}\phantom{a}} \\
\Dq{T}& \mapleftsb{q_{\adelesigma}|L}& \Dq{L}.\mystrutd{12pt}
\end{array}
$$
We see that $\lambda\inv\cdot \alpha\cdot (\funcq{\adelesigma }|_{L})\inv$ belongs to $\OGDq{L}$.
By Proposition~\ref{prop:adicsurjadele},
there exists an element $\adelerho\in \OGAz(L)$
such that $\funcq{\adelerho}|_{L}=\lambda\inv\cdot \alpha\cdot (\funcq{\adelesigma }|_{L})\inv$.
Then we have $(T, \alpha)=(L^{\adelerho\adelesigma }, \lambda^{\adelerho\adelesigma })$.
Therefore the mapping $\adelesigma \mapsto (L^\adelesigma, \lambda^\adelesigma )$ 
is surjective, and we obtain
$$
 \OGA(L) \surj \TTTGGG\sprime \surj \TTTGGG\sprime/\mathordsim.
$$
\par
Let $U_p$ denote the image of 
the natural homomorphism $\Zp\sptimes\inj \Q_p\sptimes \to \Qp\sptimes/(\Qp\sptimes)^2$.
Recall that $\Det=\{1, -1\}$.
We put
$$
\Gamma_{p, 0}:=\Det\times U_p\;\; \subset \;\; \Gamma_p:=\Det\;\times\; \Qp\sptimes/(\Qp\sptimes)^2.
$$
Note that $\Gamma_p$ is an elementary $2$-group of rank $4$ if $p=2$ and of rank $3$ if $p>2$,
and that $\Gamma_{p, 0}$ is of index $2$ in $\Gamma_p$.
We consider the homomorphism
$$
\detspin\colon \OG(L\tensor\Zp)\to \Gamma_p.
$$
\begin{definition}
Let $\OGsharp (L\tensor \Zp)$ denote the kernel of the natural homomorphism
$\OG(L\tensor\Zp)\to \OGDq{L\tensor\Zp}$, and 
let $\SigmaSharp(L\tensor\Zp)$ denote the image of $\OGsharp(L\tensor\Zp)$ by $\detspin$.
\end{definition}
The abelian group $\SigmaSharp(L\tensor\Zp)$ 
is completely calculated in~\cite{MR0839800} and~\cite[Chapter VII]{MMB}.
In particular, we have the following proposition.
(See~Theorems 12.1-12.4
and Corollary 12.11 in~\cite[Chapter VII]{MMB}.)
Recall that we have assumed that $L$ is of rank $\ge 3$.
\begin{proposition}\label{prop:sigmasharp}
\begin{enumerate}[{\rm (1)}]
\item We have $\SigmaSharp(L\tensor\Zp)\subset \Gamma_{p, 0}$.
\item If $L\tensor\Zp$ is unimodular, 
then $\SigmaSharp(L\tensor\Zp)=\Gamma_{p, 0}$.
\end{enumerate}
\end{proposition}
We put 
$$
\GammaAz:= \prod_{p} \Gamma_{p, 0} \;\; \subset\;\; 
\GammaA:=\set{(\gamma_p)\in \prod_{p} \Gamma_p}{ \gamma_p\in \Gamma_{p,0} \;\;\textrm{for almost all}\;\; p}.
$$
If $L\tensor\Zp$ is unimodular, we have $\OGsharp (L\tensor \Zp)=\OG(L\tensor \Zp)$,
and hence Proposition~\ref{prop:sigmasharp} implies that the image of $\OG(L\tensor \Zp)$
by $\detspin$ is $\Gamma_{p, 0}$.
Since $L\tensor\Zp$ is unimodular for almost all $p$, 
we obtain a homomorphism
$$
\detspin\colon \OGA(L)\to \GammaA.
$$
We put 
$$
\SigmaSharpA(L):=\prod_{p} \SigmaSharp(L\tensor\Zp)\;\; \subset\;\; \GammaAz.
$$
Finally, we put
$$
\GammaQ:=\Det\;\times\; \Q\sptimes/(\Q\sptimes)^2, 
$$
and embed $\GammaQ$ into $\GammaA$ naturally.
We have the following proposition. See Proposition 6.1 in~\cite[Chapter V]{MMB}.
\begin{proposition}\label{prop:surjGammaQ}
If $V$ is an indefinite $\Q$-lattice of rank $\ge 3$, then 
the homomorphism $(\det, \spin)\colon \OG(V) \to\GammaQ$ is surjective.
\end{proposition}
\par
One of the principal results of Miranda-Morrison theory is as follows
(see Theorem 3.1 in~\cite[Chapter VIII]{MMB}).
\begin{theorem}\label{thm:originalMM}
Let $\adelesigma $ and $\adeletau$ be in $\OGA(L)$.
Then we have $(L^\adelesigma, \lambda^\adelesigma)\sim (L^\adeletau, \lambda^\adeletau)$ if and only if
$$
(\det( \adelesigma ), \spin (\adelesigma ) )\equiv (\det (\adeletau), \spin (\adeletau))
\;\;\bmod\;\; \GammaQ \cdot \SigmaSharpA(L)
$$
holds in $\GammaA$.
In particular,
we can endow the set $\TTTGGG\sprime/\mathordsim$ with a structure of  abelian group. 
\end{theorem}
The main ingredient of the proof of Theorem~\ref{thm:originalMM}
is the following corollary (Theorem 2.2 in~\cite[Chapter VIII]{MMB}) of the strong approximation theorem (Theorem 7.1 in~\cite[Chapter 10]{MR522835})
for the \emph{spin group}
$$
\Theta_{\A}(L):=\Ker (\detspin\colon \OGA(L)\to\GammaA).
$$
Recall that $L$ is indefinite of rank $\ge 3$.
\begin{theorem}\label{thm:strongapproximation}
Let $\adelesigma $ be an element of $\OGA(L)$.
For any element $\adelepsi\sprime$ of $\Theta_{\A}(L)$,
there exists 
an isometry $\psi\in \OG(L\tensor\Q)$ such that 
$(\det(\psi), \spin(\psi))=(1,1)$,
that
$L^{\adelesigma \psi}=L^{\adelesigma \adelepsi\sprime}$, 
and that $\funcq{\psi}|_{L^\adelesigma }\colon \Dq{L^\adelesigma}\isom \Dq{L^{\adelesigma \psi}}$ is equal to $\funcq{\adelepsi\sprime}|_{L^\adelesigma }$.
\end{theorem}
Indeed, the set of all $\adeletau\in \Theta_{\A}(L)$
 that satisfy $L^{\adelesigma \adeletau}=L^{\adelesigma \adelepsi\sprime}$ and $\funcq{\adeletau}|_{L^\adelesigma }=\funcq{\adelepsi\sprime}|_{L^\adelesigma }$ 
is a non-empty open subset of $\Theta_{\A}(L)$ whose $p$-component coincides
with 
$$
\Theta(L \tensor\Zp):=\Ker (\detspin\colon \OG(L \tensor \Zp)\to\Gamma_{p})
$$
for almost all $p$.
\begin{remark}\label{rem:warningspin2}
Even though the definition of the spinor norm 
in~\cite{MMB} and in this paper differs from the one given in~\cite{MR522835},
the definition of the spin group is not affected, 
because, 
for any element $g=\tau (v_1)\cdots \tau (v_m)$ of a spin group,
the condition $\det(g)=1$ implies $m\equiv 0\bmod 2$. 
See Remark~\ref{rem:warningspin1}.
\end{remark}
\subsection{A refinement of Miranda-Morrison theory}
We refine Theorem~\ref{thm:originalMM}
in order to incorporate the positive sign structures and the action of $\AG$. 
\par
As in the previous section,
we fix an element $(L, \lambda, \theta)$ of $\TTTGGG$.
For each $\adelesigma \in \OGA(L)$,
we have $L^\adelesigma \tensor\R=L\tensor \R$,
and hence 
the fixed positive sign structure $\theta$ of $L$ 
induces a positive sign structure on $L^\adelesigma $,
which is denoted by the same symbol $\theta$,
and let $-\theta$ denote the other positive sign structure of $L^\adelesigma$.
Recall that $\Sign=\{\pm 1\}$.
The surjectivity of the map $\OGA(L)\to \TTTGGG\sprime/\mathordsim$ defined in the previous section implies that 
 the mapping
$$
(\adelesigma, \varepsilon) \mapsto (L^\adelesigma, \lambda^\adelesigma, \varepsilon \theta)
$$
induces a surjective map
$$
 \OGA(L)\times \Sign \to \TTTGGG/\mathord{\simAG}.
$$
\begin{definition}
We define a homomorphism
$$
\Psi_p\colon \OGDq{L\tensor\Zp} \to \Gamma_p/\SigmaSharp(L\tensor\Zp)
$$
by $g\mapsto (\det (\tilg), \spin (\tilg))\bmod \SigmaSharp(L\tensor\Zp)$,
where $\tilg\in \OG(L\tensor\Zp)$ is an isometry that induces $g$ on $\Dq{L\tensor\Zp}$.
(Since the natural homomorphism 
$\OG(L\tensor\Zp)\to \OGDq{L\tensor\Zp}$
is surjective (see Proposition~\ref{prop:etasurjZp}), 
we can always find a lift $\tilg$ of $g$,
and by the definition of $\SigmaSharp(L\tensor\Zp)$,
we see that $\Psi_p(g)$ does not depend on the choice of the lift $\tilg$.)
\end{definition}
Since $\OGDq{L}$ is a product of $\OGDq{L\tensor\Zp}$,
all of which is trivial except for $p$ dividing $|\rmdisc(L)|=|D_{\GGG}|$,
we obtain a homomorphism
\begin{equation*}\label{eq:liftedhom}
\Psi_{\A}\colon \OGDq{L}\to \GammaA/ \SigmaSharpA(L).
\end{equation*}
\begin{definition}
Let $\AG^\lambda$ denote the subgroup of $\OGDq{L}$ corresponding to the given subgroup 
$\AG\subset \OGDq{\GGG}$
by the fixed isomorphism $\lambda\colon \Dq{\GGG}\isom \Dq{L}$:
$$
\AG^\lambda:=\set{\lambda\inv\cdot g\cdot \lambda \in \OGDq{L}}{g\in \AG}.
$$
We then define 
$\Sigma (L, \AG^\lambda)$
to be the subgroup 
of $\GammaA$ containing $\SigmaSharpA(L)$ such that
\begin{equation}\label{eq:SigmaLG}
\Psi_{\A}(\AG^\lambda)=\Sigma (L, \AG^\lambda)/\SigmaSharpA(L)
\end{equation}
holds; that is, 
$$
\Sigma (L, \AG^\lambda):=\sethd{\adele{\gamma}\in \GammaA}{7cm}{%
\textrm{there exists an element $\adelesigma\in \OGAz(L)$ such that $\funcq{\adelesigma}|_L\in \AG^\lambda$ and 
that $(\det(\adelesigma), \spin(\adelesigma))=\adele{\gamma}$}%
}.
$$
\end{definition}
We have a natural homomorphism 
$\GammaQ \to \Sign$
that maps $(d, s)$ to the sign of $ds$, where $d\in \Det$ and $s\in \Q\sptimes/(\Q\sptimes)^2$.
We then 
define an embedding 
$$
\GammaQ\inj \GammaA\times \Sign
$$
by the product of the natural embedding $\GammaQ\inj \GammaA$ and the above homomorphism $\GammaQ \to \Sign$,
and denote by $\Gamma^{\mathordsim}_{\Q}$ the image of $\GammaQ$  in $\GammaA\times \Sign$.
\par
The main result of this subsection is as follows:
\begin{theorem}\label{thm:refinedMM}
Let $(\adelesigma, \varepsilon)$ and $(\adeletau, \eta)$ be elements of $\OGA(L)\times \Sign$.
Then we have $(L^\adelesigma, \lambda^\adelesigma, \varepsilon)\simAG(L^\adeletau, \lambda^\adeletau, \eta)$ if and only if
\begin{equation}\label{eq:thecond}
(\det( \adelesigma ), \spin (\adelesigma ), \varepsilon )\equiv (\det (\adeletau), \spin (\adeletau), \eta)\;\;\bmod\;\; \Gamma^{\mathordsim}_{\Q} \cdot (\Sigma (L, \AG^\lambda)\times \{1\})
\end{equation}
holds in $\GammaA\times \Sign$. 
\end{theorem}
\begin{proof}
The proof is parallel to that of Theorem~\ref{thm:originalMM} in~\cite[Chapter VIII]{MMB}.
Suppose that $(L^\adelesigma, \lambda^\adelesigma, \varepsilon)\simAG(L^\adeletau, \lambda^\adeletau, \eta)$ holds.
Then there exist an element $\agg\in \AG^\lambda$ and an isometry $\phi\colon L^\adelesigma \isom L^\adeletau$
of even $\Z$-lattices 
such that
the diagram
\begin{equation}\label{eq:diagram}
\begin{array}{ccc}
\Dq{L} & \xrightarrow{\agg} & \Dq{L} \\
\mapdownleft{\hskip -10pt \funcq{\adelesigma }|_L} & & \mapdownright{\funcq{\adeletau}|_L} \\
\Dq{L^\adelesigma } &\maprightsb{\funcq{\phi} }& \Dq{L^{\adeletau}}
\end{array}
\end{equation}
commutes, and that
\begin{equation}\label{eq:phiR}
\varepsilon\cdot \det(\phi\tensor\R)\cdot \spin( \phi\tensor\R) =\eta
\end{equation}
holds by Proposition~\ref{prop:detspin}.
(Note that we have $\spin(\phi\tensor\R) \in \R\sptimes/(\R\sptimes)^2=\{\pm 1\}$.)
We have an element $\tilde{\adele{\agg}}\in \OGAz (L)$ that induces $\agg$ on $\Dq{L}$
by the surjectivity of $\OGAz(L)\to \OGDq{L}$ (see Proposition~\ref{prop:adicsurjadele}).
Then the product $\tilde{\adele{\agg}}\cdot \adeletau\cdot\phi\inv\cdot \adelesigma \inv$ belongs to $\OGAz (L)$, and it induces an identity on $\Dq{L}$ by the commutativity of~\eqref{eq:diagram}.
Hence we have 
$$
(\,\det(\adelesigma \cdot\phi\cdot \adeletau\inv), \; \spin(\adelesigma \cdot\phi\cdot \adeletau\inv)\, )
\bmod \SigmaSharpA(L)\;\; =\;\; \Psi_{\A}(\agg).
$$
In particular, we have 
$$
(\det(\adelesigma ), \spin(\adelesigma ))\cdot (\det(\adeletau), \spin(\adeletau))\inv \cdot (\det(\phi), \spin(\phi))\in \Sigma(L, \AG^\lambda).
$$
Note that $ (\det(\phi), \spin(\phi))\in \GammaQ$, and 
that~\eqref{eq:phiR} implies that
$(\det(\phi), \spin(\phi), \varepsilon\inv \eta )$ belongs to $\Gamma^{\mathordsim}_{\Q}$.
Hence~\eqref{eq:thecond} holds.
\par
Conversely,
suppose that~\eqref{eq:thecond} holds.
By the definition of $\Sigma(L, \AG^\lambda)$ and the surjectivity of $\OG(L\tensor\Q)\to \GammaQ$
(see Proposition~\ref{prop:surjGammaQ}), 
we obtain $\tilde{\adele{\agg}}\in \OGAz(L)$ and $\xi\in \OG(L\tensor\Q)$ such that
\begin{enumerate}[(i)]
\item the automorphism $\agg$ of $\Dq{L}$ induced by $\tilde{\adele{\agg}}$ belongs to $\AG^\lambda$,
\item $(\det(\tilde{\adele{\agg}}), \spin(\tilde{\adele{\agg}}))\cdot (\det(\adeletau), \spin(\adeletau))=(\det(\adelesigma ), \spin(\adelesigma ))\cdot (\det(\xi), \spin(\xi))$ holds in $\GammaA$, and
\item $\varepsilon\cdot \det(\xi\tensor\R)\cdot\spin(\xi\tensor\R)=\eta$.
\end{enumerate}
We put
$$
\adelepsi\sprime:=\xi\inv\cdot \adelesigma \inv\cdot \tilde{\adele{\agg}}\cdot \adeletau \in \OGA(L).
$$
Note that we have $L^{\adelesigma \xi \adelepsi\sprime}=L^\adeletau$.
Since we have $(\det(\adelepsi\sprime), \spin(\adelepsi\sprime))=(1,1)$ by property (ii) above,
Theorem~\ref{thm:strongapproximation} implies 
that there exists an element $\psi\in \OG(L\tensor\Q)$ such that
$(\det(\psi), \spin(\psi))=(1, 1)$,
that $L^{\adelesigma \xi\psi}=L^\adeletau$,
and that the isomorphism from $\Dq{L^{\adelesigma \xi}}$ to $\Dq{L^\adeletau}$ induced by $\psi$ is equal to 
$$
\funcq{\adelepsi\sprime}|_{L^{\adelesigma \xi}}=(\funcq{\xi}|_{L^{\adelesigma }})\inv \cdot (\funcq{\adelesigma }|_{L})\inv \cdot g \cdot (\funcq{\adeletau}|_{L}).
$$
We put 
$$
\phi:=\xi\cdot \psi,
$$
which belongs to $\OG(L\tensor \Q)$.
Then the diagram~\eqref{eq:diagram}
commutes.
Moreover, 
since $(\det(\psi), \spin(\psi))=(1, 1)$,
we have $\varepsilon\cdot \det(\phi\tensor\R)\cdot\spin(\phi\tensor\R)=\eta$
by property (iii) above.
Therefore we obtain $(L^\adelesigma, \lambda^\adelesigma, \varepsilon)\simAG(L^\adeletau, \lambda^\adeletau, \eta)$.
\end{proof}
Therefore the set $\TTTGGG/\mathord{\simAG}$
can be equipped with a structure of abelian group:
\begin{equation}\label{eq:TTTGGG}
\TTTGGG/\mathord{\simAG}\;\;=\;\;(\GammaA\times \Sign)\;/ \;(\Gamma^{\mathordsim}_{\Q} \cdot (\Sigma (L, \AG^\lambda)\times \{1\})).
\end{equation}
Let $\PD=\{p_1, \dots, p_m\}$ denote the set of primes that divide
$$
d:=|D_{\GGG}|=|D_L|=|\rmdisc(L)|.
$$
We put
$$
\GammaD:=\prod_{p\in \PD} \Gamma_p,
\quad
\GammaAD:=\left(\prod_{p\notin \PD} \Gamma_{p, 0}\right)\; \times \;\GammaD.
$$
We show that the abelian group $\TTTGGG/\mathord{\simAG}$ is isomorphic to a quotient
of $\GammaD\times \Sign$,
and present a set of generators of the kernel $K$ of the quotient  homomorphism 
$\GammaD\times \Sign \to\TTTGGG/\mathord{\simAG}$.
Note that the finite abelian group  $\GammaD\times \Sign$ is  $2$-elementary,
and hence the computation below can be carried out by linear algebra over $\F_2$.
\begin{lemma}\label{lem:GammaAD}
We have $\GammaA\times\Sign =(\GammaAD\times \Sign)\cdot \GammaQtilde$.
\end{lemma}
\begin{proof}
This follows from $\GammaAz\subset\GammaAD$ and $\GammaA=\GammaAz\cdot \GammaQ$ (see Lemma 4.1 in~\cite[Chapter VIII]{MMB}).
\end{proof}
By this lemma,
we have an exact sequence
$$
0
\;\;\to\;\;
(\GammaAD\times \Sign)\cap \GammaQtilde
\;\;\to\;\;
\GammaAD\times \Sign
\;\;\to\;\;
(\GammaA\times\Sign)/\GammaQtilde
\;\;\to\;\;
0.
$$
By the definition~\eqref{eq:SigmaLG} of $\Sigma(L, \AG^\lambda)$
and Proposition~\ref{prop:sigmasharp}, 
we see that $\Sigma(L, \AG^\lambda)$ is contained in $\GammaAD$.
Hence the finite abelian group $\TTTGGG/\mathord{\simAG}$
is isomorphic to the cokernel of
$$
(\GammaAD\times \Sign)\cap \GammaQtilde 
\;\;\inj\;\;
\GammaAD\times \Sign
\;\;\surj\;\;
(\,\GammaAD\,/\,\Sigma(L, \AG^\lambda)\,)\times \Sign.
$$
Recall that $\AG^\lambda$ is a subgroup of
\begin{equation}\label{eq:OGpdecomp}
\OGDq{L}=\prod_{p \in \PD}\OGDq{L\tensor\Z_{p}}.
\end{equation}
Suppose that $\AG^\lambda$ is generated by $g_1, \dots, g_k$.
Let $p$ be a prime in $\PD$.
We denote by $g_i[p]$ the $p$-component of $g_i\in \AG^\lambda$ under the direct-sum decomposition~\eqref{eq:OGpdecomp}.
We then choose an isometry
$$
g_i[p]\lift \;\in\; \OG(L\tensor\Z_{p})
$$
that induces $g_i[p]$ on $\Dq{L\tensor\Z_{p}}$.
Then $\Psi_{p}(g_i) \in \Gamma_{p}/\SigmaSharp(L\tensor\Z_{p})$ 
is represented by $(\det(g_i[p]\lift ), \spin(g_i[p]\lift ))\in \Gamma_{p}$.
We put
$$
\gamma(g_i)\;\;:=\;\; \bigl(\; (\det(g_i[p]\lift ), \;\spin(g_i[p]\lift ))\; \mid\; p\in \PD \;\bigr)\;\;\in\;\; \GammaD.
$$
Remark that $\gamma(g_i)$ \emph{does} depend on the choice of the lifts $g_i[p]\lift$ of $g_i[p]$,
but that $\gamma(g_i)$ modulo
$\prod_{p\in \PD} \SigmaSharp(L\tensor\Z_{p})$ is uniquely determined by $g_i$.
By Proposition~\ref{prop:sigmasharp},
the projection $\GammaAD\to \GammaD$ induces an isomorphism from 
 $\GammaAD\,/\,\Sigma(L, \AG^\lambda)$ to the group 
$$
\GammaD\;\;/\;\; \gen{\;\SigmaSharp(L\tensor\Z_{p_1}), \dots, \SigmaSharp(L\tensor\Z_{p_m}), \; \gamma(g_1), \dots, \gamma(g_k)\; }.
$$
On the other hand, 
the group
$(\GammaAD\times \Sign)\cap \GammaQtilde $ is generated by the
following $2+|P(d)|$ elements of $\GammaQtilde\subset \Det \times (\Q\sptimes/(\Q\sptimes)^2) \times \Sign$:
$$
(-1, 1, -1),
\;\;
(1, -1, -1),\;\;
\textrm{and}
\;\;
(1, p_j, 1)\;\; \textrm{for $p_j\in\PD$}. 
$$
We put $S(d):=\shortset{p_1^{\nu_1} \cdots p_m^{\nu_m}}{ \nu_j=0 \textrm{\;or\,} 1\;\textrm{for}\; j=1, \dots, m}$.
The image of $(\varepsilon, \eta s, \varepsilon \eta)\in (\GammaAD\times \Sign)\cap \GammaQtilde$, where $\varepsilon\in\Det$, $\eta\in \Sign$, 
and $s\in S(d)$, by the projection $\GammaAD\times \Sign \to \GammaD\times\Sign$ is
$$
\beta(\varepsilon, \eta s, \varepsilon \eta):=\bigl(\;(\,\varepsilon, \,\eta s\bmod (\Q_{p_j}\sptimes)^2\,) \mid p_j \in \PD ), \;\; \varepsilon\eta\;\bigr). 
$$
Hence we obtain the following:
\begin{proposition}\label{prop:quotK}
The finite abelian group $\TTTGGG/\mathord{\simAG}$ is isomorphic to $(\GammaD\times\Sign)/K$,
where $K$ is generated by the following subgroups and elements of $\GammaD\times\Sign$:
\begin{enumerate}[{\rm (i)}]
\item
$\SigmaSharp(L\tensor \Z_{p_j})\times \{1\}$ for $p_j\in \PD$,
\item 
$(\gamma(g_i), 1)$, where $\AG^\lambda$ is generated by $g_1, \dots, g_k$, and
\item 
$\beta(-1,1,-1)$, $\beta(1, -1, -1)$, and 
$\beta(1, p_j, 1)$ for $p_j\in \PD$.
\end{enumerate} 
\end{proposition}
The groups $\SigmaSharp(L\tensor\Zp)$ for $p\in \PD$ have been calculated in terms of 
$\rank(L\tensor\Zp)$, $\disc(L\tensor\Zp)$, 
and $ \Dq{L\tensor\Zp}$
in~\cite{MR0839800} and~\cite[Chapter VII]{MMB}.
The computation of $\beta(\varepsilon, \eta s, \varepsilon \eta)$  can be 
carried out by an elementary number theory.
Therefore,
in order to make an algorithm to calculate $\TTTGGG/\mathord{\simAG}$,
it is enough to write a sub-algorithm to calculate $\gamma(g)$ for an arbitrary element $g\in \OGDq{L}$.
For $p\in \PD$,
the $p$-part $g[p] \in \OGDq{L\tensor\Zp}$ of $g$ is easily calculated,
because $D_{L\tensor\Zp}$ is the $p$-part of the finite abelian group $D_L$.
An algorithm to find a lift $g[p]\lift \in \OG(L\tensor\Zp)$ and to calculate its value by $\detspin$
is presented in the next section.
\begin{remark}\label{rem:308}
Let $K\sprime$ denote the subgroup
of $\GammaD\times\Sign$ generated by the subgroups in (i) and 
the elements in  (iii) of Proposition~\ref{prop:quotK}.
If $\dim_{\F_2} (\GammaD\times\Sign)/K\sprime=0$,
then $\TTTGGG/\mathord{\simAG}$ is obviously trivial,
and we do not have to calculate $\Psi_p(g[p])$.
We have $\dim_{\F_2} (\GammaD\times\Sign)/K\sprime>0$ for $319$ algebraic equivalence classes
of connected components.
See Section~\ref{subsec:monodromy} for some cases where $K\ne K\sprime$.
\end{remark}
\begin{remark}\label{rem:complexconj}
The cokernel of the natural $\F_2$-linear homomorphism
$$
K\;\;\inj\;\; \GammaD\times\Sign \xrightarrow{\pr_1} \Gamma_d
$$
calculates the set of connected components modulo complex conjugation.
By this method,
we can show that the two connected components of the moduli
of each type $(\Phi, A)$ in Corollary~\ref{cor:twelve}
are complex conjugate to each other.
\end{remark}
\section{Computation of the homomorphism $\Psi_p$}\label{sec:Psip}
Throughout this section,
we fix a prime $p$,
a non-degenerate $p$-adic finite quadratic form
$\Dq{}$,
and an automorphism $g\in \OGDq{}$.
We assume the following:
\begin{assumption}\label{assump:L}
The finite quadratic form $\Dq{}$ is isomorphic to the discriminant form of an even $\Zp$-lattice $L$
of rank $r$ and discriminant $d$.
(By Proposition~\ref{prop:isomclassZp},
this even $\Zp$-lattice $L$ is unique up to isomorphism.)
\end{assumption}
Our goal is to construct an algorithm to calculate an element of 
$\Gamma_{p}$ that represents $\Psi_p(g)\in \Gamma_{p}/\SigmaSharp(L)$;
that is, an algorithm that finds an isometry $\tilg\in \OG(L)$ inducing $g$ on $\Dq{}$,
and then calculates $(\det(\tilg), \spin(\tilg))\in \Gamma_p$.
%
%
Let $\bp\colon D\times D\to \Q/\Z$ denote the bilinear form associated with $\Dq{}$.
We put
$$
\ell :=\leng(D),
$$
and suppose that
$$
D\;\;\cong\;\; \Z/p^{\nu_1}\Z \times \cdots \times \Z/p^{\nu_{\ell }}\Z.
$$
We fix 
generators $\basisDp_{1}, \dots, \basisDp_{\ell }$  of $D$ such that
$\basisDp_j$ generates the $j$th factor $\Z/p^{\nu_j}\Z$ of $D$.
We denote by $\Matl(R)$ the $R$-module of $\ell \times \ell $ matrices with components in $R$,
where $R$ is $\Z, \Q, \Zp, \Qp, \F_p$, or the localization $\Zpp$ of $\Z$ at the prime ideal $(p)$.
A matrix in $\Matl(R)$ is said to be \emph{even symmetric} if it is symmetric and its diagonal components 
are in $2R$.
We denote by $\ESMat(R)$ the submodule of $\Matl(R)$ 
consisting of even symmetric matrices.
Note that, if $A\in \ESMat(R)$ and $B\in \Matl(R)$, 
then we have $\TMTT{B}{A}\in \ESMat (R)$,
where $\transpose B$ is the transpose of $B$.
\par
Then the quadratic form $q$ on $D$ is expressed by 
$$
\Fq \bmod \ESMat(\Z), 
$$
where $\Fq$ is a symmetric matrix in $\Matl(\Q)$ whose $(i, j)$-component represents
$$
q(\basisDp_i) \in \Q/2\Z \;\;\;\;\textrm{if $i=j$},\qquad 
\bp(\basisDp_i, \basisDp_j)\in \Q/\Z \;\;\;\;\textrm{if $i\ne j$}.
$$
Since $q$ is non-degenerate, we have $\det \Fq\ne 0$.
\par
We denote by $\KerD (R)$ the submodule of $\Matl (R)$ consisting of matrices
whose components in the $j$th column are divisible by $p^{\nu_j}$.
The given automorphism $v\mapsto v^g$ of the finite abelian group $D$
is expressed by
$$
T_0\bmod \KerD(\Z),
$$
where $T_0$ is an element of $\Matl(\Z)$ whose $(i, j)$-components $t_{ij}$ satisfy
$$
\basisDp_{i}^g=\sum_{j=1}^{\ell } t_{ij}\basisDp_{j}.
$$
Since $g$ preserves $q$, we have 
\begin{equation}\label{eq:ToFqtT0}
\TMTT{T_0}{\Fq}\;\equiv\;\Fq\;\bmod \ESMat(\Z).
\end{equation}
Since $g\inv\in \OGDq{}$ exists,
there exists a matrix $T_0^{(-1)} \in \Matl(\Z)$
such that
$$
T_0\cdot T_0^{(-1)} \;\equiv\; I_{\ell}\;\bmod \KerD(\Z).
$$
In particular, the matrix $T_0\bmod p \in \Matl(\F_p)$ is invertible.
\par
Therefore,
the algorithm we are going to construct
is specified as follows:
\begin{itemize}
\item[{\bf Input}]
\begin{enumerate}[(1)]
\item A sequence $[p^{\nu_1}, \dots, p^{\nu_{\ell }}]$ that describes 
the order of each element in a minimal set of generators $\basisDp_1, \dots, \basisDp_{\ell }$ of $D$.
\item A symmetric matrix $\Fq\in \Matl(\Q)$ that represents $q$ with respect to 
$\basisDp_1, \dots, \basisDp_{\ell }$.
\item A matrix $T_{0}\in \Matl(\Z)$ that represents the automorphism $g\in \OG\Dq{}$ with respect to 
$\basisDp_1, \dots, \basisDp_{\ell }$.
\end{enumerate}
\item[{\bf Output}] An element $(\det(\tilg), \spin(\tilg))$ of $\Gamma_p$ that represents $\Psi_p(g)$.
\end{itemize}
\subsection{Step 1}\label{subsec:step1}
%
By the normal form theorem (Proposition~\ref{prop:normalformpfqf}) of non-degenerate $p$-adic 
finite quadratic forms,
there exists an algorithm to calculate an automorphism $v\mapsto v^h$ of $D$ 
represented by $H\bmod \KerD(\Z)$
such that
$\TMTT{H}{\Fq}$ is equivalent modulo $\ESMat(\Z)$ 
to a matrix $F\sprime\in \Matl(\Q)$ in normal form;
that is, $F\sprime$ is a block-diagonal matrix with 
diagonal components being matrices that appear
in Table~\ref{table:indecomppfqf}.
We replace the basis $\basisDp_1, \dots, \basisDp_{\ell }$ of $D$
with the new basis,
and assume that $\Fq$ is in normal form.
Accordingly, we replace the matrix $T_0 $ representing $g\in \OG\Dq{}$
by $H\, T_0\, H^{(-1)}$,
where $H^{(-1)}\in \Matl(\Z)$ is a matrix such that $H^{(-1)} \bmod \KerD(\Z)$ 
represents $h\inv$.
\subsection{Step 2}\label{subsec:step2}
%
We put
$$
\MLam:=\Fq\inv \in \Matl(\Q).
$$
Looking at Table~\ref{table:indecomppfqf},
we see that 
\begin{equation}\label{eq:MLam}
\MLam\in \ESMat(\Zpp) \quad{\rm and}\quad \MLam\equiv O \bmod p,
\end{equation}
where $O$ is the zero matrix.
Let $\Lam$ be an even $\Zp$-lattice of rank $\ell$ 
with a fixed basis $e_1, \dots, e_{\ell}$
whose Gram matrix is $\MLam$.
Then the discriminant form of $\Lam$ is isomorphic to $\Dq{}$.
Recall from Assumption~\ref{assump:L} that $L$ is an even $\Zp$-lattice of
rank $r$, discriminant $d$,
and 
with $\Dq{L}\cong \Dq{}$.
By the normal form theorem of even $\Zp$-lattices~(Proposition~\ref{prop:normalformZplat}),
we obtain the following.
%
%
\begin{itemize}
\item Suppose that $r>\ell $.
Then there exists an even \emph{unimodular} $\Zp$-lattice $\Lam_0$ such that
$L$ is isomorphic to the orthogonal direct-sum $\Lam_0\oplus \Lam$.
(In particular, if $p=2$,  we have $r\equiv \ell \bmod 2$.)
\item Suppose that $r=\ell $ and $p$ is odd.
Then $L$ is isomorphic to $\Lam$.
\item Suppose that $r=\ell $ and $p=2$.
Suppose that $L$ is not isomorphic to $\Lam$.
Then at least one of the matrices on the diagonal 
of $\Fq$ is of the form $[\varepsilon/2]$, where $\varepsilon \in \{1, 3\}$.
We replace one of such components with $[5\varepsilon/2]$,
and re-calculate $\MLam:=\Fq\inv$.
(This change does not affect the class of $\Fq$ modulo $\ESMat(\Z)$
and 
preserves the property~\eqref{eq:MLam}.)
Then $L$ is isomorphic to $\Lam$.
\end{itemize}
Thus we obtain an even $\Zp$-lattice $\Lam$ of rank $\ell $
with the following properties.
\begin{enumerate}[(i)]
\item\label{property:summand}
 The $\Zp$-lattice $\Lam$ is an orthogonal direct summand of $L$.
\item The Gram matrix $M$ of $\Lam$ with respect to a basis $\basisLam_1, \dots, \basisLam_{\ell }$
satisfies the property~\eqref{eq:MLam},
and $M\inv \bmod \ESMat(\Z)$ expresses $\Dq{}$.
More precisely,
let $\basisLamdual_1, \dots, \basisLamdual_{\ell }$ denote  the basis of  $\Lam\dual$
dual to $\basisLam_1, \dots, \basisLam_{\ell }$.
Then the homomorphism $\Lam\dual \to D$ given by $\basisLamdual_i\mapsto\basisDp_i$
induces an isomorphism $\Dq{\Lam}\isom \Dq{}$.
\end{enumerate}
We have a surjective homomorphism $\OG(\Lam)\to \OG\Dq{}$ by Proposition~\ref{prop:etasurjZp}.
By property~\eqref{property:summand} of $\Lam$,
every isometry $\tilde{h}_{\Lam}$ of $\Lam$ can be extended to
an isometry $\tilde{h}_{L}$ of $L$ by letting it act on the orthogonal complement $\Lam\sperp\subset L$ trivially.
Note that $\tilde{h}_{\Lam}$ and $\tilde{h}_{L}$
induce the same automorphism on $\Dq{}$,
and their $\detspin$-values are equal.
Therefore 
it suffices to find an element $\tilg_{\Lam}\in \OG(\Lam)$ 
that induces  the given automorphism $g$ of $\Dq{}$, and 
then to calculate $(\det(\tilg_{\Lam}), \spin(\tilg_{\Lam}))$.
\subsection{Step 3}\label{subsec:step3}
Our next task is to find a sequence
$\tilT_{\nu}$ ($\nu=0,1,\dots$)
of matrices in $\Matl(\Zpp)$ 
converging to a matrix $\tilT\in \Matl(\Zp)$
that represents with respect to the basis $\basisLam_1, \dots, \basisLam_{\ell }$ of $\Lam$
an isometry $\tilg\in \OG(\Lam)$ inducing $g\in \OG\Dq{}$.
\begin{lemma}\label{lem:YM}
We have
$$
\KerD(\Zp)=\set{Y\MLam}{Y\in \Matl(\Zp)}.
$$
\end{lemma}
\begin{proof}
Let $v=(p^{\nu_1} a_1, \dots, p^{\nu_{\ell }} a_{\ell })$
be a row vector of a matrix belonging to $\KerD(\Zp)$.
Since $p^{\nu_i}\basisDp_i=0$ 
holds for $i=1,\dots, \ell$, the mapping $x\mapsto v \Fq \transpose x$
expresses the homomorphism $D\to \Qp/\Zp$ given by $x\mapsto \bp(0, x)$.
Therefore $v\Fq=v\MLam\inv$ has components in $\Zp$.
Conversely, note that the $i$-th row vector $(m_{i1}, \dots, m_{i\ell })$ of $\MLam$
is a vector representation of $\basisLam_i\in \Lam$
with respect to the dual basis $\basisLamdual_{1}, \dots, \basisLamdual_{\ell }$ of $\Lam\dual$.
Therefore $m_{i1}\basisDp_1+ \dots + m_{i \ell } \basisDp_{\ell }=0$ holds in $D$,
and hence $m_{ij}\basisDp_j=0$ holds for $j=1, \dots, \ell$.
In particular, we have $M\in \KerD(\Zp)$.
\end{proof}
The Gram matrix of 
$\Lam\dual$
with respect to the basis $\basisLamdual_1, \dots, \basisLamdual_{\ell }$ is
$\MLam\inv=\Fq$.
Recall that $T_0\in \Matl(\Z)$ is a matrix such that $T_0\bmod \KerD(\Z)$ represents 
 $g\in \OG\Dq{}$.
\begin{lemma}\label{lem:Tequiv}
For $\tilT\in \Matl(\Qp)$,
the following conditions are equivalent.
\begin{enumerate}[{\rm (i)}]
\item The matrix $\tilT$ represents 
with respect to the basis $\basisLam_1, \dots, \basisLam_{\ell }$ of $\Lam$
an isometry $\tilg\in \OG(\Lam)$ that induces $g\in \OG\Dq{}$.
\item
We put
$$
T:=\MLam\inv\tilT \MLam.
$$
Then $(T-T_0)\MLam\inv\in \Matl(\Zp)$ and $\TMTT{T}{\MLamdual}=\MLamdual$ hold.
\end{enumerate}
\end{lemma}
\begin{proof} 
Suppose that $\tilT$ satisfies condition (i).
The isometry $\tilg$
induces an isometry of $\Lam\dual$,
and $T=\MLam\inv\tilT \MLam$ is the matrix representation 
of this isometry 
with respect to $\basisLamdual_1, \dots, \basisLamdual_{\ell }$.
Hence we have $\TMTT{T}{\MLamdual}=\MLamdual$.
Since $T$ induces $g$ on $D$
and the identification $D=\Lam\dual/\Lam$ is given by 
$\basisLamdual_i\mapsto\basisDp_i$,
we have
$$
T \equiv T_0 \bmod \KerD(\Zp).
$$
By Lemma~\ref{lem:YM},
we have $(T-T_0)\MLam\inv \in \Matl(\Zp)$.
\par
Conversely,
suppose that $\tilT$ satisfies condition (ii).
Then $\tilT=\MLam T\MLam\inv$ satisfies
$\TMTT{\tilT}{ \MLam}=\MLam$.
We show that $\tilT$ has components in $\Zp$.
As was seen above,
 $T_0\bmod p\in \Matl(\F_p)$ is invertible, and hence we have $\transpose T_0\inv \in \Matl(\Zp)$.
Since $g$ preserves $q$ and $\MLamdual$ is equal to $\Fq$, we see from~\eqref{eq:ToFqtT0} that
\begin{equation}\label{eq:E0}
E_0:=\TMTT{T_0}{\MLamdual}-\MLamdual \;\in\; \ESMat(\Zpp)\subset \ESMat(\Zp).
\end{equation}
In particular,
we have
$$
\MLam T_0 \MLam\inv=(I_{\ell } +\MLam E_0) \transpose T_0\inv \in \Matl(\Zp).
$$
By the assumption, 
we have $T=T_0+Y\MLam$
for some $Y\in \Matl(\Zp)$.
Therefore $\tilT=\MLam T_0 \MLam\inv+\MLam Y$
 has components in $\Zp$.
Hence $\tilT$ is a matrix representation of an isometry of $\Lam$.
Since $T=T_0+Y\MLam$,
 this isometry induces $g$ on $D=\Lam\dual/\Lam$.
\end{proof}
We denote by $\Matlp{p}$ the set of square matrices
of size $\ell $
whose components are in $\{0,\dots, p-1\}\subset\Z$.
By the surjectivity of $\OG(\Lam)\to\OG\Dq{}$ and Lemma~\ref{lem:Tequiv},
 there exists a sequence
 $Z_0, Z_1, Z_2, \dots$ of elements of $\Matlp{p}$
 such that the matrix 
 $$
 T:=T_0+YM, \;\; \textrm{where}\;\; Y:=Z_0+ pZ_1+p^2Z_2+\cdots \in \Matl(\Zp), 
 $$
 satisfies
 $$
 \TMTT{T}{\MLamdual}=\MLamdual.
 $$
 Let $Z_0, Z_1, Z_2, \dots$ be such a sequence.
 For $\nu>0$,
 we put
 \begin{eqnarray*}
 Y_{\nu-1} &:=& Z_0+ pZ_1+\cdots+p^{\nu-1} Z_{\nu-1}, \\
 T_{\nu} &:=& T_0+ Y_{\nu-1} \MLam.
 \end{eqnarray*}
Since $\MLam\equiv O \bmod p$, we have $T\equiv T_{\nu}\equiv T_0\bmod p$.
Then, for $\nu\ge 0$,
we have
\begin{equation}\label{eq:Tnu}
\TMTT{T_{\nu}}{\MLamdual}=\MLamdual+ p^{\nu} E_{\nu}\quad\textrm{for some $E_{\nu}\in \ESMat(\Zpp)$}.
\end{equation}
Indeed, 
since 
 $T_{\nu}=T+p^\nu\, WM$ for some $W\in \Matl(\Zp)$,
we have 
$$
\TMTT{T_{\nu}}{\MLamdual}-\MLamdual=p^{\nu}(W \transpose T_{\nu} +T_{\nu} \transpose W + p^\nu\; \TMTT{W}{M}).
$$
By~\eqref{eq:MLam}, we see that 
$$
E_{\nu}:=W \transpose T_{\nu} +T_{\nu} \transpose W + p^\nu\; \TMTT{W}{M}\;\;\in\;\; \ESMat(\Zp).
$$
Since $E_{\nu}=p^{-\nu}(\TMTT{T_{\nu}}{\MLamdual}-\MLamdual)$ has components in $\Q$,
we have $E_{\nu}\in \ESMat(\Zpp)$.
\par
We calculate such a  sequence $Z_{\nu}$ ($\nu=0, 1, \cdots$)
inductively on $\nu$.
Suppose that we have obtained $T_{\nu}\in \Matl(\Zpp)$ satisfying~\eqref{eq:Tnu}
and 
\begin{equation}\label{eq:TnuT0}
T_{\nu}\equiv T_0\bmod p.
\end{equation}
(By~\eqref{eq:E0}, 
 we can use the input data $T_0$ for $\nu=0$.) 
Our task is to  search for $Z_{\nu}\in \Matlp{p}$ such that
 $T_{\nu+1}:= T_{\nu} + p^{\nu} Z_{\nu} \MLam$ satisfies~\eqref{eq:Tnu} with $\nu$ replaced by $\nu+1$.
 Since
 $$
 \TMTT{T_{\nu+1}}{\MLamdual}-\MLamdual=
 p^\nu (E_{\nu} + Z_{\nu} \transpose T_{\nu} + T_{\nu}\transpose Z_{\nu} + p^{\nu}\; \TMTT{Z_{\nu}}{M}),
 $$
 it is enough to find a matrix $X\in \Matlp{p}$
 that satisfies 
\begin{equation}\label{eq:inDelta}
\frac{1}{p} (E_{\nu} + X \transpose T_{\nu} + T_{\nu}\transpose X) +p^{\nu-1}\, \TMTT{X}{\MLam} \; \in \; \ESMat(\Zpp).
\end{equation}
 \subsubsection{Suppose that $p>2$}
 Then every symmetric matrix in $\Matl(\Zpp)$ is even symmetric.
 Since $\MLam\equiv O\bmod p$, we see that $p^{\nu-1}\, \TMTT{X}{\MLam}$
 is a symmetric matrix in $\Matl(\Zpp)$ for any $X\in \Matlp{p}$
 even when $\nu=0$.
It is obvious that $E_{\nu} + X \transpose T_{\nu} + T_{\nu}\transpose X$ is symmetric for any $X\in \Matlp{p}$.
Therefore, combining 
this with~\eqref{eq:TnuT0}, we see that the condition~\eqref{eq:inDelta} is equivalent to the affine linear equation
\begin{equation}\label{eq:affeqpodd}
E_{\nu} + X \transpose T_{0} + T_{0}\transpose X\;\equiv\; O\;\;\bmod p
\end{equation}
over $\F_p$.
We solve~\eqref{eq:affeqpodd}
and lift a solution in $\Matl(\F_p)$ to $Z_{\nu}\in \Matlp{p}$.
\subsubsection{Suppose that $p=2$}
We put
\begin{eqnarray*}
h_{ii} &:=&\frac{1}{2} (\,\textrm{the $(i, i)$-component of $E_{\nu}$}\,) \;\; \bmod 2, \\
f_{ii}(X) &:=& \textrm{the $(i, i)$-component of $X\transpose T_0$} \;\; \bmod 2,
\end{eqnarray*}
for $i=1, \dots, \ell $.
Note that, since $E_{\nu}\in \ESMat(\Z_{(2)})$, the definition of $h_{ii}\in \F_2$ above
makes sense.
\par
Suppose that $\nu>0$.
Then we have 
$2^{\nu-1}\cdot \TMTT{X}{\MLam}\in \ESMat(\Z_{(2)})$ for any $X\in \Matlp{2}$.
Therefore, by~\eqref{eq:TnuT0}, 
we see that the condition~\eqref{eq:inDelta} is equivalent to the affine linear equation
\begin{equation}\label{eq:affeqp2nupos}
\begin{cases}
E_{\nu} + X \transpose T_{0} + T_{0}\transpose X\;\equiv\; O\;\;\bmod 2, \\
h_{ii}+f_{ii}(X) \;\equiv\; 0\;\;\bmod 2\;\;\; (i=1, \dots, \ell )
\end{cases}
\end{equation}
over $\F_2$.
We solve~\eqref{eq:affeqp2nupos}
and lift a solution in $\Matl(\F_2)$ to $Z_{\nu}\in \Matlp{2}$.
\par
Suppose that $\nu=0$.
Then $2\inv\,\TMTT{X}{M}$ is symmetric with components in $\Z_{(2)}$, 
but some of its diagonal components may fail to be even.
Hence we put
$$
g_{ii}(X):=\frac{1}{2} (\,\textrm{the $(i, i)$-component of $\TMTT{X}{M}$}\,)\;\; \bmod 2,
$$
which is a homogeneous quadratic polynomial over $\F_2$ of the components of $X=(x_{ij})$.
Note that, since $M\equiv O\bmod 2$,
the definition of $g_{ii}(X)$ makes sense.
Recall that $M=\Fq\inv$ is block-diagonal
with diagonal components
$$
W_{\mu, \varepsilon}:=\left[\frac{2^\mu}{\varepsilon}\right]\;\;\; (\varepsilon\in \{1,3,5,7\}),
\quad
U_{\mu}:=2^\mu \left[\begin{array}{cc} 0 & 1 \\ 1 & 0\end{array}\right],
\;\;\textrm{or}\;\;\;
V_{\mu}:=\frac{2^\mu}{3} \left[\begin{array}{cc} 2 & -1 \\ -1 & 2\end{array}\right],
$$
where $\mu>0$.
Note that the quadratic forms
$$
[x, y]\;U_{\mu}\left[\begin{array}{c} x \\ y\end{array}\right]=2^{\mu+1}xy,
\quad
[x, y]\;V_{\mu}\left[\begin{array}{c} x \\ y\end{array}\right]=\frac{2^{\mu+1}}{3}(x^2-xy+y^2)
$$
are always divisible by $4$ in $\Z_{(2)}$.
We put
$$
J:=\set{j}{\textrm{the $(j, j)$-component of $M$ is $2/\varepsilon_{j}$}}.
$$
Since $\varepsilon_j\equiv 1\bmod 2$,
the quadratic polynomial $g_{ii}(X)$ is of the form
$$
\sum_{j\in J} x_{ij}^2/\varepsilon_{j}=\sum_{j\in J} x_{ij}^2.
$$
Since $x^2=x$ in $\F_2$,
the equation $g_{ii}(X)=b$ over $\F_2$ 
with $b\in \F_2$ is equivalent to the affine linear equation $\bar{g}_{ii}(X)=b$,
where 
$$
\bar{g}_{ii}(X):=\sum_{j\in J} x_{ij}.
$$
Therefore, by~\eqref{eq:TnuT0}, we see that the condition~\eqref{eq:inDelta} is equivalent to the affine linear equation
\begin{equation}\label{eq:affeqp2nu0}
\begin{cases}
E_{\nu} + X \transpose T_{\nu} + T_{\nu}\transpose X\;\equiv\; O\;\;\bmod 2, \\
h_{ii}+f_{ii}(X)+\bar{g}_{ii}(X) \;\equiv\; 0\;\;\bmod 2\;\;\; (i=1, \dots, \ell )
\end{cases}
\end{equation}
over $\F_2$.
We solve~\eqref{eq:affeqp2nu0}
and lift a solution in $\Matl(\F_2)$ to $Z_{\nu}\in \Matlp{2}$.
\begin{remark}
The fact that the equations~\eqref{eq:affeqpodd} and~\eqref{eq:affeqp2nupos}
always have solutions in $\Matl(\F_p)$ is easily proved from $\det T_0\not \equiv 0 \bmod p$.
For example, when $p>2$,
the image of the linear map $\Matl(\F_p)\to \Matl(\F_p)$
given by $X\mapsto X\transpose{T_0} +T_0\transpose{X}$
 is equal to $\ESMat(\F_p)$.
The fact that the equation~\eqref{eq:affeqp2nu0}
is always soluble in $\Matl(\F_2)$ is non-trivial;
it is a consequence of the surjectivity of $\OG(\Lam)\to\OGDq{\Lam}$. 
\end{remark}
For an element $a\in \Qp\sptimes$, let $\ordp(a)$ denote the maximal integer $n$ such that $p^{-n} a\in \Zp$.
We put $\ordp(0):=\infty$.
For a matrix $M=(m_{ij})$ with components in $\Zp$,
we put
$$
\minordp(M):=\textrm{the minimum of $\ordp(m_{ij})$}.
$$
We define $\minordp(v)$ for a vector $v$ with components in $\Zp$
in the same way.
By the argument above, 
we have proved the following:
\begin{proposition}
For an arbitrarily large integer $\nu$,
we can calculate a matrix $T_{\nu}\in \Matl(\Zpp)$
such that there exists a matrix $T\in \Matl(\Zp)$
with the following properties:
\begin{enumerate}[{\rm (i)}]
\item $\minordp(T-T_{\nu})\ge \nu$,
\item $\MLam T \MLam\inv$ represents an isometry $\tilg$ of $\Lam$
with respect to $\basisLam_1, \dots, \basisLam_{\ell }$, and
\item $\tilg$ induces the given automorphism $g$ on $\Dq{}$.
\end{enumerate}
\end{proposition}
\subsection{Step 4}\label{subsec:step4}
Let $\Lam$ and $\tilg\in \OG(\Lam)$ be as in Step 3.
Let $\nu$ be a sufficiently large integer.
We put 
$$
V:=\Lam\tensor\Qp=\Lam\dual\tensor\Qp.
$$
In Step 3,
we have calculated a matrix
$$
\aT :=T_{\nu}\in \Matl(\Zpp)
$$
that is approximate to the matrix $T\in \Matl(\Zp)$
representing
 $\tilg\in \OG(V)$ with respect to the basis $\basisLamdual_1, \dots, \basisLamdual_{\ell }$ of $V$.
The approximate accuracy $\minordp(T-\aT)$ of $\aT$ satisfies
$$
\minordp(T-\aT)\ge \nu.
$$
In order to calculate $(\det(\tilg), \spin(\tilg))$,
we present an algorithm to decompose $\tilg$ into a product of reflections in $\OG(V)$
using only the computed matrix $\aT$.
This algorithm works when $\nu$ is sufficiently large.
\begin{remark}\label{rem:recal}
It is possible to state explicitly how large $\nu$ should be
for the algorithm to work.
However, the result would be complicated,
and, for most practical applications,
the theoretical bound seems to be unnecessarily large.
Therefore we present an algorithm of the style that,
if it fails to continue at some point because $\nu$ is not large enough,
it quits, goes back to Step 3,
re-calculate an approximate matrix $\aT$ with higher accuracy $\nu$,
and re-start the algorithm from the beginning.
If this algorithm reaches the end,
the result $(\det(\tilg), \spin(\tilg))$ is correct.
\end{remark}
Note that the Gram matrix $\MLam\inv$ of $V$
with respect to $\basisLamdual_1, \dots, \basisLamdual_{\ell }$
has components in $\Q$.
Hence we can find an orthogonal basis $f_1, \dots, f_{\ell }$ of $V$
by the Gram-Schmidt orthogonalization \emph{in $\Q$};
that is, we can calculate an invertible matrix $S\in \Matl(\Q)$
of basis transformation
such that the new Gram matrix 
$$
\MV:=\TMTT{S}{\MLam\inv}
$$
with respect to the new basis $f_1, \dots, f_{\ell }$
is diagonal.
We replace $T$ and $\aT$ by
$S\,T\,S\inv$
and 
$S\,\aT \,S\inv$,
respectively, 
so that $T$ represents $\tilg$ with respect to $f_1, \dots, f_{\ell }$.
The lower bound $\nu$ of the approximate accuracy $\minordp(T-\aT)$
is replaced by
$$
\nu +\min(0, \minordp(S))+\min(0, \minordp(S\inv)).
$$
(See Lemma~\ref{lem:AB} below.)
\par
To simplify the notation, 
we fix this orthogonal basis $f_1, \dots, f_{\ell }$ of $V$
in the rest of this section.
We identify vectors in $V$ with row vectors in $\Qp^\ell$, 
and linear transformations of $V$ with matrices in $\Matl(\Qp)$.
In particular,
if $A\in \Matl(\Qp)$
represents $a\in \OG(V)$,
we write $vA$ instead of $v^a$ for $v\in V$.
For $v\in \Qp^\ell$ and $A\in \Matl(\Qp)$,
we use the notation $\apprx{\,v}\in \Q^\ell$ and $\apprx{A\in \Matl(\Q)}$
to denote computed objects that we intend to be approximate values of $v$ and $A$,
respectively.
\par
For $k$ with $0\le k\le \ell$, 
let $\gen{f_1, \dots, f_k}$ denote the subspace of $V$ generated by $f_1, \dots, f_k$.
Then the orthogonal complement of $\gen{f_1, \dots, f_k}$ in $V$
is $\gen{f_{k+1}, \dots, f_{\ell}}$.
We put
$$
\gamma_i:=\ordp(\intf{f_i, f_i}), \quad \gamma:=\minordp(\MV)=\min(\gamma_1, \dots, \gamma_{\ell}).
$$
The following lemma is easy to prove, and will be used frequently.
\begin{lemma}\label{lem:AB}
{\rm (1)}
Let $A, B$ be matrices in $\Matl(\Qp)$,
and suppose that $\apprx{A}, \apprx{B}\in \Matl(\Q)$ satisfy
$$
\minordp(A-\apprx{A})\ge\alpha,\;\; \minordp(B-\apprx{B})\ge \beta.
$$
Then we have
$$
\minordp(AB-\apprx{A}\apprx{B})\ge  \min (\; \minordp(\apprx{A})+\beta, \; \alpha+\minordp(\apprx{B}), \; \alpha+\beta \;).
$$
{\rm (2)}
For any $u, v\in \Qp^\ell$, we have 
\begin{equation*}\label{eq:ordpuv}
\ordp(\intf{u, v})\;\ge\; \gamma+\minordp(u)+\minordp(v).
\end{equation*}
\end{lemma}
We also put
$$
\delta:=\begin{cases}
1 & \textrm{if $p=2$}, \\
0 & \textrm{if $p>2$}.
\end{cases}
$$
The following lemma is also easy to prove.
\begin{lemma}\label{lem:abinQp}
Let $a$ and $b$ be elements of the multiplicative group $\Qp\sptimes$.
\par
{\rm (1)}
If $\ordp(a)+\delta<\ordp(b-a)$,
then we have $\ordp(a+b)=\ordp(a)+\delta$.
\par
{\rm (2)}
 We have $a\equiv b \bmod (\Qp\sptimes)^2$ if 
$\ordp(1-a/b)\ge 1+2\delta$. 
\end{lemma}

Our algorithm proceeds as follows.
We start from 
$$
T\spar{0}:=T,
\quad
\aT\spar{0}:=\aT,
\quad
\nu_0:=\nu,
\quad
i(0):=0.
$$
By the induction on $k$ up to $k=\ell$,
we compute a matrix $\aT\spar{k}\in \Matl(\Q)$,
an integer $\nu_k$,
and a sequence $\ar_j$
of vectors in $\Q^\ell$ 
for $j=i(k-1)+1, \dots, i(k)$ with the following properties:
\begin{enumerate}[{\rm (P1)}]
\item
For $j$ with $i(k-1)<j\le i(k)$,
we have $\sqnorm{\ar_j}\ne 0$.
In particular, we have 
the reflection $\tau(\ar_j)\in \OG(V)$.
\item 
The vector $\ar_j$ is approximate to a vector $r_j\in \gen{f_1, \dots, f_k} \subset \Qp^\ell$
with accuracy high enough to ensure that $\sqnorm{r_j}\ne 0$
and $\sqnorm{r_j}\equiv \sqnorm{\ar_j} \bmod (\Qp\sptimes)^2$ holds
in the multiplicative group $\Qp\sptimes$.
In particular, we have the reflection $\tau(r_j)\in \OG(V)$.
\item
The isometry
$$
T\spar{k}:=T\spar{0} \tau(r_1)\cdots \tau(r_{i(k)})=T\spar{k-1} \tau(r_{i(k-1)+1})\cdots \tau(r_{i(k)})
$$
preserves the subspace $\gen{f_1, \dots, f_k}$ of $V$,
and acts trivially on $\gen{f_1, \dots, f_k}$.
\item
The matrix 
$$
\aT\spar{k}:=\aT\spar{0} \tau(\ar_1)\cdots \tau(\ar_{i(k)})=\aT\spar{k-1} \tau(\ar_{i(k-1)+1})\cdots \tau(\ar_{i(k)}).
$$
is approximate to $T\spar{k}$, and 
we have 
$\minordp(T\spar{k}-\aT\spar{k})\ge \nu_k$.
\end{enumerate}
Suppose that we reach $k=\ell$.
Then $T^{(\ell)}$ is the identity matrix by property (3),
and hence we have
$$
T= \tau(r_{i(\ell)})\cdots \tau(r_{1}).
$$
Therefore we have $\det(T)=(-1)^{i(\ell)}$ and
$$
\spin(T)=Q(r_{i(\ell)})\cdots Q(r_{1})\bmod (\Qp\sptimes)^2
=Q(\ar_{i(\ell)})\cdots Q(\ar_{1}) \bmod (\Qp\sptimes)^2,
$$
where the second equality follows from property (2).
Since $\ar_{1}, \dots, \ar_{i(\ell)}$ 
are computed, $(\det(\tilg), \spin(\tilg))$ is also computed.
\par
Suppose that we have calculated $\aT\spar{k-1}$, $\nu_{k-1}$ and $\ar_{1}, \dots, \ar_{i(k-1)}$.
Recall that $\fk$ is the vector $(0, \dots, 1, \dots, 0)\in \Q^\ell$,
where $1$ is at the $k$th position.
We put
$$
\gk:=\fk\, T\spar{k-1}, \quad
\agk:=\fk\, \aT\spar{k-1}.
$$
By the induction hypothesis,
the isometry $T\spar{k-1}$  of $V$ preserves the subspace $\gen{f_1, \dots, f_{k-1}}$,
and hence preserves $\gen{f_1, \dots, f_{k-1}}\sperp=\gen{f_k, \dots, f_{\ell}}$.
Therefore we have $\gk\in \gen{f_k, \dots, f_{\ell}}$.
Since $T\spar{k-1}$ is an isometry,
we have 
\begin{equation}\label{eq:gkfk}
\sqnorm{\gk}=\sqnorm{\fk}, 
\end{equation}
and hence we have 
$\ordp(\sqnorm{\gk})=\ordp(\sqnorm{\fk})=\gamma_k$. 
We estimate the approximation error $\sqnorm{\gk}-\sqnorm{\agk}$.
For this purpose, we put
$$
\lambda_k:=\minordp(\agk),
\;\; 
\rho:=\min(\;\delta+\nnu+\gamma+\lambda_k, \; 2\nnu+\gamma\;).
$$
By property (4) for $\nnu$,
we have a matrix $A\in \Matl(\Zp)$
such that
$T\spar{k-1}=\aT\spar{k-1}+\pnnu A$.
Hence we have a vector $v\in \Zp^\ell$ such that
$$
\gk=\agk+\pnnu v.
$$
Therefore we have
$$
\sqnorm{\gk}-\sqnorm{\agk} =2\intf{\agk, \pnnu v}+\sqnorm{\,\pnnu v}.
$$
From $\ordp(\intf{\agk, \pnnu v})\ge \lambda_k+\gamma+\nnu$ and $\ordp(\sqnorm{\,\pnnu v})\ge 2\nnu+\gamma$,
we obtain 
$$
\ordp(\,\sqnorm{\gk}-\sqnorm{\agk}\,)\ge \rho.
$$
If $\rho\le \gamma_k+\delta$, then we quit and go to the re-calculation process (Remark~\ref{rem:recal}).
Suppose that $\rho>\gamma_k+\delta$.
Then, by Lemma~\ref{lem:abinQp},
we have
\begin{equation}\label{eq:gsum}
\ordp(\,\sqnorm{\gk}+\sqnorm{\agk}\,)=\gamma_k+\delta.
\end{equation}
We put
$$
b^+:=\fk+\gk, \;\; \ab^+:=\fk+\agk, \;\; b^-:=\fk-\gk, \;\; \ab^-:=\fk-\agk.
$$
We have
$$
\sqnorm{\ab^+}+\sqnorm{\ab^-}=2(\sqnorm{\fk}+\sqnorm{\agk}).
$$
By~\eqref{eq:gkfk} and~\eqref{eq:gsum}, 
we see that the $\ordp$ of at least one of $\sqnorm{\ab^+}$ or $\sqnorm{\ab^-}$ is $\le \gamma_k+2\delta$.
If $\ordp(\sqnorm{\ab^-})\le \gamma_k+2\delta$, we put $b:=b^-$ and $\ab:=\ab^-$;
otherwise, we put $b:=b^+$ and $\ab:=\ab^+$.
Note that we have $b\in \gen{\fk, \dots, f_{\ell}}$.
Then we have
$$
b=\ab+\pnnu w, 
$$
where $w=\pm v\in \Zp^{\ell}$.
In order to estimate $\sqnorm{b}-\sqnorm{\ab}$,
we put
$$
\sigma:=\min(\,\delta+\nnu+\gamma, \;\; \delta+\nnu+\gamma+\lambda_k, \;\;2\nnu+\gamma\,).
$$
Since 
$$
\sqnorm{b}-\sqnorm{\ab}=2\intf{\fk, \pnnu w}\pm 2\intf{\agk, \pnnu w}+\sqnorm{\,\pnnu w},
$$
we have 
$$
\ordp(\, \sqnorm{b}-\sqnorm{\ab}\,)\ge \sigma.
$$
We then put
$$
\kappa:=\sigma-( \gamma_k+2\delta).
$$
Since $\ordp(\sqnorm{\ab})\le \gamma_k+2\delta$,
we see that
$$
\sqnorm{b}=\sqnorm{\ab} (1+p^\kappa c)\;\;\textrm{for some}\;\; c\in \Zp.
$$
If $\kappa <1+2\delta$, then we quit and go to the re-calculation process (Remark~\ref{rem:recal}).
Suppose that
$\kappa\ge 1+2\delta$.
Then, by Lemma~\ref{lem:abinQp},
we have
$$
\sqnorm{b}\equiv \sqnorm{\ab} \bmod (\Qp\sptimes)^2.
$$
When $b=b^-$, we put
$$
i(k):=i(k-1)+1, \;\; 
r_{i(k-1)+1}:=b, \;\;
\ar_{i(k-1)+1}:=\ab, 
$$
so that
$$
T\spar{k}=T\spar{k-1} \tau(b), \;\; \aT\spar{k}=\aT\spar{k-1} \tau(\ab).
$$
When $b=b^+$, we put
$$
i(k):=i(k-1)+2, \;\;
r_{i(k-1)+1}:=b, \;\;
\ar_{i(k-1)+1}:=\ab, \;\;
r_{i(k-1)+2}:=\ar_{i(k-1)+2}:=\fk, 
$$
so that
$$
T\spar{k}=T\spar{k-1} \tau(b)\tau(\fk), \;\; \aT\spar{k}=\aT\spar{k-1} \tau(\ab)\tau(\fk).
$$
By construction, we have $\fk \,T\spar{k}=\fk$.
Using the induction hypothesis on $T\spar{k-1}$,
we can easily verify that $T\spar{k}$ preserves $\gen{f_1, \dots, f_k}$
and acts on $\gen{f_1, \dots, f_k}$ trivially.
Thus the constructed data satisfies the properties (P1), (P2), (P3). 
\par
It remains to give a lower bound $\nu_k$ of $\minordp(T\spar{k}-\aT\spar{k})$.
First we calculate $\minordp(\tau(b)-\tau(\ab))$.
For $x\in \Qp^\ell$, we have 
$$
x\cdot {(\tau(b)-\tau(\ab))}=\frac{2}{\sqnorm{b}} \phi(x),
$$
where
\begin{eqnarray*}
\phi(x)&=&(1+p^\kappa c)\,\intf{\ab, x}\,\ab-\intf{\ab+\pnnu w, x}\,(\ab+\pnnu w) \\
&=& p^\kappa c \,\intf{\ab, x}\,\ab -\intf{\,\pnnu w, x}\,\ab -\intf{\ab, x}\,\pnnu w-\intf{\,\pnnu w, x}\,\pnnu w.
\end{eqnarray*}
Since $\minordp(\fk)=0$,
we have $\minordp(\ab)\ge \bar{\lambda}_k:=\min (0, \lambda_k)$.
Hence, whenever $\minordp(x)\ge 0$, we have
$$
\minordp(\phi(x))\ge \theta:=\min(\; \kappa+2\bar{\lambda}_k+\gamma,\; \nnu+\gamma+\bar{\lambda}_k,\;2\nnu+\gamma\;).
$$
Combining this with $\ordp(\sqnorm{b})\le \gamma_k+2\delta$,
we see that
\begin{equation}\label{eq:difftau}
\minordp(\tau(b)-\tau(\ab))\ge \delta+\theta-(\gamma_k+2\delta)=\theta-\gamma_k-\delta.
\end{equation}
We put
$$
\lambda:=\minordp(\aT\spar{k-1}),
\quad
\alpha:=\minordp(\tau(\ab)),
\quad
\beta:=\minordp(\tau(\fk)),
$$
and
$$
\nu\sprime:=\min(\;\nnu +\alpha, \; \lambda+\theta-\gamma_k-\delta, \; \nnu+\theta-\gamma_k-\delta\; ).
$$
By Lemma~\ref{lem:AB} and~\eqref{eq:difftau}, we see that
$$
\minordp(\,T\spar{k-1}\tau(b)-\aT\spar{k-1} \tau(\ab)\,)\ge \nu\sprime.
$$
Therefore, in the case where $b=b^-$, we put $\nu_k:=\nu\sprime$.
In the case where $b=b^+$,
we have
$$
\minordp(\,T\spar{k-1}\tau(b)\tau(\fk)-\aT\spar{k-1} \tau(\ab)\tau(\fk)\,)\ge \nu\sprime+\beta, 
$$
and hence
we put $\nu_k:=\nu\sprime+\beta$.
\par
The values of $\gamma_k, \gamma$ and $\beta$
do not depend on the initial approximate accuracy $\nu_0=\nu$.
The values of $\lambda_k, \lambda$ and $\alpha$
stabilize to constants when $\nu$ goes to infinity.
Suppose that $\nnu/\nu$ converges to $1$
when $\nu$ goes to infinity.
By definitions,
we see that $\sigma/\nnu$, $\kappa/\nnu$ and $\theta/\nnu$
also converge to $1$,
and hence $\nu_k/\nu$ converges to $1$.
Therefore, if $\nu$ is large enough,
this algorithm reaches $k=\ell$.
\section{Examples}\label{sec:examples}
\subsection{Algebraically distinguished connected components}\label{subsec:algdist}
Let $(X, f, s)$ be an elliptic $K3$ surface.
We use the notation $A_f$, $U_f, \Lgen{\Phi_f}$,  and $\barLgen{\Phi_f}$
defined in Introduction.
Since we can perturb $(X, f, s)$ to an elliptic $K3$ surface $(X\sprime, f\sprime, s\sprime)$
in such a way that $A_{f\sprime}\cong A_{f}$ and $S_{X\sprime}\cong U_f\oplus \barLgen{\Phi_f}$, 
we see that, 
for each torsion section $\tau\in A_f$,
the class $[\tau]\in H^2(X, \Z)$
of the curve $\tau(\P^1)$ is contained in $U_f\oplus \barLgen{\Phi_f}$.
In this subsection,
we present a method to calculate these classes $[\tau]$.
\par
We denote by $\vf\in U_f$ the class of a fiber of $f$.
Let $P\in \P^1$ be a point such that $f\inv (P)$ is reducible.
Suppose that the reduced part of $f\inv (P)$ consists of $\rho+1$ smooth rational curves.
A smooth rational curve $\Theta$ in $f\inv(P)$
is said to be \emph{a simple component of $f\inv(P)$}
if the divisor $f\inv(P)$ of $X$ is reduced at a general point of $\Theta$.
If a section of $f$ intersects $f\inv(P)$
at a point of $\Theta$, 
then $\Theta$ is a simple component.
Let $\Theta_0$ be a simple component of $f\inv (P)$,
and let $\Theta_1, \dots, \Theta_{\rho}$
be the other smooth rational curves in $f\inv(P)$.
Let $\theta_{\nu}$ be the class of $\Theta_{\nu}$
for $\nu=0, \dots, \rho$.
Then $\theta_1, \dots, \theta_{\rho}$
span a root sublattice
$L(P)$ in $U_f\oplus L(\Phi_f)$,
and $\theta_1, \dots, \theta_{\rho}$ form a fundamental root system of $L(P)$.
Moreover, $\vf$ is orthogonal to $L(P)$, and $\theta_0\in \Z \vf \oplus L(P)$.
\begin{proposition}\label{prop:torsionsection}
Let $\{P_1, \dots, P_N\}$
be the set of points $P_i\in \P^1$
such that $f\inv (P_i)$ is reducible.
A vector $u\in U_f \oplus M(\Phi_f)$
is the class $[\tau]$ of a torsion section
$\tau\in A_f$ if and only if $u$ satisfies the following:
\begin{enumerate}[{\rm (i)}]
\item
$\intf{u, u}=-2$ and $\intf{u, \vf}=1$.
\item
For each $i=1, \dots, N$,
there exists a simple component $\Theta\spari_0$ of $f\inv (P_i)$
such that $\intf{u, \Theta\spari_0}=1$,  and that  $\intf{u, \Theta\spari_{\nu}}=0$ holds for all smooth rational 
curves $\Theta\spari_{\nu}$ in $f\inv (P_i)$ other than $\Theta\spari_0$.
\end{enumerate}
\end{proposition}
For the proof, we need a preparation.
Let $P\in \P^1$, $L(P)$,
$\Theta_0, \dots, \Theta_{\rho}$, and $\theta_0, \dots, \theta_{\rho}$ be as above.
We have
$$
\theta_0=\vf-\sum_{i=1}^{\rho} m_{\nu} \theta_{\nu},
$$
where $m_{\nu}\in \Z_{>0}$ is the multiplicity of $\Theta_{\nu}$
in the divisor $f\inv (P)$.
The values of $m_{\nu}$ are classically known for all types of singular fibers of elliptic surfaces.
(See~\cite{MR0184257}. See also~\cite[Figure 1.8]{MR2977354}.)
The following lemma can be confirmed by explicit computation.
\begin{lemma}\label{lem:theta}
Let $\theta_1\dual, \dots, \theta_{\rho}\dual$ be the basis of $L(P)\dual$
dual to the fundamental root system $\theta_1, \dots, \theta_{\rho}$ of $L(P)$.
Then there exists no index $\mu>0$ such that $m_{\mu}=1$ and $\theta_{\mu}\dual\in L(P)$.
\end{lemma}
\begin{proof}[Proof of Proposition~\ref{prop:torsionsection}]
The necessity of conditions (i) and (ii) is obvious.
Suppose that $u$ satisfies (i) and (ii).
By condition (i), we see that $u$ is the class of an effective divisor
$$
H+\sum_{i=0}^N \varGamma_i
$$
on $X$, 
where $H$ is a reduced curve mapped isomorphically to $\P^1$ by $f$,
and $\varGamma_i$ is an effective divisor whose support is contained in 
the support of $f\inv(P_i)$.
It is enough to show that $\varGamma_i=0$
for each $i$.
Indeed, if $u$ is the class of a section $H$,
then $H$ must be a torsion section because $u\in U_f\oplus M(\Phi_f)$.
\par
Let $\Theta\spari_{1}, \dots, \Theta\spari_{\rho(i)}$ be the smooth rational curves 
in $f\inv (P_i)$ other than the simple component $\Theta\spari_{0}$ given in condition (ii).
Since $H$ is a section,
there exists an index ${\mu(i)}$ with $0\le \mu(i)\le \rho(i)$ such that
$$
\intf{H, \Theta\spari_{\nu}}=
\begin{cases}
1 & \textrm{if $\nu=\mu(i)$, } \\
0 &\textrm{otherwise.}
\end{cases}
$$
It suffices to show that $\mu(i)=0$.
Indeed, suppose that $\mu(i)=0$.
Then we have $\intf{u, \Theta\spari_{\nu}}=\intf{H, \Theta\spari_{\nu}}=0$
for all $\nu>0$,
and hence $\intf{\varGamma_i, \Theta\spari_{\nu}}=0$ for all $\nu>0$.
Since the root lattice $L(P_i)$ spanned by the classes of $\Theta\spari_{1}, \dots, \Theta\spari_{\rho(i)}$
is non-degenerate, and $[\varGamma_i]\in \Z \vf\oplus L(P_i)$,
we see that $\varGamma_i$ is a multiple of the divisor $f\inv (P_i)$.
We put $\varGamma_i=k_i\, f\inv (P_i)$ with $k_i\in \Z_{\ge 0}$.
Then we have $u=[H]+k\,\vf$, where $k=\sum_{i=1}^N k_i$.
From $\intf{u,u}=\intf{H,H}=-2$ and $\intf{H, \vf}=1$, we obtain $k=0$.
\par
Now we prove $\mu(i)=0$.
Let $\theta_1\dual, \dots, \theta_{\rho(i)}\dual$
be the basis of $L(P_i)\dual$ dual to the basis 
$[\Theta\spari_{1}], \dots, [\Theta\spari_{\rho(i)}]$ of $L(P_i)$.
Suppose that $\mu(i)>0$.
By $\intf{u, \Theta\spari_{0}}=1$ and $\intf{H, \Theta\spari_{0}}=0$,
we have
\begin{equation}\label{eq:thetazero}
\intf{\varGamma_i, \Theta\spari_{0}}=1.
\end{equation}
By $\intf{u, \Theta\spari_{\mu(i)}}=0$ and $\intf{H, \Theta\spari_{\mu(i)}}=1$,
we have
\begin{equation}\label{eq:thetamu}
\intf{\varGamma_i, \Theta\spari_{\mu(i)}}=-1.
\end{equation}
If $\nu\ne 0$ and $\nu\ne \mu(i)$, then we have 
 $\intf{u, \Theta\spari_{\nu}}=\intf{H, \Theta\spari_{\nu}}=0$, and hence 
we have
\begin{equation}\label{eq:thetanu}
\intf{\varGamma_i, \Theta\spari_{\nu}}=0.
\end{equation}
Let $z\in L(P_i)$ be the image of $[\varGamma_i]\in \Z \vf\oplus L(P_i)$ by the projection
$\Z \vf\oplus L(P_i)\to L(P_i)$.
Then~\eqref{eq:thetamu} and~\eqref{eq:thetanu} imply that
$z$ is equal to $-\theta_{\mu(i)}\dual$.
In particular, $\theta_{\mu(i)}\dual$ is in $L(P_i)$.
On the other hand,~\eqref{eq:thetazero} implies that the coefficient $m_{\mu(i)}$
of $[\Theta\spari_{0}]=\vf-\sum_{\nu}m_{\nu}[\Theta\spari_{\nu}]$ is $1$,
which contradicts Lemma~\ref{lem:theta}.
\end{proof}
Let $v_s\in U_f$ denote the class of the zero section $s$.
It is easy to make the complete list of vectors $u_L$ of $U_f\oplus L(\Phi_f)\dual$
that satisfies condition (ii) in Proposition~\ref{prop:torsionsection}.
If $u_L\in L(\Phi_f)\dual$ satisfies condition (ii) in Proposition~\ref{prop:torsionsection}
and belongs to $U_f\oplus M(\Phi_f)$,
then
$$
-\frac{\intf{u_L, u_L}}{2} \vf +v_s+u_L
$$
is the class of a torsion section.
The classes of all torsion sections are obtained in this way.
Thus we can calculate the set $\shortset{[\tau]}{\tau\in A_f}$,
and see how the torsion sections intersect irreducible components of reducible fibers.
\par
We say that a torsion section $\tau\in A_f$ is \emph{narrow at $P\in \P^1$}
if $\tau$ and $s$ intersect the same irreducible component of $f\inv (P)$.
%
%
\begin{example}\label{example:extremal64}
We consider the extremal elliptic $K3$ surfaces $(X, f, s)$  of
type 
$$
(\,A_9+A_5+A_3+A_1, \;\Z/2\Z\,),
$$
which have two algebraically distinguished connected components that cannot be distinguished by the transcendental lattices.
(See no.\,64 of Table~I.)  
Let $P(A_l)\in \P^1$ denote the point such that $f\inv (P_i)$ is of type $A_l$.
The non-trivial torsion section of an elliptic $K3$ surface in one connected components
is not narrow at $P(A_9), P(A_3), P(A_1)$,
and narrow at $P(A_5)$,
whereas the non-trivial torsion section of an elliptic $K3$ surface in the other connected components
is not narrow at $P(A_9), P(A_5)$,
and narrow at $P(A_3), P(A_1)$.
\end{example}
\begin{example}\label{example:onextremal93}
We consider the non-extremal elliptic $K3$ surfaces of
type 
$$
(\,A_5+A_3+6A_1, \;\Z/2\Z\,),
$$
which have three algebraically distinguished connected components.
(See no.\,91 of Table~II.)
These connected components can be distinguished by the 
narrowness of the non-trivial torsion section as follows:
\newcommand{\narrow}{\mathord{\rm narrow}}
\newcommand{\notnarrow}{\mathord{\hbox{\rm not narrow}}}
$$
\begin{array}{lll}
A_5 & A_3 & A_1, A_1, A_1, A_1, A_1, A_1\\
\hline
\narrow & \notnarrow & \textrm{$\notnarrow$ at all $6$ points}\\
\notnarrow & \narrow & \textrm{$\narrow$ at only one point}\\
\notnarrow & \notnarrow & \textrm{$\narrow$ at exacly $3$ points}.\\
\end{array}
$$
\end{example}
\subsection{Connected components when $G$ is trivial}\label{subsec:monodromy}
%
%
\begin{example}
Consider the combinatorial type
$$
(\Phi, A)=(7A_2, \Z/3\Z).
$$
We have $|\CCCC(\Phi, A, \Aut(\Phi))|=1$.
The discriminant form of $L(\Phi)$
is isomorphic to
$(\F_3, [4/3])^7$, 
and we have
$$
\Aut( \Phi)=(\Z/2\Z)^7 \semidirectproduct \SSSS_7.
$$
The set $\EEE(\Phi, A)/\mathord{\Aut}( \Phi)$ consists 
of only one element $[M]$,
where $M$ corresponds to the totally isotropic subspace 
of dimension $1$ over $\F_3$ generated by
$(0,1,1,1,1,1,1)$.
Hence we have
$$
\Stab (M)=(\Z/2\Z) \times ((\Z/2\Z)^6 \semidirectproduct \SSSS_6).
$$
The genus $\GGG$ determined by the signature $(2, 4)$
and the discriminant form $(D_M, -q_M)$
consists of only one isomorphism class.
We have 
$|\CCCC(\Phi, A, \{\id\})|=2$,
and the two connected components in $\CCCC(\Phi, A, \{\id\})$
are complex conjugate to each other.
\end{example}
%
%
\begin{example}
For the combinatorial type
$$
(\Phi, A)=(4A_4, \{0\}),
$$ 
we have $|\CCCC(\Phi, A, \Aut(\Phi))|=1$,
whereas $|\CCCC(\Phi, A, \{\id\})|=2$,
and the two connected components in $\CCCC(\Phi, A, \{\id\})$
are real.
\end{example}
%
%
\begin{example}
For the combinatorial type
$(\Phi, A)=(2D_4+4A_2, \{0\})$, 
we have $|\CCCC(\Phi, A, \Aut(\Phi))|=1$,
whereas $|\CCCC(\Phi, A, \{\id\})|=4$,
and the four connected components in $\CCCC(\Phi, A, \{\id\})$
are divided into  two complex conjugate pairs.
\end{example}
\section{Tables}\label{sec:tables}
See Introduction for the explanation of the entries of the tables below.
%
\subsection{{\normalsize Table I: Non-connected moduli of extremal elliptic $K3$ surfaces}}\label{subsec:TableI}
{\footnotesize
$$
\begin{array}{ccccc}
 {\rm no.} & \Phi & A & T & [r, c] \\ 
\hline
1 & E_{8}+A_{9}+A_{1} & [ 1 ] & [ 2, 0, 10 ] & [ 2, 0 ] \\ 
\hline
2 & E_{8}+A_{6}+A_{3}+A_{1} & [ 1 ] & [ 6, 2, 10 ] & [ 0, 2 ] \\ 
\hline
3 & E_{8}+2A_{5} & [ 1 ] & [ 6, 0, 6 ] & [ 0, 2 ] \\ 
\hline
4 & E_{7}+E_{6}+A_{5} & [ 1 ] & [ 6, 0, 6 ] & [ 0, 2 ] \\ 
\hline
5 & E_{7}+D_{5}+A_{6} & [ 1 ] & [ 6, 2, 10 ] & [ 0, 2 ] \\ 
\hline
6 & E_{7}+A_{11} & [ 1 ] & [ 4, 0, 6 ] & [ 0, 2 ] \\ 
\hline
7 & E_{7}+A_{10}+A_{1} & [ 1 ] & [ 2, 0, 22 ] & [ 1, 0 ] \\ 
 & & &[ 6, 2, 8 ] & [ 0, 2 ] \\ 
\hline
8 & E_{7}+A_{8}+A_{2}+A_{1} & [ 1 ] & [ 6, 0, 18 ] & [ 1, 2 ] \\ 
\hline
9 & E_{7}+A_{7}+A_{4} & [ 1 ] & [ 6, 2, 14 ] & [ 0, 2 ] \\ 
\hline
10 & E_{7}+A_{7}+A_{3}+A_{1} & [ 2 ] & [ 4, 0, 8 ] & [ 0, 2 ] \\ 
\hline
11 & E_{7}+A_{6}+A_{5} & [ 1 ] & [ 4, 2, 22 ] & [ 0, 2 ] \\ 
\hline
12 & E_{7}+A_{6}+A_{4}+A_{1} & [ 1 ] & [ 2, 0, 70 ] & [ 1, 0 ] \\ 
 & & &[ 8, 2, 18 ] & [ 0, 2 ] \\ 
\hline
13 & E_{7}+A_{5}+A_{4}+A_{2} & [ 1 ] & [ 6, 0, 30 ] & [ 2, 0 ] \\ 
\hline
14 & E_{6}+D_{5}+A_{7} & [ 1 ] & [ 8, 0, 12 ] & [ 0, 2 ] \\ 
\hline
15 & E_{6}+A_{12} & [ 1 ] & [ 4, 1, 10 ] & [ 0, 2 ] \\ 
\hline
16 & E_{6}+A_{11}+A_{1} & [ 1 ] & [ 6, 0, 12 ] & [ 0, 2 ] \\ 
\hline
17 & E_{6}+A_{9}+A_{2}+A_{1} & [ 1 ] & [ 12, 6, 18 ] & [ 0, 2 ] \\ 
\hline
18 & E_{6}+A_{8}+A_{4} & [ 1 ] & [ 12, 3, 12 ] & [ 1, 2 ] \\ 
\hline
19 & E_{6}+A_{8}+A_{3}+A_{1} & [ 1 ] & [ 12, 0, 18 ] & [ 1, 2 ] \\ 
\hline
20 & E_{6}+A_{7}+A_{5} & [ 1 ] & [ 6, 0, 24 ] & [ 0, 2 ] \\ 
\hline
21 & E_{6}+A_{6}+A_{5}+A_{1} & [ 1 ] & [ 6, 0, 42 ] & [ 0, 2 ] \\ 
\hline
22 & E_{6}+A_{6}+A_{3}+A_{2}+A_{1} & [ 1 ] & [ 6, 0, 84 ] & [ 1, 0 ] \\ 
 & & &[ 12, 0, 42 ] & [ 1, 0 ] \\ 
\hline
23 & E_{6}+A_{5}+A_{4}+A_{3} & [ 1 ] & [ 12, 0, 30 ] & [ 2, 0 ] \\ 
\hline
24 & D_{11}+A_{6}+A_{1} & [ 1 ] & [ 6, 2, 10 ] & [ 0, 2 ] \\ 
\hline
25 & D_{9}+D_{5}+A_{4} & [ 1 ] & [ 4, 0, 20 ] & [ 2, 0 ] \\ 
\hline
26 & D_{7}+A_{6}+A_{3}+A_{2} & [ 1 ] & [ 8, 4, 44 ] & [ 0, 2 ] \\ 
\hline
27 & D_{6}+A_{9}+A_{2}+A_{1} & [ 2 ] & [ 4, 2, 16 ] & [ 1, 0 ] \\ 
\cline{4-5}
 & & &[ 6, 0, 10 ] & [ 1, 0 ] \\ 
\hline
28 & D_{6}+A_{7}+A_{4}+A_{1} & [ 2 ] & [ 6, 2, 14 ] & [ 0, 2 ] \\ 
\hline
29 & D_{6}+2A_{6} & [ 1 ] & [ 14, 0, 14 ] & [ 0, 2 ] \\ 
\hline
30 & D_{5}+A_{13} & [ 1 ] & [ 6, 2, 10 ] & [ 0, 2 ] \\ 
\hline
31 & D_{5}+A_{12}+A_{1} & [ 1 ] & [ 2, 0, 52 ] & [ 1, 0 ] \\ 
 & & &[ 6, 2, 18 ] & [ 0, 2 ] \\ 
\hline
32 & D_{5}+A_{10}+A_{2}+A_{1} & [ 1 ] & [ 14, 4, 20 ] & [ 0, 2 ] \\ 
\hline
33 & D_{5}+A_{9}+A_{4} & [ 1 ] & [ 10, 0, 20 ] & [ 1, 2 ] \\ 
\hline
34 & D_{5}+A_{9}+A_{3}+A_{1} & [ 2 ] & [ 8, 4, 12 ] & [ 0, 2 ] \\ 
\hline
 &&&& \llap{(continues)}
\end{array}
$$
\vfill
\eject
$$
\begin{array}{ccccc} &&&& \llap{(continued)} \\
   {\rm no.} & \Phi & A & T & [r, c] \\
   \hline
   35 & D_{5}+A_{8}+A_{5} & [ 1 ] & [ 12, 0, 18 ] & [ 1, 2 ] \\ 
\hline 
36 & D_{5}+A_{8}+A_{4}+A_{1} & [ 1 ] & [ 2, 0, 180 ] & [ 1, 0 ] \\ 
 & & &[ 18, 0, 20 ] & [ 1, 0 ] \\ 
\hline
37 & D_{5}+2A_{6}+A_{1} & [ 1 ] & [ 14, 0, 28 ] & [ 0, 2 ] \\ 
\hline
38 & D_{5}+A_{6}+A_{5}+A_{2} & [ 1 ] & [ 6, 0, 84 ] & [ 1, 0 ] \\ 
 & & &[ 12, 0, 42 ] & [ 1, 0 ] \\ 
\hline
39 & A_{17}+A_{1} & [ 1 ] & [ 4, 2, 10 ] & [ 0, 2 ] \\ 
\hline
40 & A_{16}+2A_{1} & [ 1 ] & [ 2, 0, 34 ] & [ 1, 0 ] \\ 
 & & &[ 4, 2, 18 ] & [ 1, 0 ] \\ 
\hline
41 & A_{15}+A_{2}+A_{1} & [ 1 ] & [ 10, 2, 10 ] & [ 0, 2 ] \\ 
\hline
42 & A_{14}+A_{4} & [ 1 ] & [ 10, 5, 10 ] & [ 0, 2 ] \\ 
\hline
43 & A_{14}+A_{3}+A_{1} & [ 1 ] & [ 10, 0, 12 ] & [ 0, 2 ] \\ 
\hline
44 & A_{14}+A_{2}+2A_{1} & [ 1 ] & [ 12, 6, 18 ] & [ 0, 2 ] \\ 
\hline
45 & A_{13}+A_{5} & [ 1 ] & [ 4, 2, 22 ] & [ 0, 2 ] \\ 
\hline
46 & A_{13}+A_{4}+A_{1} & [ 1 ] & [ 2, 0, 70 ] & [ 1, 0 ] \\ 
 & & &[ 8, 2, 18 ] & [ 0, 2 ] \\ 
\hline
47 & A_{13}+A_{3}+2A_{1} & [ 2 ] & [ 6, 2, 10 ] & [ 0, 2 ] \\ 
\hline
48 & A_{12}+A_{5}+A_{1} & [ 1 ] & [ 10, 2, 16 ] & [ 0, 2 ] \\ 
\hline
49 & A_{12}+A_{4}+2A_{1} & [ 1 ] & [ 2, 0, 130 ] & [ 1, 0 ] \\ 
 & & &[ 18, 8, 18 ] & [ 1, 0 ] \\ 
\hline
50 & A_{11}+A_{6}+A_{1} & [ 1 ] & [ 4, 0, 42 ] & [ 0, 2 ] \\ 
\hline
51 & A_{11}+A_{4}+A_{2}+A_{1} & [ 1 ] & [ 12, 0, 30 ] & [ 0, 4 ] \\ 
\hline
52 & A_{11}+A_{3}+A_{2}+2A_{1} & [ 2 ] & [ 12, 0, 12 ] & [ 0, 2 ] \\ 
\hline
53 & A_{10}+A_{7}+A_{1} & [ 1 ] & [ 2, 0, 88 ] & [ 1, 0 ] \\ 
 & & &[ 10, 2, 18 ] & [ 0, 2 ] \\ 
\hline
54 & A_{10}+A_{6}+A_{2} & [ 1 ] & [ 4, 1, 58 ] & [ 0, 2 ] \\ 
 & & &[ 16, 5, 16 ] & [ 1, 0 ] \\ 
\hline
55 & A_{10}+A_{6}+2A_{1} & [ 1 ] & [ 12, 2, 26 ] & [ 0, 2 ] \\ 
\hline
56 & A_{10}+A_{5}+A_{3} & [ 1 ] & [ 4, 0, 66 ] & [ 1, 0 ] \\ 
 & & &[ 12, 0, 22 ] & [ 1, 0 ] \\ 
\hline
57 & A_{10}+A_{5}+A_{2}+A_{1} & [ 1 ] & [ 6, 0, 66 ] & [ 1, 0 ] \\ 
 & & &[ 18, 6, 24 ] & [ 0, 2 ] \\ 
\hline
58 & A_{10}+A_{4}+A_{3}+A_{1} & [ 1 ] & [ 12, 4, 38 ] & [ 0, 2 ] \\ 
 & & &[ 20, 0, 22 ] & [ 1, 0 ] \\ 
\hline
59 & A_{10}+A_{4}+2A_{2} & [ 1 ] & [ 6, 3, 84 ] & [ 1, 0 ] \\ 
 & & &[ 24, 9, 24 ] & [ 1, 0 ] \\ 
\hline
60 & 2A_{9} & [ 1 ] & [ 10, 0, 10 ] & [ 2, 0 ] \\ 
\hline
61 & A_{9}+A_{8}+A_{1} & [ 1 ] & [ 10, 0, 18 ] & [ 2, 0 ] \\ 
\hline
62 & A_{9}+A_{6}+A_{2}+A_{1} & [ 1 ] & [ 10, 0, 42 ] & [ 2, 0 ] \\ 
\hline
63 & A_{9}+A_{5}+A_{4} & [ 1 ] & [ 10, 0, 30 ] & [ 1, 2 ] \\ 
\hline
64 & A_{9}+A_{5}+A_{3}+A_{1} & [ 2 ] & [ 10, 0, 12 ] & [ 1, 0 ] \\ 
\cline{4-5}
 & & &[ 10, 0, 12 ] & [ 1, 0 ] \\ 
\hline
65 & A_{9}+2A_{4}+A_{1} & [ 5 ] & [ 2, 0, 10 ] & [ 2, 0 ] \\ 
\hline
66 & A_{9}+A_{4}+A_{3}+2A_{1} & [ 2 ] & [ 10, 0, 20 ] & [ 1, 2 ] \\ 
\hline
67 & 2A_{8}+2A_{1} & [ 1 ] & [ 18, 0, 18 ] & [ 1, 2 ] \\ 
\hline
68 & A_{8}+A_{7}+A_{2}+A_{1} & [ 1 ] & [ 18, 0, 24 ] & [ 1, 2 ] \\ 
\hline
69 & A_{8}+A_{6}+A_{3}+A_{1} & [ 1 ] & [ 10, 4, 52 ] & [ 0, 2 ] \\ 
\hline
70 & A_{8}+A_{6}+A_{2}+2A_{1} & [ 1 ] & [ 18, 0, 42 ] & [ 1, 2 ] \\ 
\hline
71 & A_{8}+A_{5}+A_{4}+A_{1} & [ 1 ] & [ 18, 0, 30 ] & [ 1, 2 ] \\ 
\hline
72 & A_{8}+A_{5}+2A_{2}+A_{1} & [ 3 ] & [ 6, 0, 18 ] & [ 1, 2 ] \\ 
\hline
73 & A_{8}+A_{4}+A_{3}+A_{2}+A_{1} & [ 1 ] & [ 6, 0, 180 ] & [ 1, 2 ] \\ 
\hline
74 & 2A_{7}+2A_{2} & [ 1 ] & [ 24, 0, 24 ] & [ 0, 2 ] \\ 
\hline
75 & A_{7}+A_{6}+A_{5} & [ 1 ] & [ 16, 4, 22 ] & [ 0, 2 ] \\ 
\hline 
 &&&& \llap{(continues)}
\end{array}
$$
\vfill
\eject
$$
\begin{array}{ccccc} &&&& \llap{(continued)} \\
   {\rm no.} & \Phi & A & T & [r, c] \\\hline
76 & A_{7}+A_{6}+A_{4}+A_{1} & [ 1 ] & [ 2, 0, 280 ] & [ 1, 0 ] \\ 
 & & &[ 18, 4, 32 ] & [ 0, 2 ] \\ 
\hline
77 & A_{7}+A_{6}+A_{3}+A_{2} & [ 1 ] & [ 4, 0, 168 ] & [ 0, 2 ] \\ 
\hline
78 & A_{7}+A_{6}+A_{3}+2A_{1} & [ 2 ] & [ 12, 4, 20 ] & [ 0, 2 ] \\ 
\hline
79 & A_{7}+2A_{5}+A_{1} & [ 2 ] & [ 6, 0, 24 ] & [ 0, 2 ] \\ 
\hline
80 & A_{7}+A_{5}+A_{4}+A_{2} & [ 1 ] & [ 6, 0, 120 ] & [ 1, 0 ] \\ 
 & & &[ 24, 0, 30 ] & [ 1, 0 ] \\ 
\hline
81 & A_{7}+A_{5}+A_{3}+A_{2}+A_{1} & [ 2 ] & [ 12, 0, 24 ] & [ 2, 0 ] \\ 
\hline
82 & A_{7}+A_{4}+A_{3}+2A_{2} & [ 1 ] & [ 12, 0, 120 ] & [ 2, 0 ] \\ 
\hline
83 & 2A_{6}+A_{4}+A_{2} & [ 1 ] & [ 28, 7, 28 ] & [ 2, 0 ] \\ 
\hline
84 & 2A_{6}+2A_{3} & [ 1 ] & [ 28, 0, 28 ] & [ 0, 2 ] \\ 
\hline
85 & 2A_{6}+2A_{2}+2A_{1} & [ 1 ] & [ 42, 0, 42 ] & [ 2, 0 ] \\ 
\hline
86 & A_{6}+A_{5}+A_{4}+A_{2}+A_{1} & [ 1 ] & [ 18, 6, 72 ] & [ 0, 2 ] \\ 
 & & &[ 30, 0, 42 ] & [ 1, 0 ] \\ 
\hline
87 & A_{6}+2A_{4}+A_{3}+A_{1} & [ 1 ] & [ 10, 0, 140 ] & [ 1, 0 ] \\ 
 & & &[ 20, 0, 70 ] & [ 1, 0 ] \\ 
\hline
88 & 2A_{5}+2A_{4} & [ 1 ] & [ 30, 0, 30 ] & [ 2, 0 ] \\ 
\hline
89 & 2A_{5}+4A_{2} & [ 3, 3 ] & [ 6, 0, 6 ] & [ 0, 2 ] \\ 
\hline
\end{array}
$$
}
\subsection{\normalsize Table II: Non-connected moduli of non-extremal elliptic $K3$ surfaces}\label{subsec:TableII}
{\footnotesize
$$
\begin{array}{ccccc}
 {\rm no.} & r &\Phi & A & [c_1, \dots, c_k] \\ 
\hline
1 & 17 & E_{7}+D_{6}+A_{3}+A_{1} & [ 2 ] & [ 1, 1 ] \\
 \hline
2 & 17 & E_{7}+2A_{5} & [ 1 ] & [ 2 ] \\
 \hline
3 & 17 & E_{7}+A_{5}+A_{3}+2A_{1} & [ 2 ] & [ 1, 1 ] \\
 \hline
4 & 17 & E_{6}+A_{11} & [ 1 ] & [ 2 ] \\
 \hline
5 & 17 & E_{6}+A_{6}+A_{5} & [ 1 ] & [ 2 ] \\
 \hline
6 & 17 & E_{6}+2A_{5}+A_{1} & [ 1 ] & [ 2 ] \\
 \hline
7 & 17 & D_{12}+A_{3}+2A_{1} & [ 2 ] & [ 1, 1 ] \\
 \hline
8 & 17 & D_{10}+D_{6}+A_{1} & [ 2 ] & [ 1, 1 ] \\
 \hline
9 & 17 & D_{8}+A_{7}+2A_{1} & [ 2 ] & [ 1, 1 ] \\
 \hline
10 & 17 & D_{8}+A_{5}+A_{3}+A_{1} & [ 2 ] & [ 1, 1 ] \\
 \hline
11 & 17 & 2D_{6}+A_{3}+2A_{1} & [ 2, 2 ] & [ 1, 1 ] \\
 \hline
12 & 17 & D_{6}+D_{5}+A_{5}+A_{1} & [ 2 ] & [ 1, 1 ] \\
 \hline
13 & 17 & D_{6}+A_{9}+2A_{1} & [ 2 ] & [ 1, 1 ] \\
 \hline
14 & 17 & D_{6}+A_{7}+A_{3}+A_{1} & [ 2 ] & [ 1, 1 ] \\
 \hline
15 & 17 & D_{6}+A_{7}+A_{2}+2A_{1} & [ 2 ] & [ 1, 1 ] \\
 \hline
16 & 17 & D_{6}+A_{5}+A_{3}+A_{2}+A_{1} & [ 2 ] & [ 1, 1 ] \\
 \hline
17 & 17 & D_{6}+A_{5}+A_{3}+3A_{1} & [ 2, 2 ] & [ 1, 1 ] \\
 \hline
18 & 17 & D_{5}+2A_{6} & [ 1 ] & [ 2 ] \\
 \hline
19 & 17 & D_{4}+2A_{6}+A_{1} & [ 1 ] & [ 2 ] \\
 \hline
20 & 17 & A_{11}+A_{5}+A_{1} & [ 1 ] & [ 2 ] \\
 \hline
21 & 17 & A_{9}+A_{5}+3A_{1} & [ 2 ] & [ 1, 1 ] \\
 \hline
22 & 17 & A_{9}+A_{3}+A_{2}+3A_{1} & [ 2 ] & [ 1, 1 ] \\
 \hline
23 & 17 & A_{7}+2A_{5} & [ 1 ] & [ 2 ] \\
 \hline
24 & 17 & A_{7}+A_{5}+A_{3}+2A_{1} & [ 2 ] & [ 1, 1 ] \\
 \hline
25 & 17 & 2A_{6}+A_{3}+2A_{1} & [ 1 ] & [ 2 ] \\
 \hline
26 & 17 & A_{6}+2A_{5}+A_{1} & [ 1 ] & [ 2 ] \\
 \hline
27 & 17 & 2A_{5}+2A_{3}+A_{1} & [ 2 ] & [ 1, 1 ] \\
 \hline
28 & 17 & 2A_{5}+A_{3}+A_{2}+2A_{1} & [ 2 ] & [ 1, 1 ] \\
 \hline
29 & 16 & E_{7}+D_{6}+3A_{1} & [ 2 ] & [ 1, 1 ] \\
 \hline
30 & 16 & E_{7}+2A_{3}+3A_{1} & [ 2 ] & [ 1, 1 ] \\
 \hline 
 &&&& \llap{(continues)}
\end{array}
$$
\vfill
\eject
$$
\begin{array}{ccccc} &&&& \llap{(continued)} \\
  {\rm no.} & r & \Phi & A & [c_1, \dots, c_k] \\\hline
31 & 16 & E_{6}+2A_{5} & [ 1 ] & [ 2 ] \\
 \hline
32 & 16 & D_{10}+A_{3}+3A_{1} & [ 2 ] & [ 1, 1 ] \\
 \hline
33 & 16 & D_{8}+D_{6}+2A_{1} & [ 2 ] & [ 1, 1 ] \\
 \hline
34 & 16 & D_{8}+A_{5}+3A_{1} & [ 2 ] & [ 1, 1 ] \\
 \hline
35 & 16 & D_{8}+2A_{3}+2A_{1} & [ 2 ] & [ 1, 1, 1 ] \\
 \hline
36 & 16 & 2D_{6}+A_{3}+A_{1} & [ 2 ] & [ 1, 1 ] \\
 \hline
37 & 16 & 2D_{6}+4A_{1} & [ 2, 2 ] & [ 1, 1 ] \\
 \hline
38 & 16 & D_{6}+D_{5}+A_{3}+2A_{1} & [ 2 ] & [ 1, 1 ] \\
 \hline
39 & 16 & D_{6}+D_{4}+A_{5}+A_{1} & [ 2 ] & [ 1, 1 ] \\
 \hline
40 & 16 & D_{6}+A_{9}+A_{1} & [ 2 ] & [ 1, 1 ] \\
 \hline
41 & 16 & D_{6}+A_{7}+3A_{1} & [ 2 ] & [ 1, 1 ] \\
 \hline
42 & 16 & D_{6}+A_{5}+A_{3}+2A_{1} & [ 2 ] & [ 1, 1, 1 ] \\
 \hline
43 & 16 & D_{6}+A_{5}+A_{2}+3A_{1} & [ 2 ] & [ 1, 1 ] \\
 \hline
44 & 16 & D_{6}+3A_{3}+A_{1} & [ 2 ] & [ 1, 1 ] \\
 \hline
45 & 16 & D_{6}+2A_{3}+A_{2}+2A_{1} & [ 2 ] & [ 1, 1 ] \\
 \hline
46 & 16 & D_{6}+2A_{3}+4A_{1} & [ 2, 2 ] & [ 1, 1 ] \\
 \hline
47 & 16 & D_{5}+A_{5}+A_{3}+3A_{1} & [ 2 ] & [ 1, 1 ] \\
 \hline
48 & 16 & D_{4}+A_{7}+A_{3}+2A_{1} & [ 2 ] & [ 1, 1 ] \\
 \hline
49 & 16 & A_{11}+A_{3}+2A_{1} & [ 2 ] & [ 1, 1 ] \\
 \hline
50 & 16 & A_{9}+A_{3}+4A_{1} & [ 2 ] & [ 1, 1 ] \\
 \hline
51 & 16 & A_{7}+A_{5}+4A_{1} & [ 2 ] & [ 1, 1 ] \\
 \hline
52 & 16 & A_{7}+A_{3}+A_{2}+4A_{1} & [ 2 ] & [ 1, 1 ] \\
 \hline
53 & 16 & 3A_{5}+A_{1} & [ 1 ] & [ 2 ] \\
 \hline
54 & 16 & 2A_{5}+A_{3}+3A_{1} & [ 2 ] & [ 1, 1, 1 ] \\
 \hline
55 & 16 & A_{5}+3A_{3}+2A_{1} & [ 2 ] & [ 1, 1 ] \\
 \hline
56 & 16 & A_{5}+2A_{3}+A_{2}+3A_{1} & [ 2 ] & [ 1, 1 ] \\
 \hline
57 & 16 & A_{5}+2A_{3}+5A_{1} & [ 2, 2 ] & [ 1, 1 ] \\
 \hline
58 & 15 & E_{7}+A_{3}+5A_{1} & [ 2 ] & [ 1, 1 ] \\
 \hline
59 & 15 & D_{8}+A_{3}+4A_{1} & [ 2 ] & [ 1, 1, 1 ] \\
 \hline
60 & 15 & 2D_{6}+3A_{1} & [ 2 ] & [ 1, 1 ] \\
 \hline
61 & 15 & D_{6}+D_{5}+4A_{1} & [ 2 ] & [ 1, 1 ] \\
 \hline
62 & 15 & D_{6}+D_{4}+A_{3}+2A_{1} & [ 2 ] & [ 1, 1 ] \\
 \hline
63 & 15 & D_{6}+A_{7}+2A_{1} & [ 2 ] & [ 1, 1 ] \\
 \hline
64 & 15 & D_{6}+A_{5}+A_{3}+A_{1} & [ 2 ] & [ 1, 1 ] \\
 \hline
65 & 15 & D_{6}+A_{5}+4A_{1} & [ 2 ] & [ 1, 1 ] \\
 \hline
66 & 15 & D_{6}+2A_{3}+3A_{1} & [ 2 ] & [ 1, 1, 1 ] \\
 \hline
67 & 15 & D_{6}+A_{3}+A_{2}+4A_{1} & [ 2 ] & [ 1, 1 ] \\
 \hline
68 & 15 & D_{6}+A_{3}+6A_{1} & [ 2, 2 ] & [ 1, 1 ] \\
 \hline
69 & 15 & D_{5}+2A_{3}+4A_{1} & [ 2 ] & [ 1, 1 ] \\
 \hline
70 & 15 & D_{4}+A_{5}+A_{3}+3A_{1} & [ 2 ] & [ 1, 1 ] \\
 \hline
71 & 15 & D_{4}+3A_{3}+2A_{1} & [ 2 ] & [ 1, 1 ] \\
 \hline
72 & 15 & A_{9}+A_{3}+3A_{1} & [ 2 ] & [ 1, 1 ] \\
 \hline
73 & 15 & A_{7}+2A_{3}+2A_{1} & [ 2 ] & [ 1, 1 ] \\
 \hline
74 & 15 & A_{7}+A_{3}+5A_{1} & [ 2 ] & [ 1, 1 ] \\
 \hline
75 & 15 & 2A_{5}+A_{3}+2A_{1} & [ 2 ] & [ 1, 1 ] \\
 \hline
76 & 15 & 2A_{5}+5A_{1} & [ 2 ] & [ 1, 1 ] \\
 \hline
77 & 15 & A_{5}+2A_{3}+4A_{1} & [ 2 ] & [ 1, 1, 1 ] \\
 \hline
78 & 15 & A_{5}+A_{3}+A_{2}+5A_{1} & [ 2 ] & [ 1, 1 ] \\
 \hline
79 & 15 & 3A_{3}+A_{2}+4A_{1} & [ 2 ] & [ 1, 1 ] \\
 \hline
80 & 15 & 3A_{3}+6A_{1} & [ 2, 2 ] & [ 1, 1 ] \\
 \hline 
 &&&& \llap{(continues)}
\end{array}
$$
\vfill
\eject
$$
\begin{array}{ccccc} &&&& \llap{(continued)} \\
  {\rm no.} & r & \Phi & A & [c_1, \dots, c_k] \\\hline
81 & 14 & D_{8}+6A_{1} & [ 2 ] & [ 1, 1 ] \\
 \hline
82 & 14 & D_{6}+D_{4}+4A_{1} & [ 2 ] & [ 1, 1 ] \\
 \hline
83 & 14 & D_{6}+A_{5}+3A_{1} & [ 2 ] & [ 1, 1 ] \\
 \hline
84 & 14 & D_{6}+2A_{3}+2A_{1} & [ 2 ] & [ 1, 1 ] \\
 \hline
85 & 14 & D_{6}+A_{3}+5A_{1} & [ 2 ] & [ 1, 1, 1 ] \\
 \hline
86 & 14 & D_{6}+A_{2}+6A_{1} & [ 2 ] & [ 1, 1 ] \\
 \hline
87 & 14 & D_{5}+A_{3}+6A_{1} & [ 2 ] & [ 1, 1 ] \\
 \hline
88 & 14 & D_{4}+2A_{3}+4A_{1} & [ 2 ] & [ 1, 1 ] \\
 \hline
89 & 14 & A_{7}+A_{3}+4A_{1} & [ 2 ] & [ 1, 1 ] \\
 \hline
90 & 14 & A_{5}+2A_{3}+3A_{1} & [ 2 ] & [ 1, 1 ] \\
 \hline
91 & 14 & A_{5}+A_{3}+6A_{1} & [ 2 ] & [ 1, 1, 1 ] \\
 \hline
92 & 14 & 4A_{3}+2A_{1} & [ 2 ] & [ 1, 1 ] \\
 \hline
93 & 14 & 3A_{3}+5A_{1} & [ 2 ] & [ 1, 1 ] \\
 \hline
94 & 14 & 2A_{3}+A_{2}+6A_{1} & [ 2 ] & [ 1, 1 ] \\
 \hline
95 & 14 & 2A_{3}+8A_{1} & [ 2, 2 ] & [ 1, 1 ] \\
 \hline
96 & 13 & D_{6}+A_{3}+4A_{1} & [ 2 ] & [ 1, 1 ] \\
 \hline
97 & 13 & D_{6}+7A_{1} & [ 2 ] & [ 1, 1 ] \\
 \hline
98 & 13 & D_{4}+A_{3}+6A_{1} & [ 2 ] & [ 1, 1 ] \\
 \hline
99 & 13 & A_{5}+A_{3}+5A_{1} & [ 2 ] & [ 1, 1 ] \\
 \hline
100 & 13 & A_{5}+8A_{1} & [ 2 ] & [ 1, 1 ] \\
 \hline
101 & 13 & 3A_{3}+4A_{1} & [ 2 ] & [ 1, 1 ] \\
 \hline
102 & 13 & 2A_{3}+7A_{1} & [ 2 ] & [ 1, 1 ] \\
 \hline
103 & 13 & A_{3}+A_{2}+8A_{1} & [ 2 ] & [ 1, 1 ] \\
 \hline
104 & 12 & D_{6}+6A_{1} & [ 2 ] & [ 1, 1 ] \\
 \hline
105 & 12 & 2A_{3}+6A_{1} & [ 2 ] & [ 1, 1 ] \\
 \hline
106 & 12 & A_{3}+9A_{1} & [ 2 ] & [ 1, 1 ] \\
 \hline
107 & 11 & A_{3}+8A_{1} & [ 2 ] & [ 1, 1 ] \\
 \hline
\end{array}
$$
}

\bibliographystyle{plain}

\begin{thebibliography}{10}

\bibitem{MR3447795}
Ay{\c{s}}eg{\"u}l Akyol and Alex Degtyarev.
\newblock Geography of irreducible plane sextics.
\newblock {\em Proc. Lond. Math. Soc. (3)}, 111(6):1307--1337, 2015.

\bibitem{MR2541164}
Ken-ichiro Arima and Ichiro Shimada.
\newblock Zariski-van {K}ampen method and transcendental lattices of certain
  singular {$K3$} surfaces.
\newblock {\em Tokyo J. Math.}, 32(1):201--227, 2009.

\bibitem{MR2030225}
Wolf~P. Barth, Klaus Hulek, Chris A.~M. Peters, and Antonius Van~de Ven.
\newblock {\em Compact complex surfaces}, volume~4 of {\em Ergebnisse der
  Mathematik und ihrer Grenzgebiete. 3. Folge. A Series of Modern Surveys in
  Mathematics}.
\newblock Springer-Verlag, Berlin, second edition, 2004.

\bibitem{MR522835}
J.~W.~S. Cassels.
\newblock {\em Rational quadratic forms}, volume~13 of {\em London Mathematical
  Society Monographs}.
\newblock Academic Press Inc. [Harcourt Brace Jovanovich Publishers], London,
  1978.

\bibitem{MR1228206}
Henri Cohen.
\newblock {\em A course in computational algebraic number theory}, volume 138
  of {\em Graduate Texts in Mathematics}.
\newblock Springer-Verlag, Berlin, 1993.

\bibitem{CSB}
J.~H. Conway and N.~J.~A. Sloane.
\newblock {\em Sphere packings, lattices and groups}, volume 290 of {\em
  Grundlehren der Mathematischen Wissenschaften}.
\newblock Springer-Verlag, New York, third edition, 1999.

\bibitem{MR2914799}
Alex Degtyarev.
\newblock Transcendental lattice of an extremal elliptic surface.
\newblock {\em J. Algebraic Geom.}, 21(3):413--444, 2012.

\bibitem{MR2977354}
Wolfgang Ebeling.
\newblock {\em Lattices and codes}.
\newblock Advanced Lectures in Mathematics. Springer Spektrum, Wiesbaden, third
  edition, 2013.

\bibitem{GAP}
The~GAP Group.
\newblock {G}{A}{P} - {G}roups, {A}lgorithms, and {P}rogramming.
\newblock Version 4.7.9; 2015 (http://www.gap-system.org).

\bibitem{2015arXiv150805251G}
{\c C}.~{G{\"u}ne{\c s} Akta{\c s}}.
\newblock {Classification of simple quartics up to equisingular deformation}.
\newblock {\em Hiroshima Math. J.}, 46(1):87--112, 2017.


\bibitem{MR0184257}
K.~Kodaira.
\newblock On compact analytic surfaces. {II}, {III}.
\newblock {\em Ann. of Math.},  (2) 77 (1963), 563--626; {\em ibid.}, 78:1--40, 1963.
\newblock Reprinted in {K}unihiko {K}odaira: Collected works. {V}ol. {III},
  Iwanami Shoten, Publishers, Tokyo; Princeton University Press, Princeton,
  N.J. 1975, pp. 1269--1372.

\bibitem{MR682664}
A.~K. Lenstra, H.~W. Lenstra, Jr., and L.~Lov{\'a}sz.
\newblock Factoring polynomials with rational coefficients.
\newblock {\em Math. Ann.}, 261(4):515--534, 1982.

\bibitem{MR842425}
Eduard Looijenga and Jonathan Wahl.
\newblock Quadratic functions and smoothing surface singularities.
\newblock {\em Topology}, 25(3):261--291, 1986.

\bibitem{MMB}
Rick Miranda and David~R. Morrison.
\newblock {\em Embeddings of integral quadratic forms}.
\newblock electronic, 2009, http://www.math.ucsb.edu/drm/manuscripts/eiqf.pdf.

\bibitem{MR834537}
Rick Miranda and David~R. Morrison.
\newblock The number of embeddings of integral quadratic forms. {I}.
\newblock {\em Proc. Japan Acad. Ser. A Math. Sci.}, 61(10):317--320, 1985.

\bibitem{MR0839800}
Rick Miranda and David~R. Morrison.
\newblock The number of embeddings of integral quadratic forms. {II}.
\newblock {\em Proc. Japan Acad. Ser. A Math. Sci.}, 62(1):29--32, 1986.

\bibitem{MR525944}
V.~V. Nikulin.
\newblock Integer symmetric bilinear forms and some of their geometric
  applications.
\newblock {\em Izv. Akad. Nauk SSSR Ser. Mat.}, 43(1):111--177, 238, 1979.
\newblock English translation: Math USSR-Izv. 14 (1979), no. 1, 103--167
  (1980).

\bibitem{MR0284440}
I.~I. Piatetski-Shapiro and I.~R. Shafarevich.
\newblock Torelli's theorem for algebraic surfaces of type {${\rm K}3$}.
\newblock {\em Izv. Akad. Nauk SSSR Ser. Mat.}, 35:530--572, 1971.
\newblock Reprinted in I. R. Shafarevich, Collected Mathematical Papers,
  Springer-Verlag, Berlin, 1989, pp.~516--557.

\bibitem{MR2346573}
Matthias Sch{\"u}tt.
\newblock Fields of definition of singular {$K3$} surfaces.
\newblock {\em Commun. Number Theory Phys.}, 1(2):307--321, 2007.

\bibitem{MR1813537}
Ichiro Shimada.
\newblock On elliptic {$K3$} surfaces.
\newblock {\em Michigan Math. J.}, 47(3):423--446, 2000.

\bibitem{MR2369942}
Ichiro Shimada.
\newblock On normal {$K3$} surfaces.
\newblock {\em Michigan Math. J.}, 55(2):395--416, 2007.

\bibitem{MR2604087}
Ichiro Shimada.
\newblock Non-homeomorphic conjugate complex varieties.
\newblock In {\em Singularities---{N}iigata--{T}oyama 2007}, volume~56 of {\em
  Adv. Stud. Pure Math.}, pages 285--301. Math. Soc. Japan, Tokyo, 2009.

\bibitem{MR2452829}
Ichiro Shimada.
\newblock Transcendental lattices and supersingular reduction lattices of a
  singular {$K3$} surface.
\newblock {\em Trans. Amer. Math. Soc.}, 361(2):909--949, 2009.

\bibitem{compdataConnEllK3}
Ichiro Shimada.
\newblock Connected components of the moduli of elliptic {$K3$} surfaces:
  computational data, 2016.
\newblock http://www.math.sci.hiroshima-u.ac.jp/$\sim$shimada/K3.html.

\bibitem{MR1820211}
Ichiro Shimada and De-Qi Zhang.
\newblock Classification of extremal elliptic {$K3$} surfaces and fundamental
  groups of open {$K3$} surfaces.
\newblock {\em Nagoya Math. J.}, 161:23--54, 2001.

\bibitem{MR0441982}
T.~Shioda and H.~Inose.
\newblock On singular {$K3$} surfaces.
\newblock In {\em Complex analysis and algebraic geometry}, pages 119--136.
  Iwanami Shoten, Tokyo, 1977.

\bibitem{MR1081832}
Tetsuji Shioda.
\newblock On the {M}ordell-{W}eil lattices.
\newblock {\em Comment. Math. Univ. St. Paul.}, 39(2):211--240, 1990.

\bibitem{MR1387816}
Jin-Gen Yang.
\newblock Sextic curves with simple singularities.
\newblock {\em Tohoku Math. J. (2)}, 48(2):203--227, 1996.

\bibitem{MR1468283}
Jin-Gen Yang.
\newblock Enumeration of combinations of rational double points on quartic
  surfaces.
\newblock In {\em Singularities and complex geometry (Beijing, 1994)}, volume~5
  of {\em AMS/IP Stud. Adv. Math.}, pages 275--312. Amer. Math. Soc.,
  Providence, RI, 1997.
  
\end{thebibliography}
\def\cftil#1{\ifmmode\setbox7\hbox{$\accent"5E#1$}\else
  \setbox7\hbox{\accent"5E#1}\penalty 10000\relax\fi\raise 1\ht7
  \hbox{\lower1.15ex\hbox to 1\wd7{\hss\accent"7E\hss}}\penalty 10000
  \hskip-1\wd7\penalty 10000\box7} \def\cprime{$'$} \def\cprime{$'$}
  \def\cprime{$'$} \def\cprime{$'$}

\end{document}